\documentclass[11pt]{article}

\usepackage[a4paper,margin=1in]{geometry}
\usepackage{amsmath,amssymb,amsfonts,amsthm}
\usepackage{graphicx}
\usepackage{subcaption}
\usepackage{algorithmic}
\usepackage{stmaryrd}
\usepackage{hyperref}

\let\OriginalIncludeGraphics\includegraphics
\renewcommand{\includegraphics}[2][]{%
  \IfFileExists{#2}{%
    \OriginalIncludeGraphics[#1]{#2}%
  }{%
    \fbox{%
      \parbox[c][0.28\textheight][c]{0.9\linewidth}{%
        \centering Missing figure:\\\texttt{\detokenize{#2}}%
      }%
    }%
  }%
}

\hypersetup{
  colorlinks=true,
  linkcolor=blue,
  citecolor=blue,
  urlcolor=blue
}

\newtheorem{theorem}{Theorem}[section]
\newtheorem{lemma}[theorem]{Lemma}
\newtheorem{corollary}[theorem]{Corollary}
\newtheorem{proposition}[theorem]{Proposition}

\theoremstyle{definition}
\newtheorem{definition}[theorem]{Definition}
\newtheorem{example}[theorem]{Example}

\theoremstyle{remark}
\newtheorem{remark}[theorem]{Remark}

\numberwithin{equation}{section}

\title{Symplectic Hamiltonian Direct Discontinuous Galerkin Method for Wave Propagation}

\author{
Haomiao Li\textsuperscript{1,2}
\and
Yumiao Li\textsuperscript{3}
\and
Jiaxin Wang\textsuperscript{2}
\and
Tiegang Liu\textsuperscript{2,4}
\and
Kun Wang\textsuperscript{2,4,*}
}

\date{}

\begin{document}

\maketitle

\begin{center}
\small
\textsuperscript{1}Sino-French Carbon Neutrality Research Center, Centrale Pekin,
Beihang University, Beijing 100191, China\\
\textsuperscript{2}LMIB and School of Mathematical Sciences, Beihang University,
Beijing 100191, China\\
\textsuperscript{3}School of Mathematics and Computational Sciences, Hunan First Normal University, Changsha 410006, Hunan, China\\
\textsuperscript{4}International Research Center for Mathematics and Interdisciplinary Sciences, Hangzhou International Innovation Institute of
Beihang University, Hangzhou 311115, China\\
\textsuperscript{*}Corresponding author. Email: \texttt{wangkun@buaa.edu}
\end{center}

\begin{abstract}
    This paper presents a symplectic Hamiltonian direct discontinuous Galerkin (DDG) method for approximating wave propagation problems, including the linear and semilinear wave equations. Within an auxiliary-variable-free DG framework, we prove that the symmetry of the numerical flux bilinear form is equivalent to the existence of a discrete Hamiltonian structure.  It follows that methods such as the symmetric interior penalty, BR2, and symmetric DDG (SDDG) methods admit a discrete Hamiltonian structure, whereas the Baumann--Oden and DDG methods do not possess this property. Exploiting this structure, we construct fully discrete symplectic schemes by combining the SDDG spatial discretization with symplectic time integrators. We also develop a unified treatment of periodic, Dirichlet, Neumann, and Robin boundary conditions that preserves the discrete Hamiltonian structure of the SDDG scheme. We further derive error estimates for the SDDG method applied to semilinear wave equations under these boundary conditions, showing the optimal convergence rate for the displacement and the suboptimal convergence rate for the velocity. Numerical experiments validate the theoretical convergence rates and demonstrate that the symplectic Hamiltonian DDG method achieves superior long-time energy conservation and accuracy. 
\end{abstract}


\section{Introduction}
\label{intro}
In this paper, we consider the following semi-linear wave equation
\begin{align}
\label{eq:governing_eq}
  &\partial_{tt} u(\boldsymbol{x},t)+g(u(\boldsymbol{x},t))=  \nabla\cdot( \kappa \nabla u(\boldsymbol{x},t)) +f(t), &\text{ in }  \Omega \times J.
\end{align}
where $J$ is a finite time interval and $\Omega$ is an open and bounded domain. It arises in a wide range of applications, including  plasmas, hydrodynamics, and magnetohydrodynamics \cite{cuevas2014sine}. The governing equation \eqref{eq:governing_eq} belongs to the class of Hamiltonian partial differential equations (PDEs). 
Hamiltonian systems preserve a symplectic structure on phase space. This property has motivated the development of symplectic integrators, which generate symplectic maps when applied to Hamiltonian ordinary differential equations (ODEs) \cite{feng2010symplectic}.
Owing to their ability to respect the underlying geometric structure, such methods
exhibit favourable long-time stability and near-energy conservation, and have
therefore become a standard tool for the numerical simulation of conservative
dynamical systems \cite{Tao2016}. Motivated by these successes, the concept of symplecticity has been extended from finite-dimensional ODEs to infinite-dimensional Hamiltonian PDEs. Two main approaches have been developed for this purpose, namely the multisymplectic method and the Hamiltonian method-of-lines method.

In this work, we focus on the Hamiltonian method-of-lines framework, in contrast to the multisymplectic formulation which treats space and time on an equal footing \cite{bridges2006numerical,mcdonald2016travelling,mclachlan2015multisymplectic} but can be computationally demanding. In the method-of-lines approach, the Hamiltonian PDE is first discretized in space to obtain a finite-dimensional Hamiltonian ODE system, which is then integrated in time using symplectic time integrators. This separation of space and time simplifies the numerical analysis and allows one to directly exploit mature high-order symplectic integrators with well-understood stability and accuracy properties. Within this framework, the main challenge lies in the construction of spatial discretizations that preserve the Hamiltonian structure while maintaining high-order accuracy. Various Hamiltonian-preserving spatial discretizations have been studied, including finite difference methods \cite{duncan1997sympletic} and finite element methods \cite{brink2017hamiltonian}. However, achieving both Hamiltonian structure preservation after spatial semi-discretization and high-order accuracy on complex geometries remains challenging. In this regard, discontinuous Galerkin (DG) methods provide a flexible and powerful framework due to their local formulation, geometric adaptability, and suitability for high-order approximations.

For the semi-linear wave equations, the DG discretization of the weak form of $\nabla\cdot(\kappa\nabla u)$ is essential. The primary difficulty lies in treating the discontinuity of $\kappa \nabla u$ across interfaces. This treatment determines whether the semi-discrete scheme is stable and consistent. In our problem, it is additionally required to admit a Hamiltonian structure. Existing DG discretizations of this operator can be classified into two main categories, depending on whether auxiliary variables are introduced.

The first category introduces auxiliary variables to rewrite the second-order operator as a first-order system, with LDG and HDG methods as representative examples. Xing and Chou \cite{xing2013energy} proposed an energy-conserving LDG method for wave equations combined with the leapfrog time integrator, making it an early instance of symplectic DG methods. Sánchez et al. \cite{sanchez2017symplectic} developed the first symplectic HDG method for linear wave equations, combining a Hamiltonian HDG spatial discretization with symplectic Runge–Kutta time integration, and proved discrete energy preservation and high-order accuracy. This framework was subsequently extended to electromagnetic wave equations \cite{sanchez2022symplectic} and to semi-linear wave equations \cite{sanchez2024symplectic}. However, the introduction of auxiliary variables inevitably increases the number of degrees of freedom, computational cost, and implementation complexity.

The second category, which we refer to as auxiliary-variable-free DG methods, directly constructs numerical fluxes using solution values and their derivatives on both sides of each interface, without introducing auxiliary variables. Grote et al. \cite{grote2006discontinuous} proposed an SIPDG method combined with Newmark time stepping for linear wave equations and established spatial optimal convergence rates both theoretically and numerically. He and Yang \cite{he2019symplectic} applied SIPDG with a third-order symplectic time integrator to seismic scalar wave equations. However, neither of these works analyses the discrete Hamiltonian structure or symplectic properties of the resulting scheme. 

Within this auxiliary-variable-free class, the DDG method \cite{liu2009direct} stands out for its compact formulation and flexible flux design. It has been successfully applied to diffusion problems and Navier-Stokes equations \cite{carrillo2025positivity,cheng2017parallel,cheng2016direct,liu2015optimal,yin2025pnp}, but has not yet been extended to Hamiltonian PDEs. The non-symmetry of the original DDG flux limits rigorous error analysis, motivating the development of SDDG methods \cite{vidden2013new,yue2017symmetric}. 

Despite these advances, the symplectic properties of auxiliary-variable-free DG methods for Hamiltonian PDEs remain unexplored, and in particular the relationship between numerical flux design and the existence of a discrete Hamiltonian structure has not been systematically investigated. 
These observations motivate the present work. The main contributions of this work are summarized as follows:
\begin{itemize}
\item Within an auxiliary-variable-free DG framework, we construct a numerical
flux bilinear form and a discrete energy, and prove that the symmetry of
the numerical flux bilinear form is equivalent to the existence of a discrete
Hamiltonian structure. Combined with symplectic time integrators, this leads
to a fully discrete symplectic scheme.

\item Rigorous error estimates are derived for the SDDG method applied to semi-
linear wave equations, including optimal convergence for the displacement
and suboptimal convergence for the velocity, together with admissible
parameter ranges.

\item Numerical experiments for linear and sine-Gordon wave equations are
performed to validate the proposed method, illustrating its long-time energy
behavior and accuracy, its ability to reproduce two-dimensional soliton
reflection and cloning, and the role of the discrete Hamiltonian structure in
controlling dispersion errors.
\end{itemize}
The rest of the paper is organized as follows. Section~\ref{sec:hamil_ddg_scheme} introduces the Hamiltonian DDG framework, including the DG formulation and the construction of the discrete Hamiltonian.
Section~\ref{sec:tempo-disc} describes the fully discrete schemes based on both symplectic and non-symplectic time integrators.
Section~\ref{sec:error-analysis} is devoted to the error analysis of the SDDG discretization for semi-linear wave equations.
Section~5 presents numerical experiments and concluding remarks are given in Section~\ref{sec:conclusion}.

\section{The Hamiltonian DDG method}\label{sec:hamil_ddg_scheme}
Throughout this work, we assume that the scalar coefficient $\kappa =\kappa(\boldsymbol{x}) \in L^\infty(\Omega)$  is uniformly positive, namely $\kappa_{\min} \leq \kappa \leq \kappa_{\max}$ almost everywhere in $\Omega$, with constants $0<\kappa_{\min}\leq \kappa_{\max}<\infty$. Moreover, the source term satisfies $f(t)\in L^2(\Omega)$ for all $t\geq 0$. $g(u)$ is a real-valued functional admitting a primitive \(G\), with notable example $g(u) = \sin(u)$
(sine-Gordon equation). In the special case $g(u)=0$, the equation reduces to the linear wave equation. We rewrite \eqref{eq:governing_eq} as the following system
\begin{equation}
\label{eq:system}
\begin{cases}
\partial_t u(\boldsymbol{x},t) = v(\boldsymbol{x},t),\\
\partial_t v(\boldsymbol{x},t) = \nabla \cdot (\kappa \nabla u)(\boldsymbol{x},t) + f(t) - g\big(u(\boldsymbol{x},t)\big),
\end{cases}\quad \text{for }(\boldsymbol{x},t)\in \Omega\times J, 
\end{equation}
The initial conditions are prescribed as
\begin{equation}\label{eq:initial_condition}
u(\boldsymbol{x},0)=u_0(\boldsymbol{x}),\quad v(\boldsymbol{x},0)=v_0(\boldsymbol{x}), \quad u_0,v_0\in L^2(\Omega).
\end{equation}
We consider periodic boundary conditions as well as the following general boundary condition:
\begin{equation}
  \label{eq:boundary_def}
  \alpha \kappa\nabla u(t) \cdot \boldsymbol{n}+ \beta u(t) = b(t),\text{ on } \Gamma :=\partial \Omega,
\end{equation}
where the time-independent coefficients $\alpha,\beta\in L^\infty(\Gamma)$ may vary along the boundary, and $b(t)\in H^{1/2}(\Gamma)$ for all $t\geq0$. We assume $|\alpha|$ is bounded away from zero on the Neumann and Robin parts, and $|\beta|$ is bounded away from zero on the Dirichlet part, where $b(t)/\beta\in H^{1/2}$.
It recovers Dirichlet boundary conditions ($\alpha = 0$, 
$\beta \neq 0$), Neumann boundary conditions ($\alpha \neq 0$, $\beta = 0$), 
and Robin boundary conditions ($\alpha, \beta \neq 0$).

\begin{proposition}[Continuous Hamiltonian formulations]
\label {prop:continuous_hamiltonian_form}
Let \(G\) be a primitive of \(g\), namely \(G'(u)=g(u)\). At each fixed
\(t\in J\), define
\begin{equation}
\label{eq:continuous_bulk_hamiltonian}
\mathcal H_0(u,v,t)
=
\int_\Omega        
\left(
\frac12 v^2
+
\frac12\kappa |\nabla u|^2
+
G(u)
-
fu
\right)\,\mathrm{d}\boldsymbol{x}.
\end{equation}
Under periodic boundary conditions, the Hamiltonian is \(\mathcal H_P=\mathcal H_0\).
Under the Dirichlet condition \(u=b/\beta\) on \(\Gamma\), the Hamiltonian is again
\begin{equation}
\label{eq:continuous_dirichlet_hamiltonian}
\mathcal H_D=\mathcal H_0
\quad
\text{on }
\{u:\; u|_\Gamma=b/\beta\}.
\end{equation}
For the Neumann/Robin boundary condition \(\alpha\neq0\), the Hamiltonian is
\begin{equation}
\label{eq:continuous_general_hamiltonian}
\mathcal H_B(u,v,t)
=
\mathcal H_0(u,v,t)
+
\int_\Gamma
\left(
\frac{\beta}{2\alpha}u^2
-
\frac{b}{\alpha}u
\right)\,\mathrm{d}s .
\end{equation}
In each case, the system \eqref{eq:system} admits the Hamiltonian form
\begin{equation}
\label{eq:continuous_hamiltonian_system}
\partial_t u=\frac{\delta\mathcal H}{\delta v},
\qquad
\partial_t v=-\frac{\delta\mathcal H}{\delta u}.
\end{equation}
\end{proposition}

\begin{proof}
For a functional \(\mathcal H\), \(D_u\mathcal H(u,v,t)[\eta]\) denotes the
Gâteaux derivative of \(\mathcal H\) with respect to \(u\) in the direction
\(\eta\):
\begin{equation}
D_u\mathcal H(u,v,t)[\eta]
:=
\left.
\frac{\mathrm{d}}{\mathrm{d}\varepsilon}
\mathcal H(u+\varepsilon\eta,v,t)
\right|_{\varepsilon=0}=\int_\Omega \frac{\delta \mathcal{H}}{\delta u} \eta \mathrm{d}\boldsymbol{x}.
\end{equation}
The variational derivative \(\delta\mathcal H/\delta u\) is defined by the
\(L^2(\Omega)\) representation of this directional derivative. It follows immediately that
\begin{equation}
\label{eq:variation_v}
D_v\mathcal H[\xi]
=
\int_\Omega v\xi\,\mathrm{d}\boldsymbol{x},
\end{equation}
and hence \(\delta\mathcal H/\delta v=v\). For the bulk part,
\begin{align}
D_u\mathcal H_0[\eta]
&=
\int_\Omega
\left(
\kappa\nabla u\cdot\nabla\eta
+
g(u)\eta
-
f\eta
\right)\,d\boldsymbol{x}
\notag\\
&=
\int_\Omega
\left(
-\nabla\cdot(\kappa\nabla u)
+
g(u)
-
f
\right)\eta\,\mathrm{d}\boldsymbol{x}
+
\int_\Gamma
\kappa\nabla u\cdot\boldsymbol n\,\eta\,\mathrm{d}s .
\label{eq:variation_bulk}
\end{align}
For periodic boundary conditions, the last term in \eqref{eq:variation_bulk}
cancels by periodicity. For Dirichlet condition, the
admissible variations satisfy \(\eta|_\Gamma=0\), so the boundary term also
vanishes.

It remains to consider \(\alpha\neq 0\). The boundary part in
\eqref{eq:continuous_general_hamiltonian} gives
\begin{equation}
\label{eq:variation_boundary}
D_u
\int_\Gamma
\left(
\frac{\beta}{2\alpha}u^2
-
\frac{b}{\alpha}u
\right)\,\mathrm{d}s[\eta]
=
\int_\Gamma
\left(
\frac{\beta}{\alpha}u
-
\frac{b}{\alpha}
\right)\eta\,\mathrm{d}s .
\end{equation}
Combining \eqref{eq:variation_bulk} and \eqref{eq:variation_boundary} yields
\begin{align}
D_u\mathcal H_B[\eta]
&=
\int_\Omega
\left(
-\nabla\cdot(\kappa\nabla u)
+
g(u)
-
f
\right)\eta\,d\boldsymbol{x}
\notag\\
&\quad
+
\int_\Gamma
\left(
\kappa\nabla u\cdot\boldsymbol n
+
\frac{\beta}{\alpha}u
-
\frac{b}{\alpha}
\right)\eta\,ds .
\label{eq:variation_general}
\end{align}
The boundary integral in \eqref{eq:variation_general} vanishes precisely
under the boundary condition \eqref{eq:boundary_def}. Therefore, in all
cases,
\begin{equation}
\label{eq:continuous_variational_derivatives}
\frac{\delta\mathcal H}{\delta u}
=
-\nabla\cdot(\kappa\nabla u)+g(u)-f,
\qquad
\frac{\delta\mathcal H}{\delta v}=v .
\end{equation}
Substituting \eqref{eq:continuous_variational_derivatives} into
\eqref{eq:continuous_hamiltonian_system} gives \eqref{eq:system}.
\end{proof}
\subsection{Mesh notation}
\label{sec:Notation}
We now introduce the notation of spatial discretization used throughout this
paper. Let $\Omega \subset \mathbb{R}^d$ ($d\geq 1$) be a bounded Lipschitz
domain. The spatial domain is partitioned into a conforming mesh
$\Omega_h = \bigcup_{K\in\Omega_h} K$, where each element $K$ is a polytope
(interval in $d=1$, quadrilateral or triangle in $d=2$). For each
element $K\in\Omega_h$, we denote its diameter by $h_K := \mathrm{diam}(K)$
and define the global mesh size $
   h := \max_{K\in\Omega_h} h_K$, which are used in the
approximation estimates.  

The set of element interfaces is denoted by
$\mathcal{F}_h = \mathcal{F}_h^{\mathcal{I}} \cup \mathcal{F}_h^{\mathcal{B}}$,
where $\mathcal{F}_h^{\mathcal{I}}$ collects the interior faces shared by
two neighbouring elements and $\mathcal{F}_h^{\mathcal{B}}$ collects the faces
lying on $\Gamma$. Under fully periodic boundary conditions, opposite
boundary faces are identified and treated as interior faces; in that case
$\mathcal{F}_h^{\mathcal{B}} = \emptyset$ after identification. 
The nonperiodic boundary is partitioned into Dirichlet, Neumann,
and Robin parts $\Gamma_X$, $X\in\{\mathrm D,\mathrm N,\mathrm R\}$,
whose boundary faces form the corresponding sets $\mathcal F_h^X$.
Thus, after periodic identification,
$\mathcal F_h^{\mathcal B}
=\mathcal F_h^{\mathrm D}\cup\mathcal F_h^{\mathrm N}
\cup\mathcal F_h^{\mathrm R}$.

We also define a face-normal scale
\(\tilde{h}\), which will appear in the face terms of DG methods.  Let
\(\boldsymbol{x}_K\) and \(\boldsymbol{x}_F\) be the barycentres of \(K\)
and \(F\), respectively.  For an interior face
\(F=\partial K\cap\partial K'\), with \(\boldsymbol n_{K,F}\) the outward
unit normal of \(K\) on \(F\), set
\[
   \tilde{h}|_F :=
   \left|(\boldsymbol{x}_{K'}-\boldsymbol{x}_{K})
          \cdot \boldsymbol n_{K,F}\right|.
\]
For a boundary face \(F\subset\partial K\cap\Gamma\), set
\[
   \tilde{h}|_F :=
   \left|(\boldsymbol{x}_{F}-\boldsymbol{x}_{K})
          \cdot \boldsymbol n_{K,F}\right|.
\]

We assume throughout that the mesh family is shape-regular and quasi-uniform,
and that \(\tilde{h}\) is uniformly comparable with \(h\): there exist
mesh-independent constants \(\gamma_{\mathrm{mesh}}\),
\(c_{\tilde{h}}\), and \(C_{\tilde{h}}\) such that
\[
   h \le \gamma_{\mathrm{mesh}} h_K,\qquad \forall K\in\Omega_h,
   \qquad
   c_{\tilde{h}} h
   \le \tilde{h}|_F
   \le C_{\tilde{h}} h,
   \qquad \forall F\in\mathcal F_h,
\]

We also assume $\kappa|_K\in W^{1,\infty}(K)$ on every mesh cell $K$, with bounds uniform in $h$.
The coefficient \(\tilde{\kappa}\) is defined, for \(\boldsymbol{x}\in F\), by
\[
   \tilde{\kappa}|_F(\boldsymbol{x}) :=
   \begin{cases}
      \max\{\kappa|_K(\boldsymbol{x}),\kappa|_{K'}(\boldsymbol{x})\}, & F\in\mathcal{F}_h^{\mathcal{I}},\quad F=\partial K\cap\partial K',\\
      \kappa|_K(\boldsymbol{x}), & F\in\mathcal{F}_h^{\mathcal{B}},\quad F=\partial K\cap\Gamma.
   \end{cases}
\]

For any element $K\in\Omega_h$, the local $L^2$ inner products are
\[
   (u,w)_K := \int_K u\,w \,\mathrm{d}\boldsymbol{x},
   \qquad
   \langle u,w\rangle_{\partial K}
       := \int_{\partial K} u\,w \,\mathrm{d}s.
\]
For any face $F\in\mathcal{F}_h$,
\[
   \langle u,w\rangle_F := \int_F u\,w \,\mathrm{d}s.
\]

Mesh-wide inner products are defined by summation over elements and faces:
\[
   (u,w)_{\Omega_h} := \sum_{K\in\Omega_h} (u,w)_K,
   \quad
   \langle u,w\rangle_{\partial\Omega_h}
       := \sum_{K\in\Omega_h} \langle u,w\rangle_{\partial K},
   \quad
   \langle \cdot,\cdot\rangle_{\mathcal{F}_h}
       := \sum_{F\in\mathcal{F}_h} \langle \cdot,\cdot\rangle_F.
\]
For $X\in\{\mathrm D,\mathrm N,\mathrm R\}$, we similarly write
\[
\langle u,w\rangle_{\mathcal F_h^X}
:=\sum_{F\in\mathcal F_h^X}\langle u,w\rangle_F
=\int_{\Gamma_X}uw\,\mathrm ds.
\]

For an approximation order $k\geq 1$, we introduce the local polynomial
space
\[
   S^k(K) :=
   \begin{cases}
      \mathbb{P}_k(K), & K \text{ is an interval or a triangle},\\
      \mathbb{Q}_k(K), & K \text{ is a quadrilateral},
   \end{cases}
\]
where $\mathbb{P}_k(K)$ denotes polynomials of total degree at most $k$ on
$K$, and $\mathbb{Q}_k(K)$ denotes polynomials of degree at most $k$ in
each coordinate direction on $K$. Our numerical methods seek approximate
solutions $u_h$ and $v_h$ in the discontinuous finite element space
\begin{equation}\label{eq:Wh}
   W_h := \bigl\{\, w \in L^2(\Omega) \;:\;
                  w|_K \in S^k(K),\ \forall K\in\Omega_h \,\bigr\}.
\end{equation}

Because functions in $W_h$ are generally discontinuous across interfaces,
we recall the standard DG average and jump operators. Let
$F\in\mathcal{F}_h^{\mathcal{I}}$ be an interior face shared by two
elements $K^+$ and $K^-$, and let $\boldsymbol{n}^{\pm}$ denote the unit
inward normal vectors of $K^{\pm}$ on $F$, so that
$\boldsymbol{n}^+ = -\boldsymbol{n}^-$. Denoting by $w^{\pm}$ the traces
of a scalar function $w$ from $K^{\pm}$, we define
\[
   \{w\} := \tfrac{1}{2}\bigl(w^+ + w^-\bigr),
   \qquad
   \llbracket w \rrbracket
       := w^+\boldsymbol{n}^+ + w^-\boldsymbol{n}^-.
\]
For a vector-valued function $\boldsymbol{q}$ with traces
$\boldsymbol{q}^{\pm}$, we set
\[
   \{\boldsymbol{q}\}
       := \tfrac{1}{2}\bigl(\boldsymbol{q}^+ + \boldsymbol{q}^-\bigr),
   \qquad
   \llbracket \boldsymbol{q} \rrbracket
       := \boldsymbol{q}^+\!\cdot\boldsymbol{n}^+
        + \boldsymbol{q}^-\!\cdot\boldsymbol{n}^-.
\]
For a matrix-valued function \(\boldsymbol{\tau}\) with traces
\(\boldsymbol{\tau}^{\pm}\), we set
\[
   \{\boldsymbol{\tau}\}
       := \tfrac{1}{2}\bigl(\boldsymbol{\tau}^+ + \boldsymbol{\tau}^-\bigr),
   \qquad
   \llbracket \boldsymbol{\tau} \rrbracket
       := \boldsymbol{\tau}^+\boldsymbol{n}^+ + \boldsymbol{\tau}^-\boldsymbol{n}^-.
\]
Hence $\llbracket w \rrbracket,\llbracket \boldsymbol{\tau} \rrbracket$  are vectors while
$\llbracket \boldsymbol{q} \rrbracket$ is a scalar. The standard DG identity then reads
\begin{equation}\label{eq:DGidentity}
   \langle \boldsymbol{q}\cdot\boldsymbol{n}, w\rangle_{\partial\Omega_h}
   =- \langle \{\boldsymbol{q}\}, \llbracket w \rrbracket\rangle_{\mathcal{F}_h^{\mathcal{I}}}
   - \langle \llbracket \boldsymbol{q} \rrbracket, \{w\}\rangle_{\mathcal{F}_h^{\mathcal{I}}}+\langle \boldsymbol{q}\cdot \boldsymbol{n},w\rangle_{\mathcal F_h^{\mathcal B}},
\end{equation}
valid for any sufficiently regular $w$ and $\boldsymbol{q}$ on
$\Omega_h$, where $\boldsymbol{n}$ denotes the outward normal on each
element boundary. Under periodic boundary conditions, we have
\begin{equation}\label{eq:DGidentity_periodic}
   \langle \boldsymbol{q}\cdot\boldsymbol{n}, w\rangle_{\partial\Omega_h}
   =- \langle \{\boldsymbol{q}\}, \llbracket w \rrbracket\rangle_{\mathcal{F}_h}
   - \langle \llbracket \boldsymbol{q} \rrbracket, \{w\}\rangle_{\mathcal{F}_h}.
\end{equation}
\subsection{DG formulation}
\label{sec:DG formulation}
We discretize system \eqref{eq:system} in space using an auxiliary-variable-free DG framework, where the second-order operator is handled in weak form via a numerical flux bilinear form. We seek $u_h, v_h \in W_h$ such that for all $w \in W_h$:
\begin{subequations}
  \label{eq:semi_discrete}
  \begin{align}
  \label{eq:semi_discrete1}
(\partial_t u_h, w)_{\Omega_h} ={}& (v_h(t), w)_{\Omega_h},  \\
\label{eq:semi_discrete2}
(\partial_t v_h, w)_{\Omega_h} ={}& \theta_h(u_h, w) 
- (\kappa \nabla u_h, \nabla w)_{\Omega_h} + (f, w)_{\Omega_h} - (g(u_h), w)_{\Omega_h},  
\end{align}
\end{subequations}
where $\theta_h(\cdot,\cdot)$ denotes a numerical flux bilinear form associated
with the DG discretization of the second-order spatial operator.
In this subsection, we impose periodic boundary conditions and recall the existing DG formulations obtained from different choices of \(\theta_h\). In formulas that contain a penalty scale, \(\tilde{h}\) denotes the uniformly comparable face scale defined in Section~\ref{sec:Notation}; the original literature often writes this scale as a local face or element diameter. The extension to the general boundary condition will be addressed subsequently through the SDDG method in subsection \ref{sec:sddg-bc}.
\paragraph{(i) Baumann--Oden method \cite{baumann1998discontinuous}}
\begin{equation*}
    \theta_h^{\text{BO}}(u_h,w)
= 
 -\langle \{\kappa \nabla u_h\},\llbracket w\rrbracket\rangle_{\mathcal{F}_h}
+\langle \{\kappa \nabla w\},\llbracket u_h\rrbracket\rangle_{\mathcal{F}_h}.
\end{equation*}

\paragraph{(ii) SIPDG method}
This method introduces a stabilisation term penalising the jump of the variable with a positive penalty parameter $\sigma>0$ \cite{grote2006discontinuous}.
\begin{align*}
    \theta_h^{\text{SIPDG}}(u_h,w)
=- \langle \{\kappa \nabla u_h\},\llbracket w\rrbracket\rangle_{\mathcal{F}_h}
-\langle \{\kappa \nabla w\},\llbracket u_h\rrbracket\rangle_{\mathcal{F}_h}
-\sigma\langle \frac{\tilde{\kappa}}{\tilde{h}}\llbracket u_h\rrbracket,\llbracket w\rrbracket\rangle_{\mathcal{F}_h},
\end{align*}
\paragraph{(iii) BR2 method}
\begin{align*}
    \theta_h^{\text{BR2}}(u_h,w)
= 
- \langle \{\kappa \nabla u_h\},\llbracket w\rrbracket\rangle_{\mathcal{F}_h}
-\langle \{\kappa \nabla w\},\llbracket u_h\rrbracket\rangle_{\mathcal{F}_h} - 
\langle R_F(\llbracket u_h\rrbracket),
      \llbracket w\rrbracket\rangle_{\mathcal{F}_h},
\end{align*}
where $R_F(\phi):=\{r_F(\phi)\}|_F$. Extending the local lifting in
\cite[Section~3.2.5]{bassi2005discontinuous} to variable $\kappa$, we define
$r_F(\phi)\in[W_h|_{\omega_F}]^d$, with $\omega_F$ the union of cells
adjacent to $F$, by
\[
(\kappa^{-1}r_F(\phi),\boldsymbol\tau)_{\omega_F}
=\langle\phi,\{\boldsymbol\tau\}\rangle_F,
\qquad\forall\boldsymbol\tau\in[W_h|_{\omega_F}]^d.
\]
With the inward-normal jump convention, this gives
\[
\langle R_F(\llbracket u_h\rrbracket),\llbracket w\rrbracket\rangle_F
=(\kappa^{-1}r_F(\llbracket u_h\rrbracket),r_F(\llbracket w\rrbracket))_{\omega_F},
\]
so $\theta_h^{\mathrm{BR2}}$ is symmetric.
\paragraph{(iv) DDG method}
The DDG method of Liu and Yan \cite{liu2009direct} constructs the numerical flux 
directly using jumps, averages and its high order derivates, 
\begin{equation}
\widehat{\kappa \nabla w}  = 
 \frac{\beta_0\tilde{\kappa}\llbracket w \rrbracket}{\tilde{h}} + \{ \kappa \nabla w \} + \beta_1 \tilde{\kappa} \tilde{h} \llbracket \nabla^2 w\rrbracket. 
\end{equation}
Here, $\beta_0$ and $\beta_1$ are tunable parameters and $\nabla^2$ is the hessian operator. The associated bilinear form is
\begin{equation}
    \begin{split}
      &\theta_h^{\text{DDG}}(u_h,w)
      :={} \langle \widehat{\kappa \nabla u_h}\cdot \boldsymbol{n}, w \rangle_{\partial \Omega_h}
      ={} -\langle \widehat{\kappa \nabla u_h}, \llbracket w \rrbracket\rangle_{\mathcal{F}_h}\\
      ={}& -\beta_0\langle \frac{\tilde{\kappa}}{\tilde{h}}{\llbracket u_h \rrbracket}, \llbracket w \rrbracket \rangle_{\mathcal{F}_h} -\langle \{\kappa \nabla {u_h}\}, \llbracket w \rrbracket \rangle_{\mathcal{F}_h} - \beta_1\langle \tilde{\kappa} \tilde{h}{\llbracket \nabla^2{u_h} \rrbracket}, \llbracket w \rrbracket \rangle_{\mathcal{F}_h}.
    \end{split}
\end{equation}
\paragraph{(v) SDDG method}
For this method, the corresponding bilinear form is modified as follows:
\begin{equation}
    \begin{split}
      &\theta_h^{\text{SDDG}}(u_h,w) := \langle \widehat{\kappa \nabla u_h} \cdot \boldsymbol{n}, w \rangle_{\partial \Omega_h}+\langle \widehat{\kappa \nabla w}\cdot\boldsymbol{n}, u_h \rangle_{\partial \Omega_h}\\
      ={}& -2\beta_0\langle \frac{\tilde{\kappa}}{\tilde{h}}{\llbracket u_h \rrbracket}, \llbracket w \rrbracket \rangle_{\mathcal{F}_h} -\langle \{ \kappa \nabla {u_h}\}, \llbracket w \rrbracket \rangle_{\mathcal{F}_h} -\langle \{\kappa \nabla {w} \}, \llbracket u_h \rrbracket \rangle_{\mathcal{F}_h} \\
      &- \beta_1\langle \tilde{\kappa} \tilde{h}{\llbracket \nabla^2  {u_h}\rrbracket}, \llbracket w \rrbracket \rangle_{\mathcal{F}_h}- \beta_1\langle \tilde{\kappa} \tilde{h}{\llbracket \nabla^2 {w} \rrbracket}, \llbracket u_h \rrbracket \rangle_{\mathcal{F}_h}. 
    \end{split}
\end{equation}
When $\beta_1=0$, it degenerates to the SIPDG method with $\sigma = 2\beta_0$.

Despite their different flux constructions, all schemes share the same algebraic semi-discrete form, which forms the basis for the structural analysis in section \ref{sec:discrete-hamiltonian}.

\subsection{SDDG boundary conditions}\label{sec:sddg-bc}
Periodic faces are identified in pairs and contribute only to the interior terms.
For fully periodic boundaries, $\mathcal F_h^{\mathcal B}=\emptyset$.
For the physical boundary parts defined in Section~\ref{sec:Notation}, the total flux contribution is
\begin{equation}\label{eq:sddg-decomp}
\begin{aligned}
\theta_h^{\mathrm{SDDG}}(u_h,w;b)
&=\theta_h^{\mathrm{SDDG}}\big|_{\mathcal I}
 +\theta_h^{\mathrm{SDDG}}\big|_\Gamma\\
&=\tilde{\theta}_h^{\mathrm{SDDG}}(u_h,w)+\ell_\Gamma(w;b).
\end{aligned}
\end{equation}
Here the vertical bar denotes the contribution from the indicated part. We collect the homogeneous terms
in the symmetric bilinear form $\tilde{\theta}_h^{\mathrm{SDDG}}$ and
the prescribed data in the boundary load functional $\ell_\Gamma$:
\[
\tilde{\theta}_h^{\mathrm{SDDG}}
=\theta_h^{\mathrm{SDDG}}\big|_{\mathcal I}
+\sum_{X\in\{\mathrm D,\mathrm N,\mathrm R\}}
\tilde{\theta}_h^{\mathrm{SDDG}}\big|_{\Gamma_X},
\qquad
\ell_\Gamma=\sum_{X\in\{\mathrm D,\mathrm N,\mathrm R\}}\ell_\Gamma^X.
\]
The modified semidiscrete system reads:
find $u_h, v_h \in W_h$ such that for all $w \in W_h$,
\begin{subequations}\label{eq:semidisc-general}
    \begin{align}
    \label{eq:semidisc-bc1}
    (\partial_t u_h, w)_{\Omega_h} =& (v_h(t), w)_{\Omega_h},  \\
   \label{eq:semidisc-bc2}
   (\partial_t v_h, w)_{\Omega_h} 
   =& \tilde{\theta}_h^{\mathrm{SDDG}}(u_h, w) 
     - (\kappa\nabla u_h, \nabla w)_{\Omega_h} - (g(u_h), w)_{\Omega_h} + (f, w)_{\Omega_h}\\
    &+ \ell_\Gamma(w; b),\notag
  \end{align}
\end{subequations}
The SDDG flux 
on $\mathcal{F}_h^{\mathcal{I}}$ retains its interior form
\begin{equation}\label{eq:sddg-interior}
\begin{aligned}
   \theta_h^{\mathrm{SDDG}}\big|_{\mathcal{I}} 
   &= -2\beta_0\langle \frac{\tilde{\kappa}}{\tilde{h}}\llbracket u_h \rrbracket, \llbracket w \rrbracket \rangle_{\mathcal{F}_h^{\mathcal{I}}} 
      - \langle \{\kappa \nabla u_h\}, \llbracket w \rrbracket \rangle_{\mathcal{F}_h^{\mathcal{I}}}
      - \langle \{\kappa \nabla w\}, \llbracket u_h \rrbracket \rangle_{\mathcal{F}_h^{\mathcal{I}}} \\
   &\quad - \beta_1\langle \tilde{\kappa} \tilde{h}\llbracket \nabla^2 u_h \rrbracket, \llbracket w \rrbracket \rangle_{\mathcal{F}_h^{\mathcal{I}}}
          - \beta_1\langle \tilde{\kappa} \tilde{h}\llbracket \nabla^2 w \rrbracket, \llbracket u_h \rrbracket \rangle_{\mathcal{F}_h^{\mathcal{I}}},
\end{aligned}
\end{equation}
which is manifestly symmetric. The treatment of boundary faces 
$F \in \mathcal{F}_h^{\mathcal{B}}$ depends on the type of boundary 
condition, as detailed in §\ref{sec:dirichlet-bc}--§\ref{sec:robin-bc} 
below. The resulting expressions of $\tilde{\theta}_h^{\mathrm{SDDG}}$ 
and $\ell_\Gamma$ are summarized in Table~\ref{tab:bc-summary}.

\subsubsection{Dirichlet boundary condition($\alpha = 0$, $\beta \neq 0$)}\label{sec:dirichlet-bc}
In this case, the 
prescribed boundary value $u = b/\beta$ on $\Gamma_{\mathrm D}$ is enforced weakly 
through the trace. On a boundary face $F \in \mathcal F_h^{\mathrm D}$, 
the superscript $-$ denotes the interior trace and $\boldsymbol n^-=-\boldsymbol n$.
The jumps and averages of $u_h$ and the test function $w$ are defined as
\begin{align}\label{eq:bc-dirichlet-trace}
\llbracket u_h \rrbracket
  &= (u_h^- - b/\beta)\,\boldsymbol{n}^-,
&\qquad
\{\nabla u_h\}
  &= \nabla u_h^-,
&\qquad
\llbracket \nabla^2 u_h \rrbracket
  &= \boldsymbol{0},
\\
\llbracket w \rrbracket
  &= w^-\,\boldsymbol{n}^-,
&\qquad
\{\nabla w\}
  &= \nabla w^-,
&\qquad
\llbracket \nabla^2 w \rrbracket
  &= \boldsymbol{0}.
\end{align}
Substituting into the SDDG flux yields the boundary contribution
\begin{equation}\label{eq:sddg-dirichlet-raw}
\begin{aligned}
   \theta_h^{\mathrm{SDDG}}\big|_{\Gamma_{\mathrm D}}
      ={}& -2\beta_0\langle \frac{\tilde{\kappa}}{\tilde{h}} u_h^-, w^- \rangle_{\mathcal F_h^{\mathrm D}} -\langle { \kappa \nabla {u_h^-}},  w^- \boldsymbol{n}^-\rangle_{\mathcal F_h^{\mathrm D}} -\langle {\kappa \nabla {w}^-},  u_h^- \boldsymbol{n}^-\rangle_{\mathcal F_h^{\mathrm D}} \\
      &+2\beta_0\langle \frac{\tilde{\kappa}}{\tilde{h}} b/\beta,w^- \rangle_{\mathcal F_h^{\mathrm D}}+\langle {\kappa\nabla {w}^-}\cdot \boldsymbol{n}^-, b/\beta \rangle_{\mathcal F_h^{\mathrm D}}.
\end{aligned}
\end{equation}
The first three terms form the symmetric Dirichlet contribution
\begin{equation}
    \begin{aligned}
    \tilde{\theta}_h^{\mathrm{SDDG}}\big|_{\Gamma_{\mathrm D}} = &-2\beta_0\langle \frac{\tilde{\kappa}}{\tilde{h}} u_h^-, w^- \rangle_{\mathcal F_h^{\mathrm D}} -\langle { \kappa\nabla {u_h^-}},  w^- \boldsymbol{n}^-\rangle_{\mathcal F_h^{\mathrm D}}\\
 &-\langle {\kappa\nabla {w}^-},  u_h^- \boldsymbol{n}^-\rangle_{\mathcal F_h^{\mathrm D}}.
\end{aligned}
\end{equation}
The last two terms depend 
linearly on the boundary data $b$:
\begin{equation}\label{eq:ell-dirichlet}
   \ell_\Gamma^{\mathrm{D}}(w; b)
   := 2\beta_0\langle \frac{\tilde{\kappa}}{\tilde{h}} w^- ,b/\beta\rangle_{\mathcal F_h^{\mathrm D}}+\langle {\kappa\nabla {w}^-}\cdot \boldsymbol{n}^-, b/\beta\rangle_{\mathcal F_h^{\mathrm D}}.
\end{equation}
\subsubsection{Neumann boundary condition($\alpha \neq 0$, $\beta = 0$)}\label{sec:neumann-bc}
In this case, the 
boundary data prescribes the normal flux 
$\kappa\nabla u \cdot \boldsymbol{n} = b/\alpha$ on $\Gamma_{\mathrm N}$. The boundary 
contribution to the SDDG flux then reduces to
\begin{equation}\label{eq:sddg-neumann-raw}
   \theta_h^{\mathrm{SDDG}}\big|_{\Gamma_{\mathrm N}} 
   = \langle b/\alpha, w \rangle_{\mathcal F_h^{\mathrm N}},
\end{equation}
which is independent of $u_h$. Consequently,
\begin{equation}\label{eq:ell-neumann}
   \tilde{\theta}_h^{\mathrm{SDDG}}\big|_{\Gamma_{\mathrm N}}=0,
   \qquad
   \ell_\Gamma^{\mathrm{N}}(w; b) := \langle b/\alpha, w \rangle_{\mathcal F_h^{\mathrm N}}.
\end{equation}
\subsubsection{Robin boundary condition($\alpha, \beta \neq 0$)}\label{sec:robin-bc}
In this case, the boundary 
data couples the normal flux and the solution trace through \eqref{eq:boundary_def}, which yields
\[
   \kappa\nabla u_h \cdot \boldsymbol{n}\big|_{\Gamma_{\mathrm R}} = \tfrac{1}{\alpha}(b - \beta u_h).
\]
The boundary contribution becomes
\begin{equation}\label{eq:sddg-robin-raw}
   \theta_h^{\mathrm{SDDG}}\big|_{\Gamma_{\mathrm R}} 
   = -\langle \tfrac{\beta}{\alpha}u_h, w \rangle_{\mathcal F_h^{\mathrm R}} 
     + \langle b/\alpha, w \rangle_{\mathcal F_h^{\mathrm R}}.
\end{equation}
The first term is symmetric in $(u_h, w)$ and is absorbed into 
$\tilde{\theta}_h^{\mathrm{SDDG}}|_{\Gamma_{\mathrm R}}$, while the second term, depending 
linearly on $b$, defines $\ell_\Gamma^{\mathrm{R}}$:
\begin{equation}\label{eq:ell-robin}
   \tilde{\theta}_h^{\mathrm{SDDG}}\big|_{\Gamma_{\mathrm R}}
   = - \langle \tfrac{\beta}{\alpha}u_h, w \rangle_{\mathcal F_h^{\mathrm R}},
   \qquad
   \ell_\Gamma^{\mathrm{R}}(w; b) := \langle b/\alpha, w \rangle_{\mathcal F_h^{\mathrm R}}.
\end{equation}
\begin{table}[htbp]
\centering
\caption{Symmetric boundary contributions and boundary load functionals for SDDG.}
\label{tab:bc-summary}
\renewcommand{\arraystretch}{1.4}
\setlength{\tabcolsep}{3pt}
\footnotesize
\resizebox{\textwidth}{!}{%
\begin{tabular}{c|p{0.56\textwidth}|p{0.26\textwidth}}
\hline
Boundary type & Symmetric boundary contribution & Boundary load functional \\
\hline
Periodic & $0$ & $0$ \\[0.3em]
Dirichlet & 
$\begin{aligned}
    &-2\beta_0\langle \frac{\tilde{\kappa}}{\tilde{h}} u_h^-, w^- \rangle_{\mathcal F_h^{\mathrm D}} \\
    &-\langle { \kappa\nabla {u_h^-}},  w^- \boldsymbol{n}^-\rangle_{\mathcal F_h^{\mathrm D}} \\
    &-\langle {\kappa\nabla {w}^-},  u_h^- \boldsymbol{n}^-\rangle_{\mathcal F_h^{\mathrm D}}
\end{aligned}$ 
& $ \begin{aligned}
    &2\beta_0\langle \frac{\tilde{\kappa}}{\tilde{h}} b/\beta, w^- \rangle_{\mathcal F_h^{\mathrm D}}\\
    &+\langle {\kappa\nabla {w}^-}\cdot \boldsymbol{n}^-, b/\beta \rangle_{\mathcal F_h^{\mathrm D}}
\end{aligned}$ \\[0.3em]
Neumann & $0$ & $\langle b/\alpha, w\rangle_{\mathcal F_h^{\mathrm N}}$ \\[0.3em]
Robin & $-\langle \frac{\beta}{\alpha}u_h, w\rangle_{\mathcal F_h^{\mathrm R}}$ & $\langle b/\alpha, w\rangle_{\mathcal F_h^{\mathrm R}}$ \\
\hline
\end{tabular}
}
\end{table}

\subsection{Discrete Hamiltonian formulation}\label{sec:discrete-hamiltonian}
In this section, we establish the central theoretical result of this 
paper.
\begin{definition}\label{def:basis}
Let $\{\phi_i\}_{i=1}^{N_{\mathrm{basis}}}$ be a basis of the 
discontinuous finite element space $W_h$ defined in \eqref{eq:Wh}, 
and let $M_{ij} := (\phi_i, \phi_j)_{\Omega_h}$ denote the associated 
mass matrix. Without loss of generality, we assume the basis is 
$L^2$-orthonormal, i.e., $M_{ij} = \delta_{ij}$. For a general basis, 
this can be achieved by a linear change of coordinates 
$\tilde{u} = L^{\top} u$, where $M = L L^{\top}$ is the Cholesky 
factorization of the mass matrix; all subsequent statements then 
apply in the $\tilde{u}$ coordinates and translate back through the 
inverse transformation.

In this orthonormal basis, the semidiscrete solutions $u_h$ and $v_h$ 
are expanded as
\[
   u_h(\boldsymbol{x}, t) 
   = \sum_{i=1}^{N_{\mathrm{basis}}} u_i(t)\,\phi_i(\boldsymbol{x}),
   \qquad
   v_h(\boldsymbol{x}, t) 
   = \sum_{i=1}^{N_{\mathrm{basis}}} v_i(t)\,\phi_i(\boldsymbol{x}),
\]
with degrees of freedom 
$\{u_i(t)\}, \{v_i(t)\} \in \mathbb{R}^{N_{\mathrm{basis}}}$.
\end{definition}

\begin{definition}\label{def:discrete-energy}
Let $(u_h, v_h)$ denote the semidiscrete solution of \eqref{eq:semidisc-general}. 
The \emph{discrete energy} associated with the semidiscrete system 
is defined as
\begin{equation}\label{eq:discrete-energy}
    \begin{aligned}
           \mathcal{E}_h(u_h, v_h, t) 
   := &\tfrac{1}{2}(v_h, v_h)_{\Omega_h} 
    + \tfrac{1}{2}(\kappa\nabla u_h, \nabla u_h)_{\Omega_h} - \tfrac{1}{2}\tilde{\theta}_h(u_h, u_h)\\
   & 
    + (G(u_h), 1)_{\Omega_h}
    - (f(t), u_h)_{\Omega_h} - \ell_\Gamma(u_h; b),
    \end{aligned}
\end{equation}
where $G$ is the primitive of the nonlinearity $g$, 
i.e., $G'(u) = g(u)$.
\end{definition}

\begin{theorem}\label{thm:hamiltonian}
Let $\tilde{\theta}_h: W_h \times W_h \to \mathbb{R}$ be the bilinear form 
associated with an auxiliary-variable-free DG discretization. Setting $\mathrm{q}_i := u_i$ and $\mathrm{p}_i := v_i$ 
for $i = 1, \ldots, N_{\mathrm{basis}}$, the following are equivalent:
\begin{enumerate}
\item[(i)] The bilinear form $\tilde{\theta}_h$ is symmetric:
\[
   \tilde{\theta}_h(u, w) = \tilde{\theta}_h(w, u) 
   \quad \forall\, u, w \in W_h.
\]
\item[(ii)] There exists a scalar function 
$\mathcal{H}_h(\boldsymbol{\mathrm{q}}, \boldsymbol{\mathrm{p}}, t)$ such that the 
semidiscrete system \eqref{eq:semidisc-general} can be written as the 
canonical Hamiltonian system
\begin{equation}\label{eq:hamiltonian-system}
   \frac{\mathrm{d}\mathrm{q}_i}{\mathrm{d}t} 
   = \frac{\partial \mathcal{H}_h}{\partial \mathrm{p}_i},
   \qquad
   \frac{\mathrm{d}\mathrm{p}_i}{\mathrm{d}t} 
   = -\frac{\partial \mathcal{H}_h}{\partial \mathrm{q}_i},
   \qquad i = 1, \ldots, N_{\mathrm{basis}}.
\end{equation}
\end{enumerate}
Moreover, in this case, the Hamiltonian function coincides with the 
discrete energy: 
$\mathcal{H}_h(\boldsymbol{\mathrm{q}}, \boldsymbol{\mathrm{p}}, t) = \mathcal{E}_h(u_h, v_h, t)$.
\end{theorem}

\begin{proof}
We prove (i) $\Rightarrow$ (ii) and (ii) $\Rightarrow$ (i) separately.

\medskip\noindent
\textbf{(i) $\Rightarrow$ (ii).}\quad
Define $\mathcal{H} := \mathcal{E}_h$ with $\mathcal{E}_h$ given by 
\eqref{eq:discrete-energy}. For the first equation, using $\mathrm{q}_i = u_i$, \eqref{eq:semidisc-bc1} tested with
$w = \phi_i$ together with the orthonormality of $\{\phi_i\}$ gives
\[
   \frac{\mathrm{d}{\mathrm{q}_i}}{\mathrm{d}t} 
   = \partial_t u_i 
   = (v_h, \phi_i)_{\Omega_h} 
   = \left(v_h, \frac{\partial v_h}{\partial v_i}\right)_{\!\Omega_h}
   = \frac{\partial}{\partial v_i}\!\left[\tfrac{1}{2}(v_h, v_h)_{\Omega_h}\right]
   = \frac{\partial \mathcal{E}_h}{\partial \mathrm{p}_i}
   = \frac{\partial \mathcal{H}_h}{\partial \mathrm{p}_i}.
\]

For the second equation, using $\mathrm{p}_i = v_i$ and testing \eqref{eq:semidisc-bc2} with $w = \phi_i$,
\[
   \frac{\mathrm{d}\mathrm{p}_i}{\mathrm{d}t} 
   = \partial_t v_i 
   = \tilde{\theta}_h(u_h, \phi_i) 
   - (\kappa\nabla u_h, \nabla\phi_i)_{\Omega_h}
   - (g(u_h), \phi_i)_{\Omega_h}
   + (f, \phi_i)_{\Omega_h} + \ell_\Gamma(\phi_i; b).
\]
Since $\partial u_h / \partial u_i = \phi_i$, the symmetry of 
$\tilde{\theta}_h$ implies
\[
   \tilde{\theta}_h(u_h, \phi_i) 
   = \frac{1}{2}\frac{\partial}{\partial u_i}\tilde{\theta}_h(u_h, u_h),
\]
and analogous identities hold for the remaining terms by a direct 
computation:
\[
   (\kappa\nabla u_h, \nabla\phi_i)_{\Omega_h} 
   = \frac{1}{2}\frac{\partial}{\partial u_i}(\kappa\nabla u_h, \nabla u_h)_{\Omega_h},
   \qquad
   (g(u_h), \phi_i)_{\Omega_h} 
   = \frac{\partial}{\partial u_i}(G(u_h), 1)_{\Omega_h}.
\]
Combining these identities,
\begin{align*}
       \frac{\mathrm{d}\mathrm{p}_i}{\mathrm{d}t} 
   = &\frac{\partial}{\partial u_i}\!\left[
       \frac{1}{2}\tilde{\theta}_h(u_h, u_h)
       - \tfrac{1}{2}(\kappa\nabla u_h, \nabla u_h)_{\Omega_h}
       - (G(u_h), 1)_{\Omega_h}
       + (f, u_h)_{\Omega_h} + \ell_\Gamma(u_h; b)
     \right]\\
   =& -\frac{\partial \mathcal{E}_h}{\partial u_i}
   = -\frac{\partial \mathcal{H}_h}{\partial \mathrm{q}_i}.
\end{align*}

\medskip\noindent
\textbf{(ii) $\Rightarrow$ (i).}\quad
Assume there exists $\mathcal{H}_h(\boldsymbol{\mathrm{q}}, \boldsymbol{\mathrm{p}}, t)$
such that \eqref{eq:hamiltonian-system} holds.
The first equation $\mathrm{d}\mathrm{q}_i/\mathrm{d}t = \partial \mathcal{H}_h/\partial \mathrm{p}_i$, 
combined with $\partial_t u_i = v_i = \mathrm{p}_i$ from \eqref{eq:semidisc-bc1}, 
yields $\partial \mathcal{H}_h/\partial \mathrm{p}_i = \mathrm{p}_i$ for all $i$. 
Integrating, 
\[
  \mathcal{H}_h(\boldsymbol{\mathrm{q}}, \boldsymbol{\mathrm{p}}, t)
   = \tfrac{1}{2}\boldsymbol{\mathrm{p}}^\top \boldsymbol{\mathrm{p}} + \mathcal{V}(\boldsymbol{\mathrm{q}}, t),
\]
where $\mathcal{V}(\boldsymbol{\mathrm{q}}, t)$ is a function independent of $\boldsymbol{\mathrm{p}}$.

The second equation 
$\mathrm{d}\mathrm{p}_i/\mathrm{d}t = -\partial \mathcal{H}_h/\partial \mathrm{q}_i 
= -\partial \mathcal{V}/\partial \mathrm{q}_i$, combined with \eqref{eq:semidisc-bc2}
tested with $\phi_i$, yields
\begin{equation}\label{eq:G-equation}
   -\frac{\partial \mathcal{V}}{\partial \mathrm{q}_i}
   = \tilde{\theta}_h(u_h, \phi_i) 
   - (\kappa\nabla u_h, \nabla\phi_i)_{\Omega_h}
   - (g(u_h), \phi_i)_{\Omega_h}
   + (f, \phi_i)_{\Omega_h} + \ell_\Gamma(\phi_i; b).
\end{equation}
Assuming $g \in C^1(\mathbb{R})$, the right-hand side of 
\eqref{eq:G-equation} is $C^1$ in $\boldsymbol{\mathrm{q}}$, so 
$\mathcal{V} \in C^2$ and Schwarz's theorem on the equality of mixed second 
partial derivatives yields
\begin{equation}\label{eq:schwarz}
   \frac{\partial^2 \mathcal{V}}{\partial \mathrm{q}_i \partial \mathrm{q}_j}
   = \frac{\partial^2 \mathcal{V}}{\partial \mathrm{q}_j \partial \mathrm{q}_i}
   \quad \forall\, i, j.
\end{equation}
Differentiating \eqref{eq:G-equation} with respect to $\mathrm{q}_j = u_j$
and using $\partial u_h / \partial u_j = \phi_j$, we obtain
\[
   -\frac{\partial^2 \mathcal{V}}{\partial \mathrm{q}_j \partial \mathrm{q}_i}
   = \tilde{\theta}_h(\phi_j, \phi_i)
   - (\kappa\nabla\phi_j, \nabla\phi_i)_{\Omega_h}
   - (g'(u_h)\phi_j, \phi_i)_{\Omega_h}.
\]
Here the terms $(f,\phi_i)_{\Omega_h}$ and $\ell_\Gamma(\phi_i; b)$ do not
contribute to the derivative with respect to $\mathrm{q}_j$, since the source
and boundary data are prescribed independently of $u_h$.
The weighted inner products $(\kappa\nabla\phi_j, \nabla\phi_i)_{\Omega_h}$ and 
$(g'(u_h)\phi_j, \phi_i)_{\Omega_h}$ are symmetric in $(i, j)$, 
hence \eqref{eq:schwarz} reduces to
\[
   \tilde{\theta}_h(\phi_j, \phi_i) = \tilde{\theta}_h(\phi_i, \phi_j)
   \quad \forall\, i, j.
\]
By bilinearity of $\tilde{\theta}_h$ and the fact that $\{\phi_i\}$ is a basis 
of $W_h$, this is equivalent to the symmetry of $\tilde{\theta}_h$ on 
$W_h \times W_h$, proving (i).
\end{proof}
\begin{corollary}\label{cor:energy-conservation}
Assume the bilinear form $\tilde{\theta}_h$ is symmetric, 
the discrete energy $\mathcal{E}_h$ defined in \eqref{eq:discrete-energy} is 
conserved along trajectories of the semidiscrete system 
\eqref{eq:semidisc-general}, whenever the source term $f$, the boundary data $b$,
and the nonlinearity term $g$ are independent of time.
\end{corollary}

\begin{proof}
By Theorem~\ref{thm:hamiltonian}, the semidiscrete system 
\eqref{eq:semidisc-general} is a canonical Hamiltonian system with 
$\mathcal{H}_h = \mathcal{E}_h$. The chain rule yields
\begin{equation}\label{eq:energy-conservation}
   \frac{\mathrm{d}\mathcal{E}_h}{\mathrm{d}t}
   = \sum_i \left(\frac{\partial \mathcal{H}_h}{\partial \mathrm{q}_i}\,\frac{\mathrm{d} \mathrm{q}_i}{\mathrm{d}t}
                + \frac{\partial \mathcal{H}_h}{\partial \mathrm{p}_i}\,\frac{\mathrm{d}\mathrm{p}_i}{\mathrm{d}t}\right)
   = \sum_i \left(\frac{\partial \mathcal{H}_h}{\partial \mathrm{q}_i}\,\frac{\partial \mathcal{H}_h}{\partial \mathrm{p}_i}
                - \frac{\partial \mathcal{H}_h}{\partial \mathrm{p}_i}\,\frac{\partial \mathcal{H}_h}{\partial \mathrm{q}_i}\right)
   = 0.
\end{equation}
\end{proof}

Conversely, when $\tilde{\theta}_h$ is non-symmetric, neither energy 
conservation nor a discrete Hamiltonian structure is available. We 
illustrate the energy non-conservation for the standard DDG flux 
$\theta_h^{\mathrm{DDG}}$; the analysis for other non-symmetric 
schemes is analogous. A direct computation using integration by 
parts and \eqref{eq:DGidentity} gives, for $f \equiv 0$ and 
$g \equiv 0$,
\[
   \frac{\mathrm{d}\mathcal{E}_h^{\mathrm{DDG}}}{\mathrm{d}t}
   = \tfrac{1}{2}\bigl(\langle \widehat{\kappa\nabla v_h}, \llbracket u_h\rrbracket\rangle_{\mathcal{F}_h}
   - \langle \widehat{\kappa\nabla u_h}, \llbracket v_h\rrbracket\rangle_{\mathcal{F}_h}\bigr)
   \;\not\equiv\; 0.
\]

Theorem~\ref{thm:hamiltonian} provides a general design criterion: 
any auxiliary-variable-free DG method for the wave equation admits 
a discrete Hamiltonian structure if and only if its numerical flux 
bilinear form is symmetric. This criterion is independent of the 
specific choice of penalty parameters, basis functions, mesh 
type or boundary condition and serves as a guiding principle for designing new 
symplectic-compatible DG schemes. Applying it to the five DG 
formulations introduced in §\ref{sec:DG formulation}, we obtain 
the classification summarized in Table~\ref{tab:hamiltonian-classification}.

\begin{table}[htbp]
\centering
\caption{Classification of DG methods for second-order spatial 
operators according to the symmetry of the flux bilinear form 
$\theta_h$ and the existence of a discrete Hamiltonian structure.}
\label{tab:hamiltonian-classification}
\renewcommand{\arraystretch}{1.3}
\begin{tabular}{lcc}
\hline
Method & $\theta_h$ symmetric? & Discrete Hamiltonian? \\
\hline
Baumann--Oden \cite{baumann1998discontinuous} & No  & No  \\
SIPDG \cite{grote2006discontinuous}           & Yes & Yes \\
BR2 \cite{bassi2005discontinuous}             & Yes & Yes \\
DDG \cite{liu2009direct}                      & No  & No  \\
SDDG \cite{yue2017symmetric}                  & Yes & Yes \\
\hline
\end{tabular}
\end{table}
\subsection{Spatial Assembly}
Based on the semi-discrete DG formulation introduced above, we proceed to the
spatial assembly. The initial condition functions $u_0$ and $v_0$ are projected onto $W_h$ using the $L^2$-orthogonal projection $\Pi_h$
so that $
u_h(\cdot,0) = \Pi_h u_0$ and $v_h(\cdot,0) = \Pi_h v_0$. 
Let $U(t), V(t) \in \mathbb{R}^{N\cdot\mathrm{DOF}}$ denote the global coefficient vectors of $u_h$ and $v_h$ with respect to the orthonormal basis. We define the weighted stiffness and flux matrices $K_\kappa, T_\kappa \in \mathbb{R}^{N\cdot\mathrm{DOF} \times N\cdot\mathrm{DOF}}$ and vectors $F(t), B(t), G(U) \in \mathbb{R}^{N\cdot\mathrm{DOF}}$ by

\begin{align*} 
&(K_\kappa)_{\beta ,\alpha}=(\kappa\nabla \phi_{\alpha},\nabla \phi_{\beta })_{\Omega_h},&\quad (T_\kappa)_{\beta ,\alpha}=
\tilde{\theta}_h\Bigl(\phi_{\alpha},\phi_{\beta }\Bigr),\\
&F_{\beta }(t)=(f(t),\phi_{\beta })_{\Omega_h}, &\quad B_{\beta}(t)=\ell_\Gamma(\phi_{\beta}; b(t)),\\
&{N_g(U)}_{\beta}
=\Bigl(g\bigl(\sum_{\alpha=1}^{N\cdot\mathrm{DOF}} u_\alpha\phi_\alpha\bigr),\phi_{\beta }\Bigr)_{\Omega_h},  &\quad{N_G(U)}
=\Bigl(G\bigl(\sum_{\alpha=1}^{N\cdot\mathrm{DOF}} u_\alpha\phi_\alpha\bigr),1 \Bigr)_{\Omega_h},
\end{align*}
where $\phi_\alpha$, $\alpha={(i-1)\mathrm{DOF}+k}$ represents the $k$-th basis function on the $i$-th cell and $\beta = {(j-1)\,\mathrm{DOF}+l }$. By the orthogonality of $\phi$, we have \(M=\mathbb{I}_{N\cdot \mathrm{DOF}}\) , where $\mathbb{I}_{N\cdot \mathrm{DOF}}$ denote the identity matrix of size \(N\cdot \mathrm{DOF}\). By introducing additionally the combined state variable $Z
=
\begin{bmatrix}
U,
V
\end{bmatrix}^\top
\in \mathbb{R}^{2N\cdot \mathrm{DOF}}$,
the semi-discrete system can be rewritten as:
\begin{equation}
\label{eq:semi_discrete_matrix}
\frac{\mathrm{d}Z}{\mathrm{d}t}
=
\begin{bmatrix}
\mathbb{O}_{N\cdot \mathrm{DOF}} & \mathbb{I}_{N\cdot \mathrm{DOF}}\\
-(K_\kappa-T_\kappa) & \mathbb{O}_{N\cdot \mathrm{DOF}}
\end{bmatrix}
Z
+
\begin{bmatrix}
\boldsymbol{0}_{N\cdot \mathrm{DOF}}\\
-  N_g(U) +  F(t) + B(t)
\end{bmatrix},
\end{equation}
where \(\mathbb{O}_{N\cdot \mathrm{DOF}}\) denote the zero matrix of size \(N\cdot \mathrm{DOF}\) and \(\boldsymbol{0}_{N\cdot\mathrm{DOF}} \in \mathbb{R}^{N\cdot\mathrm{DOF}}\) denotes the vector of zeros. The matrix form of discrete Hamiltonian is given by
\begin{equation}
\label{eq:discrete_hamiltonian}
H_h
=
\frac{1}{2} V^\top V
+
\left(
\frac{1}{2} U^\top (K_\kappa-T_\kappa)U
- (F(t)^\top+B(t)^\top)U + N_G(U)
\right).
\end{equation}
where $\mathcal{V}=\left(
\frac{1}{2} U^\top (K_\kappa-T_\kappa)U
- (F(t)^\top+B(t)^\top)U + N_G(U)
\right)$ is the potential energy, $\mathcal{T}=
\frac{1}{2} V^\top V$ is the kinetic energy.
\section{Fully symplectic discrete framework}\label{sec:tempo-disc}
\subsection{Explicit Symplectic Partitioned Runge-Kutta (ESPRK) methods}
We first introduce the ESPRK methods, which are well-suited for separable Hamiltonian systems. Let \((U^n, V^n)\) denote the numerical solution at time level \(t_n\), and let \((U^{n,i}, V^{n,i})\) denote the stage values of the time integrator for \(i = 1, \dots, s\). We initialize the stage iterations by setting \(U^{n,0} = U^n\) and \(V^{n,0} = V^n\). The fully discrete scheme then proceeds as follows:
\begin{align*}
    &V^{n,i} = V^{n,i-1} - \Delta t b_i
    \frac{\partial \mathcal{V}}{\partial U}
    \left(U^{n,i-1},t_n+\Delta t\sum_{j=1}^{i-1}\tilde b_j\right),\\
    &U^{n,i} = U^{n,i-1} + \Delta t\tilde b_i
    \frac{\partial \mathcal{T}}{\partial V}(V^{n,i}),
    \qquad i=1,\ldots,s,\\
    &U^{n+1} = U^{n,s}, \quad V^{n+1} = V^{n,{s}}.
\end{align*}
The coefficients $b_i$ and $\tilde{b}_i$ are listed in Appendix~\ref{app:symplectic-coefficients} \cite{sanchez2017symplectic}.
\subsection{Symplectic Diagonally Implicit Runge-Kutta (SDIRK) methods}
\label{subsec:sdirk}
We also
consider SDIRK schemes for the long-time integration of the semidiscrete Hamiltonian system.
Applying the $s$-stage SDIRK method with $s$ stages to \eqref{eq:semi_discrete_matrix},
the stage variables $Z^{n,i}$ satisfy 
\begin{align}
    \label{eq:sdirk-stage}
        &\mathcal A_i Z^{n,i}
+ \Delta t\,\frac{b_i}{2}\,\mathcal G(Z^{n,i})
= \mathcal L_i(t_n+(b_i/2+\sum_{j=1}^{i-1}b_j)\Delta t),
\quad i=1,\dots,s,\\
&K_i
=\frac{2}{\Delta t\, b_i}\bigl(Z^{n,i}-Z^n\bigr)
-2\sum_{j=1}^{i-1}\frac{b_j}{b_i}K_j,\quad i=1,\dots,s,
\end{align}
where the non-linear term and the block matrix $\mathcal A_i$ are defined as
\[
\mathcal G(Z)=
\begin{pmatrix}
\boldsymbol{0}_{N \cdot {DOF}} \\
N_g(U)
\end{pmatrix},\quad \mathcal A_i =
\begin{pmatrix}
\mathbb{I}_{N \cdot {DOF}} &
-\Delta t\,\dfrac{b_i}{2}\,\mathbb{I}_{N \cdot {DOF}} \\[1ex]
\Delta t\,\dfrac{b_i}{2}\,(K_\kappa-T_\kappa) &
\mathbb{I}_{N \cdot {DOF}}
\end{pmatrix}.
\]
the right-hand side $\mathcal L_i$ reads
\begin{equation*}
\mathcal L_i(t)
=
Z^n
+
\Delta t
\sum_{j=1}^{i-1} b_j\,K_j
+
\Delta t\,\frac{b_i}{2}
\begin{pmatrix}
\boldsymbol{0}_{N \cdot {DOF}} \\
F(t)+B(t)
\end{pmatrix}.
\end{equation*}
We solve the resulting non-linear system \eqref{eq:sdirk-stage} using a fixed point iteration based on the inversion of the matrix $\mathcal A_i$.  Finally, the solution is updated as
\begin{equation}\label{eq:update}
Z^{n+1}
=
Z^n
+
\Delta t
\sum_{i=1}^s b_i K_i.
\end{equation}
The coefficients $b_i$ are listed in Table~\ref{tab:sdirk-coefficients} \cite{sanchez2017symplectic}.
\begin{definition}[Symplectic Hamiltonian DDG method]
Let the spatial semi-discretization be given by the SDDG method, whose 
numerical flux induces a symmetric bilinear form $\tilde{\theta}_h^{\mathrm{SDDG}}$ 
and therefore admits a discrete Hamiltonian structure by Theorem~\ref{thm:hamiltonian}. 
When the resulting semi-discrete Hamiltonian system \eqref{eq:semi_discrete_matrix} 
is integrated in time using a symplectic time integrator (ESPRK or SDIRK), 
the fully discrete scheme preserves the symplectic structure. We refer to this 
combination as the {symplectic Hamiltonian DDG method}.
\end{definition}
\section{Error analysis}
\label{sec:error-analysis}
In this section, we analyse the spatial semidiscrete error of the SDDG
method for the semilinear wave equation. Let $(u,v)$ denote the exact
solution and let $(u_h,v_h)$ denote the solution of
\eqref{eq:semidisc-general}.  Define
\begin{equation}
\label{eq:errs}
e_u := u_h - u, \qquad e_v := v_h - v,
\end{equation}
which represent the errors in the displacement and velocity components,
respectively.  We use the decompositions
\begin{subequations}
\label{eq:error-decomposition}
\begin{align}
e_u&=\xi_u+\eta_u,
&\xi_u&:=u_h-\Pi_hu,
&\eta_u&:=\Pi_hu-u,
\label{eq:displacement-error-decomposition}\\
e_v&=\xi_v+\eta_v,
&\xi_v&:=v_h-\Pi_hv,
&\eta_v&:=\Pi_hv-v.
\label{eq:velocity-error-decomposition}
\end{align}
\end{subequations}
Here $\xi_u,\xi_v$ are the discrete errors and $\eta_u,\eta_v$ are the
projection errors. 
\subsection{Preliminary}
We begin with the regularity assumptions on the exact solution.  Assume
that the solution of \eqref{eq:system} satisfies
\begin{equation}
\label{eq:solution-regularity}
u, v \in L^\infty(J;H^{1+\sigma}(\Omega)), \quad
\partial_t v \in L^1(J;H^{\sigma}(\Omega)),
\quad \sigma>1+\frac{d}{2}.
\end{equation}
The Sobolev embedding theorem gives
\[
H^{1+\sigma}(\Omega)\hookrightarrow C^2(\overline\Omega).
\]
On periodic boundary pairs, the regularity in \eqref{eq:solution-regularity}
is understood for the periodic extensions of $u$ and $v$.
Thus their Hessian traces are single-valued on all interior faces,
including periodically identified faces. We also assume that $g$ and its first derivative $g'$ are Lipschitz
continuous:
\begin{align}
    \label{eq:g_lipschitz}
    |g(w_1(\boldsymbol{x}))-g(w_2(\boldsymbol{x}))|&\leq L_g |w_1(\boldsymbol{x})-w_2(\boldsymbol{x})|, \text{a.e } \boldsymbol{x}\in \Omega,\forall w_1,w_2 \in L^2(\Omega),\\
    \label{eq:g'_lipschitz}
    |g'(w_1(\boldsymbol{x}))-g'(w_2(\boldsymbol{x}))|&\leq L_{g'} |w_1(\boldsymbol{x})-w_2(\boldsymbol{x})|, \text{a.e } \boldsymbol{x}\in \Omega,\forall w_1,w_2 \in L^2(\Omega).
\end{align}

For the error analysis, we use the boundary partition and face sets defined
in Section~\ref{sec:Notation}, with the boundary contributions of
Section~\ref{sec:sddg-bc}. The label $B$ denotes this fixed choice of boundary
conditions, including any periodic identifications. We set
\begin{equation}
\label{eq:robin-coefficient}
   \gamma_{\mathrm R}:=\frac{\beta}{\alpha}\quad\hbox{on }\Gamma_{\mathrm R},
\end{equation}
and assume that
$\gamma_{\mathrm R}\in L^\infty(\Gamma_{\mathrm R})$.  Its positive and
negative parts are denoted by
\[
\gamma_{\mathrm R}^{+}:=\max\{\gamma_{\mathrm R},0\},
\qquad
\gamma_{\mathrm R}^{-}:=\max\{-\gamma_{\mathrm R},0\},
\qquad
\gamma_{\mathrm R}=\gamma_{\mathrm R}^{+}-\gamma_{\mathrm R}^{-}.
\]

To state a priori error bounds, we use the extended energy space
\begin{equation}
\label{eq:extended-space}
W(h):=H^1(\Omega)+W_h
=\{w=w_c+w_d:\ w_c\in H^1(\Omega),\ w_d\in W_h\}.
\end{equation}
Throughout this section, Dirichlet jumps use the homogeneous convention
$\llbracket w\rrbracket=w\boldsymbol n^-$ for $w\in W(h)$.
We define the boundary-dependent DG energy norm
\begin{equation}
  \label{eq:boundary-energy-norm}
\begin{aligned}
  \|w\|_{h,B}^2
  :={}&(\kappa\nabla w,\nabla w)_{\Omega_h}
  +\lambda^2\left\langle
\frac{\tilde{\kappa}}{\tilde{h}}\llbracket w\rrbracket,
\llbracket w\rrbracket
\right\rangle_{\mathcal F_h^{\mathcal I}}\\
  &+\lambda^2\left\langle\frac{\tilde{\kappa}}{\tilde{h}}w,w\right\rangle_{\mathcal F_h^{\mathrm D}}
  +\langle|\gamma_{\mathrm R}|w,w\rangle_{\mathcal F_h^{\mathrm R}},
\end{aligned}
\end{equation}
where $\lambda>1$, and $\tilde{h}$ and $\tilde{\kappa}$ are the face-wise
functions defined in Section~\ref{sec:Notation}.   For later use, we recall the semi-norm $|\cdot|$ and Sobolev norm $\|\cdot\|$ on the piecewise polynomial space,
\begin{align*}
|w|_{m,K}^2 &= \sum_{|\alpha|=m} (\partial_\alpha w,\partial_\alpha w)_{K}, 
&\|w\|_{m,K}^2 &= \sum_{|\alpha|\le m} (\partial_\alpha w,\partial_\alpha w)_{K}, \\
|w|_{m,\Omega_h}^2 &= \sum_{|\alpha|=m} (\partial_\alpha w,\partial_\alpha w)_{\Omega_h}, 
&\|w\|_{m,\Omega_h}^2 &= \sum_{|\alpha|\le m} (\partial_\alpha w,\partial_\alpha w)_{\Omega_h}.
\end{align*}
we also define the norm on the cell interfaces by
\begin{align*}
    \|w\|_{0,\mathcal{F}_h}^2&=\langle w,w\rangle_{\mathcal{F}_h}.
\end{align*}

We next introduce the projection-based bilinear form containing the homogeneous
symmetric boundary contributions.  For any \((w,\psi)\in W(h)\times W(h)\),
we define
\begin{align}
  a_h^B(w,\psi)
  &:=
    a_h^{\mathcal I}(w,\psi)
    +a_h^{\mathrm D}(w,\psi)
    +a_h^{\mathrm R}(w,\psi),
  \label{eq:ah-boundary-definition}
\end{align}
where
\begin{align}
 a_h^{\mathcal I}(w,\psi)
 :={}&(\kappa\nabla w,\nabla\psi)_{\Omega_h}
 +2\beta_0\left\langle
      \frac{\tilde{\kappa}}{\tilde{h}}\llbracket w\rrbracket,
      \llbracket\psi\rrbracket
   \right\rangle_{\mathcal F_h^{\mathcal I}}\notag\\
 &+\langle\{\kappa\nabla(\Pi_h w)\},\llbracket\psi\rrbracket
   \rangle_{\mathcal F_h^{\mathcal I}}
 +\langle\{\kappa\nabla(\Pi_h \psi)\},\llbracket w\rrbracket
   \rangle_{\mathcal F_h^{\mathcal I}}\notag\\
 &+\beta_1\langle\tilde{\kappa} \tilde{h}
       \llbracket\nabla^2(\Pi_h w)\rrbracket,\llbracket\psi\rrbracket
   \rangle_{\mathcal F_h^{\mathcal I}}
 +\beta_1\langle\tilde{\kappa} \tilde{h}
       \llbracket\nabla^2(\Pi_h \psi)\rrbracket,\llbracket w\rrbracket
   \rangle_{\mathcal F_h^{\mathcal I}},
 \label{eq:ah_interior}\\
 a_h^{\mathrm D}(w,\psi)
 :={}&2\beta_0\left\langle
       \frac{\tilde{\kappa}}{\tilde{h}}w,\psi
     \right\rangle_{\mathcal F_h^{\mathrm D}}
 +\langle\kappa\nabla(\Pi_h w)\cdot\boldsymbol n^-,\psi
   \rangle_{\mathcal F_h^{\mathrm D}}
 +\langle\kappa\nabla(\Pi_h \psi)\cdot\boldsymbol n^-,w
   \rangle_{\mathcal F_h^{\mathrm D}},
 \label{eq:ah_dirichlet}\\
 a_h^{\mathrm R}(w,\psi)
 :={}&\langle\gamma_{\mathrm R}w,\psi\rangle_{\mathcal F_h^{\mathrm R}}.
 \label{eq:ah_robin}
\end{align}
All differentiated functions in the face terms of
\eqref{eq:ah_interior}--\eqref{eq:ah_dirichlet} are polynomials,
and the coefficient traces are bounded by the assumptions on $\kappa$.  Hence \(a_h^B\) is meaningful on \(W(h)\times W(h)\); when both
arguments belong to \(W_h\), \(\Pi_h\) is the identity and the definition
satisfies
$a_h^B(w,\psi)=(\kappa\nabla w,\nabla\psi)_{\Omega_h}
-\tilde{\theta}_h^{\mathrm{SDDG}}(w,\psi)$ with the same boundary partition.
There is no homogeneous boundary addition on $\Gamma_{\mathrm N}$.
The boundary load functional of Section~\ref{sec:sddg-bc} is
\begin{equation}
 \label{eq:ell-general}
 \begin{aligned}
 \ell_\Gamma(\psi;b)
 ={}&2\beta_0\left\langle
       \frac{\tilde{\kappa}}{\tilde{h}}\frac{b}{\beta},\psi
      \right\rangle_{\mathcal F_h^{\mathrm D}}
 +\left\langle
       \kappa\nabla\psi\cdot\boldsymbol n^-,\frac{b}{\beta}
      \right\rangle_{\mathcal F_h^{\mathrm D}}\\
 &+\left\langle\frac{b}{\alpha},\psi\right\rangle_{\mathcal F_h^{\mathrm N}}
 +\left\langle\frac{b}{\alpha},\psi\right\rangle_{\mathcal F_h^{\mathrm R}} .
 \end{aligned}
\end{equation}

We define the consistency residual with the sign of the uncancelled
projection defects.  For \(s>\frac12\) and any
\((w,\psi) \in H^{1+s}(\Omega) \times W(h)\), \(r_h^B(w;\psi)\) is defined as
\begin{equation}
 \label{eq:residual_def}
 \begin{aligned}
 r_h^B(w;\psi)
 :={}&\langle\{\kappa\nabla\eta_w\},\llbracket\psi\rrbracket
       \rangle_{\mathcal F_h^{\mathcal I}}
 +\beta_1\langle\tilde{\kappa} \tilde{h}
       \llbracket\nabla^2(\Pi_h w)\rrbracket,\llbracket\psi\rrbracket
       \rangle_{\mathcal F_h^{\mathcal I}}\\
 &+\langle\kappa\nabla\eta_w,\llbracket\psi\rrbracket
       \rangle_{\mathcal F_h^{\mathrm D}},
 \end{aligned}
\end{equation}
where \(\eta_w:=\Pi_h w-w\).  The weak equation implies
$\llbracket\kappa\nabla u\rrbracket=0$ on $\mathcal F_h^{\mathcal I}$.
With this sign convention, integration by parts gives the primal consistency identity
\begin{equation}
\label{eq:primal-consistency-identity}
a_h^B(u,\psi)
=\bigl(-\nabla\cdot(\kappa\nabla u),\psi\bigr)_{\Omega_h}
+\ell_\Gamma(\psi;b)+r_h^B(u;\psi),
\qquad \forall \psi\in W_h.
\end{equation}
For later use, the homogeneous adjoint counterpart is also recorded here:
if \(z\in H^2(\Omega)\) satisfies
$\llbracket\kappa\nabla z\rrbracket=0$ on $\mathcal F_h^{\mathcal I}$
and the homogeneous boundary condition
associated with \(B\), namely \(z=0\) on \(\Gamma_{\mathrm D}\),
\(\kappa\nabla z\cdot\boldsymbol n=0\) on \(\Gamma_{\mathrm N}\),
\(\kappa\nabla z\cdot\boldsymbol n+\gamma_{\mathrm R}z=0\) on
\(\Gamma_{\mathrm R}\), and periodic identification when periodic
boundaries are used, then for every \(\chi\in W(h)\),
\begin{equation}
\label{eq:adjoint-consistency-identity}
a_h^B(\chi,z)
=\bigl(\chi,-\nabla\cdot(\kappa\nabla z)\bigr)_{\Omega_h}
+r_h^B(z;\chi).
\end{equation}

Consequently, for the exact solution $u,v$ of Eq.~\eqref{eq:system}, numerical approximation $u_h,v_h$ of Eq.~\eqref{eq:semidisc-bc2} and $\forall w \in W_h$,
\begin{flalign}
    \label{eq:ahu}
   &(\partial_t v,w)_{\Omega_h} + a_h^B(u,w) = (f,w)_{\Omega_h}-(g(u),w)_{\Omega_h} + \ell_\Gamma(w; b) + r_h^B(u;w),\\
   \label{eq:ahuh}
   &(\partial_t v_h,w)_{\Omega_h} + a_h^B(u_h,w)=(f,w)_{\Omega_h} - (g(u_h),w)_{\Omega_h} + \ell_\Gamma(w; b),
\end{flalign}
where \eqref{eq:ahu} follows from \eqref{eq:primal-consistency-identity} and
\eqref{eq:ahuh} follows from the definition of $a_h^B$.

Moreover, we will make use of the Bochner space 
\(L^p(J; \cdot)\), \(1 \leq p \leq \infty\),  equipped with different spatial norms. 
\begin{align*}
    \|w\|_{L^p(J; W(h))} ={}& \begin{cases}
\left( \int_J \|w\|_{h,B}^p \, \mathrm{d}t \right)^{1/p}, & \quad 1 \leq p < \infty, \\
\operatorname*{ess\,sup}\limits_{t \in J} \|w\|_{h,B}, & \quad p = \infty.
\end{cases}
\end{align*}
The norms $\|w\|_{L^p(J;H^m(\Omega_h))}$ and seminorms
$|w|_{L^p(J;H^m(\Omega_h))}$ are defined analogously by replacing
$\|\cdot\|_{h,B}$ with $\|\cdot\|_{m,\Omega_h}$
and $|\cdot|_{m,\Omega_h}$, respectively; $m=0$ gives the $L^2$ norm.
We next recall several
approximation properties and standard inequalities that will be used throughout
the analysis.
\begin{lemma}[Approximation properties]
\label{lem:approx}
Let $s>\frac12$.  For any $w \in H^{1+s}(\Omega)$, there exists a 
constant $C_{\mathrm{proj}}>0$, independent of $h$ and $w$ such that for $0 \le m \le \lfloor 1+s \rfloor$,
\begin{equation}
    \label{eq:L2-proj-approx}
  |w - \Pi_h w|_{m,\Omega_h} 
  \le C_{\mathrm{proj}}\, h^{\min\{k,s\}+1-m}\, \|w\|_{1+s,\Omega_h}.
\end{equation}
We also use the lower-regularity $L^2$ estimate: for any
$q\in H^s(\Omega)$,
\begin{equation}
\label{eq:L2-proj-lower-regularity}
 \|q-\Pi_hq\|_{0,\Omega_h}
 \leq C_{\mathrm{proj}}\,
 h^{\min\{k+1,s\}}\|q\|_{H^s(\Omega)}.
\end{equation}
Taking \(m=0,1\) in \eqref{eq:L2-proj-approx} gives
\begin{subequations}
    \label{eq:proj_error}
\begin{align}
    \|\eta_w\|_{0,\Omega_h}
    &\le C_{\mathrm{proj}}\, h^{\min\{k,s\}+1}\, \|w\|_{1+s,\Omega_h},\\
    \|\nabla \eta_w\|_{0,\Omega_h} = |\eta_w|_{1,\Omega_h}
    &\le C_{\mathrm{proj}}\, h^{\min\{k,s\}}\, \|w\|_{1+s,\Omega_h},
\end{align}
\end{subequations}
When \(s\ge1\), take $m=2$, the additional elementwise estimate
\[
\|\nabla^2\eta_w\|_{0,\Omega_h}
\le C_{\mathrm{proj}}\, h^{\min\{k,s\}-1}\, \|w\|_{1+s,\Omega_h}
\]
also follows from \eqref{eq:L2-proj-approx}.
In addition, there exist constants $C_{\mathrm{tr},\Pi}>0$ and $C_A>0$,
independent of $h$ and $w$, such that
\begin{subequations}
\label{eq:projection_trace_estimates}
\begin{align}
    \label{eq:projection_jump_estimate}
    \| \llbracket \eta_w \rrbracket \|_{0,\mathcal{F}_h^{\mathcal I}}
    &\le C_{\mathrm{tr},\Pi}\,
    h^{\min\{k,s\}+\frac{1}{2}}\,\|w\|_{1+s,\Omega_h},\\
    \label{eq:projection_boundary_trace}
    \|\eta_w\|_{0,\Gamma_{\mathrm D}}
    +\|\eta_w\|_{0,\Gamma_{\mathrm R}}
    &\le C_{\mathrm{tr},\Pi}\,
    h^{\min\{k,s\}+\frac{1}{2}}\,\|w\|_{1+s,\Omega_h},\\
    \label{eq:projection_gradient_trace}
    \| \{\nabla\eta_w\} \|_{0,\mathcal{F}_h^{\mathcal I}}
    +\|\nabla\eta_w\cdot\boldsymbol n^-\|_{0,\Gamma_{\mathrm D}}
    &\le C_{\mathrm{tr},\Pi}\,
    h^{\min\{k,s\}-\frac{1}{2}}\,\|w\|_{1+s,\Omega_h},\\
    \label{eq:projection_energy_estimate}
    \|\eta_w\|_{h,B}
    &\le C_A\,h^{\min\{k,s\}}\,\|w\|_{1+s,\Omega_h},
\end{align}
\end{subequations}
If \(s>\frac32\), then the higher trace estimate
\begin{equation}
\label{eq:projection_hessian_trace}
\|\llbracket\nabla^2\eta_w\rrbracket\|_{0,\mathcal F_h^{\mathcal I}}
\le C_{\mathrm{tr},\Pi}\,
   h^{\min\{k,s\}-\frac{3}{2}}\,
   \|w\|_{1+s,\Omega_h}
\end{equation}
also holds.
The constants $C_{\mathrm{proj}}$ and $C_{\mathrm{tr},\Pi}$ may depend on
$d$, $k$, $s$, the shape-regularity constant, the uniformly bounded
number of faces per element, and $\gamma_{\mathrm{mesh}}$.  The constant
$C_A$ may additionally depend on $\kappa_{\max}$,
$\|\gamma_{\mathrm R}\|_{L^\infty(\Gamma_{\mathrm R})}$, and the penalty
parameter $\lambda$ in
\eqref{eq:boundary-energy-norm}.
\end{lemma}
\begin{proof}
The estimates \eqref{eq:L2-proj-approx} and
\eqref{eq:L2-proj-lower-regularity} follow from standard approximation
properties of the $L^2$-orthogonal projection on shape-regular quasi-uniform
meshes by summing the element-wise estimates \cite{1978110}.
For $s>k$, the approximation order is limited by $k+1$, and the full
$H^{1+s}$ norm controls the required $H^{k+1}$ norm.

Applying the trace theorem for \(H^s(K)\), \(s>\frac12\), to
\(\eta_w\) and \(\nabla\eta_w\), using the corresponding local
Sobolev approximation estimates (including fractional orders),
and then using the quasi-uniformity
\(h_K\simeq h\), gives, for \(j=0,1\),
\begin{equation}
\label{eq:local-trace-projection}
\|\nabla^j\eta_w\|_{0,F}
\le C h^{\min\{k,s\}+1-j-\frac12}
\|w\|_{1+s,K},
\end{equation}
Summing
\eqref{eq:local-trace-projection} over faces and using shape regularity,
which controls the number of faces per element, gives
\eqref{eq:projection_jump_estimate}--\eqref{eq:projection_gradient_trace}.
When \(s>\frac32\), the same argument applied to \(\nabla^2\eta_w\) gives
\eqref{eq:projection_hessian_trace}; this higher trace estimate is only used
when the exact solution has a single-valued Hessian trace.

For the weighted DG seminorm defined above, using $\kappa\leq \kappa_{\max}$,
$\tilde{\kappa}\leq \kappa_{\max}$, and
$1/\tilde{h}\leq C/h$ on interior and Dirichlet faces by face-scale
comparability and quasi-uniformity,
combining \eqref{eq:proj_error} with
\eqref{eq:projection_jump_estimate} and
\eqref{eq:projection_boundary_trace} gives
\eqref{eq:projection_energy_estimate}, after absorbing all fixed
coefficient, trace, and mesh-regularity constants into $C_A$.
\end{proof}
\begin{lemma}[Trace and inverse trace inequalities]
Let \(K\in\Omega_h\), \(F\subset\partial K\), and let \(v_h\in S^k(K)\).
Under the shape-regularity and face-scale comparability assumptions above,
there is a constant \(C_{\mathrm{tr}}\), independent of \(h\), such that
\begin{equation}
\label{eq:trace_inequality}
    \|v_h\|_{0,F}^2
    \le
    C_{\mathrm{tr}}\,\tilde{h}_F^{-1}\|v_h\|_{0,K}^2.
\end{equation}
Consequently, for any \(w\in W_h\),
\begin{equation}
\label{eq:global_inverse_trace}
    \|\tilde{h}^{1/2}\{w\}\|_{0,\mathcal F_h^{\mathcal I}}
    \le C_{\mathrm{tr}}\|w\|_{0,\Omega_h},
    \qquad
    \|\tilde{h}^{1/2}\llbracket w\rrbracket\|_{0,\mathcal F_h^{\mathcal I}}
    \le C_{\mathrm{tr}}\|w\|_{0,\Omega_h}.
\end{equation}
On Dirichlet faces, the corresponding bound is
$\|\tilde h^{1/2}w\|_{0,\Gamma_{\mathrm D}}
\le C_{\mathrm{tr}}\|w\|_{0,\Omega_h}$.
The constants depend
only on \(k\), the shape-regularity constant, and the comparability constants
for \(\tilde{h}_F\); see, e.g.,
\cite{warburton2003constants}.
\end{lemma}
\begin{lemma}[Polynomial inverse inequalities \cite{chen2013estimations}]
For all \(w\in W_h\) and integers \(0\le \ell\le m\le k+1\), there exists a
constant \(C_{\mathrm{inv}}\), independent of \(h\), such that, elementwise,
\[
    |w|_{H^m(K)}
    \le C_{\mathrm{inv}} h_K^{\ell-m}|w|_{H^\ell(K)} .
\]
In particular, on a quasi-uniform mesh,
\begin{equation}
\label{eq:inv_inequality}
    \|\nabla w\|_{0,\Omega_h}
\le \frac{C_{\mathrm{inv}}}{h}\, \|w\|_{0,\Omega_h},
\qquad
    \|\nabla^2 w\|_{0,\Omega_h}
\le \frac{C_{\mathrm{inv}}}{h}\, \|\nabla w\|_{0,\Omega_h}.
\end{equation}
\end{lemma}
\subsection{Properties of the boundary-dependent bilinear form}
\begin{lemma}[Continuity]
There exists a constant $C_{\mathrm{cont}}>0$, independent of $h$, such
that
\begin{equation}
    \label{eq:Continuity}
    |a_h^B(\psi,w)|
    \leq C_{\mathrm{cont}}\|\psi\|_{h,B}\|w\|_{h,B},
    \qquad \forall\psi,w\in W_h.
\end{equation}
\end{lemma}
\begin{proof}
We estimate the terms in
\eqref{eq:ah_interior}--\eqref{eq:ah_robin} directly.  The volume term
satisfies
\[
|(\kappa\nabla\psi,\nabla w)_{\Omega_h}|
\le \|\psi\|_{h,B}\|w\|_{h,B}.
\]
For the penalty terms, Cauchy--Schwarz in the weighted face inner product
gives
\[
\begin{aligned}
&2|\beta_0|\left|
\left\langle
\frac{\tilde{\kappa}}{\tilde{h}}\llbracket\psi\rrbracket,
\llbracket w\rrbracket
\right\rangle_{\mathcal F_h^{\mathcal I}}
+\left\langle
\frac{\tilde{\kappa}}{\tilde{h}}\psi,w
\right\rangle_{\mathcal F_h^{\mathrm D}}
\right|\le
C\|\psi\|_{h,B}\|w\|_{h,B},
\end{aligned}
\]
because the active-face weighted terms are contained in
\eqref{eq:boundary-energy-norm}.  For the consistency term involving
\(\nabla\psi\), Cauchy--Schwarz with the weights
\(\tilde{h}/\tilde{\kappa}\) and \(\tilde{\kappa}/\tilde{h}\) yields
\[
\begin{aligned}
&\left|
\langle\{\kappa\nabla\psi\},\llbracket w\rrbracket
\rangle_{\mathcal F_h^{\mathcal I}}
+\langle\kappa\nabla\psi\cdot\boldsymbol n^-,w
\rangle_{\mathcal F_h^{\mathrm D}}
\right|\\
&\qquad\le
\left(
\left\langle
\frac{\tilde{h}}{\tilde{\kappa}}\{\kappa\nabla\psi\},
\{\kappa\nabla\psi\}
\right\rangle_{\mathcal F_h^{\mathcal I}}
+\left\langle
\frac{\tilde{h}}{\tilde{\kappa}}\kappa\nabla\psi\cdot\boldsymbol n^-,
\kappa\nabla\psi\cdot\boldsymbol n^-
\right\rangle_{\mathcal F_h^{\mathrm D}}
\right)^{1/2}\\
&\qquad\qquad\times
\left(
\left\langle
\frac{\tilde{\kappa}}{\tilde{h}}\llbracket w\rrbracket,
\llbracket w\rrbracket
\right\rangle_{\mathcal F_h^{\mathcal I}}
+\left\langle
\frac{\tilde{\kappa}}{\tilde{h}}w,w
\right\rangle_{\mathcal F_h^{\mathrm D}}
\right)^{1/2}\\
&\qquad\le C\|\psi\|_{h,B}\|w\|_{h,B}.
\end{aligned}
\]
Here the first factor is bounded by the inverse trace inequality, and the
second factor is controlled by the penalty part of \(\|w\|_{h,B}\).  The
term
\[
\langle\{\kappa\nabla w\},\llbracket\psi\rrbracket
\rangle_{\mathcal F_h^{\mathcal I}}
+\langle\kappa\nabla w\cdot\boldsymbol n^-,\psi
\rangle_{\mathcal F_h^{\mathrm D}}
\]
is estimated in the same way after interchanging \(w\) and \(\psi\).

For the Hessian flux, Cauchy--Schwarz gives
\[
\begin{aligned}
&|\beta_1|
\left|
\left\langle
\tilde{\kappa} \tilde{h}\llbracket\nabla^2\psi\rrbracket,
\llbracket w\rrbracket
\right\rangle_{\mathcal F_h^{\mathcal I}}
\right|\\
&\qquad\le
|\beta_1|
\left(
\sum_{F\in\mathcal F_h^{\mathcal I}}\int_F
\tilde{\kappa}\tilde{h}^3\left|\llbracket\nabla^2\psi\rrbracket\right|^2\,\mathrm ds
\right)^{1/2}
\left(
\left\langle
\frac{\tilde{\kappa}}{\tilde{h}}\llbracket w\rrbracket,
\llbracket w\rrbracket
\right\rangle_{\mathcal F_h^{\mathcal I}}
\right)^{1/2}\\
&\qquad\le C|\beta_1|\|\psi\|_{h,B}\|w\|_{h,B},
\end{aligned}
\]
where the last step uses the inverse trace inequality and
\eqref{eq:inv_inequality}.  The other Hessian-flux term is identical with
\(\psi\) and \(w\) interchanged.  Finally,
\[
|\langle\gamma_{\mathrm R}\psi,w\rangle_{\mathcal F_h^{\mathrm R}}|
\le
\langle|\gamma_{\mathrm R}|\psi,\psi\rangle_{\mathcal F_h^{\mathrm R}}^{1/2}
\langle|\gamma_{\mathrm R}|w,w\rangle_{\mathcal F_h^{\mathrm R}}^{1/2}
\le \|\psi\|_{h,B}\|w\|_{h,B}.
\]
Adding all estimates proves \eqref{eq:Continuity}.  The constant is
independent of \(h\).
\end{proof}
\begin{lemma}[Projection-defect estimate]
\label{lemma:projection-defect}
  Let \(s>\frac12\).  For any $w \in H^{1+s}(\Omega)$ and $\psi\in W_h$,
  \begin{equation}
  \label{eq:projection-defect-bound}
    |a_h^B(\Pi_h w-w,\psi)|
    \leq C_{\mathrm{cont}}^{\Pi}h^{\min\{s,k\}}
    \|w\|_{1+s,\Omega_h}\|\psi\|_{h,B}.
  \end{equation}
\end{lemma}
\begin{proof}
Let $\eta_w:=\Pi_h w-w$.  Expanding
\(a_h^B(\Pi_h w,\psi)-a_h^B(w,\psi)\) by
\eqref{eq:ah_interior}--\eqref{eq:ah_robin}, the first-argument derivative
face terms cancel because
\(\Pi_h\eta_w=0\).  The second-argument derivative face terms remain, since
\(\psi\in W_h\) and hence \(\Pi_h\psi=\psi\); they are controlled by inverse
trace estimates below.  The volume term is
\[
|(\kappa\nabla\eta_w,\nabla\psi)_{\Omega_h}|
\le C\|\nabla\eta_w\|_{0,\Omega_h}\|\psi\|_{h,B}
\le Ch^{\min\{s,k\}}\|w\|_{1+s,\Omega_h}\|\psi\|_{h,B}.
\]
For the penalty terms, the boundedness of $\tilde{\kappa}$ and
quasi-uniformity imply
\[
\begin{aligned}
&2|\beta_0|
\left|
\left\langle
\frac{\tilde{\kappa}}{\tilde{h}}\llbracket\eta_w\rrbracket,
\llbracket\psi\rrbracket
\right\rangle_{\mathcal F_h^{\mathcal I}}
+\left\langle
\frac{\tilde{\kappa}}{\tilde{h}}\eta_w,\psi
\right\rangle_{\mathcal F_h^{\mathrm D}}
\right|\\
&\qquad\le
C\left(
h^{-1}\|\llbracket\eta_w\rrbracket\|_{0,\mathcal F_h^{\mathcal I}}^2
+h^{-1}\|\eta_w\|_{0,\Gamma_{\mathrm D}}^2
\right)^{1/2}\|\psi\|_{h,B}\\
&\qquad\le
Ch^{\min\{s,k\}}\|w\|_{1+s,\Omega_h}\|\psi\|_{h,B}.
\end{aligned}
\]
The symmetric consistency terms are estimated by applying the inverse trace inequality
to $\psi$:
\[
\begin{aligned}
&\left|
\langle\{\kappa\nabla\psi\},\llbracket\eta_w\rrbracket
\rangle_{\mathcal F_h^{\mathcal I}}
+\langle\kappa\nabla\psi\cdot\boldsymbol n^-,\eta_w
\rangle_{\mathcal F_h^{\mathrm D}}
\right|\\
&\qquad\le
C\|\psi\|_{h,B}
\left(
h^{-1}\|\llbracket\eta_w\rrbracket\|_{0,\mathcal F_h^{\mathcal I}}^2
+h^{-1}\|\eta_w\|_{0,\Gamma_{\mathrm D}}^2
\right)^{1/2}\\
&\qquad\le
Ch^{\min\{s,k\}}\|w\|_{1+s,\Omega_h}\|\psi\|_{h,B}.
\end{aligned}
\]
The remaining Hessian flux is controlled by the inverse trace inequality
for $\psi$ and \eqref{eq:projection_jump_estimate}:
\[
|\beta_1|
\left|
\left\langle
\tilde{\kappa} \tilde{h}\llbracket\nabla^2\psi\rrbracket,
\llbracket\eta_w\rrbracket
\right\rangle_{\mathcal F_h^{\mathcal I}}
\right|
\le Ch^{\min\{s,k\}}\|w\|_{1+s,\Omega_h}\|\psi\|_{h,B}.
\]
Finally,
\[
|\langle\gamma_{\mathrm R}\eta_w,\psi\rangle_{\mathcal F_h^{\mathrm R}}|
\le
C\|\eta_w\|_{0,\Gamma_{\mathrm R}}\|\psi\|_{h,B}
\le Ch^{\min\{s,k\}+1/2}\|w\|_{1+s,\Omega_h}\|\psi\|_{h,B}.
\]
For \(0<h\le1\), the last bound is also of order
\(h^{\min\{s,k\}}\).  Summing the
preceding estimates proves \eqref{eq:projection-defect-bound}.
\end{proof}
\begin{lemma}[Discrete G{\aa}rding inequality]
For a sufficiently large penalty parameter $\beta_0$, there exist
constants $c_{\mathrm G}>0$ and $C_{\mathrm G}\geq0$, independent of
$h$, such that
\begin{equation}
    \label{eq:garding}
    a_h^B(w,w)
    \geq c_{\mathrm G}\|w\|_{h,B}^2
       -C_{\mathrm G}\|w\|_{0,\Omega_h}^2,
    \qquad\forall w\in W_h.
\end{equation}
If $\gamma_{\mathrm R}^{-}=0$, one may take $C_{\mathrm G}=0$; hence
\begin{equation}
\label{eq:unshifted-coercivity}
a_h^B(w,w)\geq c_{\mathrm G}\|w\|_{h,B}^2,
\qquad\forall w\in W_h.
\end{equation}
Thus the original bilinear form is coercive in the energy seminorm for
periodic, Dirichlet, Neumann, and non-negative Robin conditions.  For
periodic and pure Neumann conditions, $\|\cdot\|_{h,B}$ has the constant
kernel $\mathbb R$, so \eqref{eq:unshifted-coercivity} is a seminorm
coercivity statement.
\end{lemma}
\begin{proof}
By \eqref{eq:ah_interior}--\eqref{eq:ah_robin},
\[
\begin{aligned}
a_h^B(w,w)
={}&(\kappa\nabla w,\nabla w)_{\Omega_h}
+2\beta_0
\left\langle
\frac{\tilde{\kappa}}{\tilde{h}}\llbracket w\rrbracket,
\llbracket w\rrbracket
\right\rangle_{\mathcal F_h^{\mathcal I}}\\
&+2\beta_0
\left\langle
\frac{\tilde{\kappa}}{\tilde{h}}w,w
\right\rangle_{\mathcal F_h^{\mathrm D}}
+2\langle\{\kappa\nabla w\},\llbracket w\rrbracket
\rangle_{\mathcal F_h^{\mathcal I}}\\
&+2\langle\kappa\nabla w\cdot\boldsymbol n^-,w
\rangle_{\mathcal F_h^{\mathrm D}}
+2\beta_1
\left\langle
\tilde{\kappa} \tilde{h}\llbracket\nabla^2w\rrbracket,
\llbracket w\rrbracket
\right\rangle_{\mathcal F_h^{\mathcal I}}\\
&+\langle\gamma_{\mathrm R}^{+}w,w\rangle_{\mathcal F_h^{\mathrm R}}
-\langle\gamma_{\mathrm R}^{-}w,w\rangle_{\mathcal F_h^{\mathrm R}} .
\end{aligned}
\]
With \(\rho_\kappa:=\kappa_{\max}/\kappa_{\min}\), the inverse trace
inequality and Young's inequality give, for any \(\varepsilon_1>0\),
\[
\begin{aligned}
&2\left|
\langle\{\kappa\nabla w\},\llbracket w\rrbracket
\rangle_{\mathcal F_h^{\mathcal I}}
+\langle\kappa\nabla w\cdot\boldsymbol n^-,w
\rangle_{\mathcal F_h^{\mathrm D}}
\right|\\
&\qquad\le
\varepsilon_1(\kappa\nabla w,\nabla w)_{\Omega_h}
+\frac{\rho_\kappa C_{\mathrm{tr},B}^2}{\varepsilon_1}
\left[
\left\langle
\frac{\tilde{\kappa}}{\tilde{h}}\llbracket w\rrbracket,
\llbracket w\rrbracket
\right\rangle_{\mathcal F_h^{\mathcal I}}
+\left\langle
\frac{\tilde{\kappa}}{\tilde{h}}w,w
\right\rangle_{\mathcal F_h^{\mathrm D}}
\right].
\end{aligned}
\]
The same argument, using \eqref{eq:inv_inequality} for the Hessian trace,
gives, for any \(\varepsilon_2>0\),
\[
\begin{aligned}
&2|\beta_1|
\left|
\left\langle
\tilde{\kappa} \tilde{h}\llbracket\nabla^2w\rrbracket,
\llbracket w\rrbracket
\right\rangle_{\mathcal F_h^{\mathcal I}}
\right|\\
&\qquad\le
\varepsilon_2(\kappa\nabla w,\nabla w)_{\Omega_h}
+\frac{4\beta_1^2\rho_\kappa C_{\mathrm{tr},B}^2C_{\mathrm{inv}}^2}
{\varepsilon_2}
\left\langle
\frac{\tilde{\kappa}}{\tilde{h}}\llbracket w\rrbracket,
\llbracket w\rrbracket
\right\rangle_{\mathcal F_h^{\mathcal I}} .
\end{aligned}
\]
Therefore, if \(\varepsilon_1+\varepsilon_2<1\), the gradient term and the
interior and Dirichlet penalty terms have positive remaining coefficients whenever
\begin{equation}
\label{eq:garding-parameter-condition}
  2\beta_0>
  \rho_\kappa C_{\mathrm{tr},B}^2
  \left(\frac{1}{\varepsilon_1}
  +\frac{4\beta_1^2C_{\mathrm{inv}}^2}{\varepsilon_2}\right),
\end{equation}
where $C_{\mathrm{tr},B}$ is the trace constant on the interior
and Dirichlet faces.  Since \(\lambda\) is fixed, under
\eqref{eq:garding-parameter-condition} there is \(c_0>0\), independent of
\(h\), such that
\[
a_h^B(w,w)\ge
c_0\left[
(\kappa\nabla w,\nabla w)_{\Omega_h}
+\lambda^2\left\langle
\frac{\tilde{\kappa}}{\tilde{h}}\llbracket w\rrbracket,
\llbracket w\rrbracket
\right\rangle_{\mathcal F_h^{\mathcal I}}
+\lambda^2\left\langle\frac{\tilde{\kappa}}{\tilde{h}}w,w\right\rangle_{\mathcal F_h^{\mathrm D}}
+\langle\gamma_{\mathrm R}^{+}w,w\rangle_{\mathcal F_h^{\mathrm R}}
\right]
-\langle\gamma_{\mathrm R}^{-}w,w\rangle_{\mathcal F_h^{\mathrm R}} .
\]
For the negative Robin part, applying the BV trace inequality to $w^2$,
using $|\llbracket w^2\rrbracket|=2|\{w\}|\,|\llbracket w\rrbracket|$
and inverse traces, gives
\[
\|w\|_{0,\Gamma_{\mathrm R}}^2
\le C\|w\|_{0,\Omega_h}
\left(\|w\|_{0,\Omega_h}+\|\nabla w\|_{0,\Omega_h}
+\|\tilde h^{-1/2}\llbracket w\rrbracket\|_{0,\mathcal F_h^{\mathcal I}}\right).
\]
Multiplying by $\|\gamma_{\mathrm R}^-\|_{L^\infty}$ and applying Young's
inequality yields, for any $\delta>0$,
\begin{equation}
\label{eq:negative-robin-trace}
\langle\gamma_{\mathrm R}^{-}w,w\rangle_{\mathcal F_h^{\mathrm R}}
\leq \delta\left[
(\kappa\nabla w,\nabla w)_{\Omega_h}
+\lambda^2\left\langle
\frac{\tilde{\kappa}}{\tilde{h}}\llbracket w\rrbracket,
\llbracket w\rrbracket
\right\rangle_{\mathcal F_h^{\mathcal I}}
+\lambda^2\left\langle\frac{\tilde{\kappa}}{\tilde{h}}w,w\right\rangle_{\mathcal F_h^{\mathrm D}}
\right]
+C_\delta\|w\|_{0,\Omega_h}^2 .
\end{equation}
Choose \(\delta<c_0/2\) and absorb the resulting
\((\kappa\nabla w,\nabla w)_{\Omega_h}\) and interior and Dirichlet penalty
contributions into the previous lower bound.  This gives
\[
\begin{aligned}
a_h^B(w,w)\ge
\frac{c_0}{2}\bigg[
(\kappa\nabla w,\nabla w)_{\Omega_h}
+\lambda^2\left\langle
\frac{\tilde{\kappa}}{\tilde{h}}\llbracket w\rrbracket,
\llbracket w\rrbracket
\right\rangle_{\mathcal F_h^{\mathcal I}}
+\lambda^2\left\langle\frac{\tilde{\kappa}}{\tilde{h}}w,w\right\rangle_{\mathcal F_h^{\mathrm D}}
+\langle\gamma_{\mathrm R}^{+}w,w\rangle_{\mathcal F_h^{\mathrm R}}
\bigg]
-C_\delta\|w\|_{0,\Omega_h}^2 .
\end{aligned}
\]
This lower bound controls the positive part of the energy norm, but
\(\|w\|_{h,B}\) also contains the negative Robin trace through
\(|\gamma_{\mathrm R}|=\gamma_{\mathrm R}^{+}+\gamma_{\mathrm R}^{-}\).
Indeed, using \eqref{eq:negative-robin-trace} once more,
\[
\begin{aligned}
\|w\|_{h,B}^2
={}&
(\kappa\nabla w,\nabla w)_{\Omega_h}
+\lambda^2\left\langle
\frac{\tilde{\kappa}}{\tilde{h}}\llbracket w\rrbracket,
\llbracket w\rrbracket
\right\rangle_{\mathcal F_h^{\mathcal I}}
+\lambda^2\left\langle\frac{\tilde{\kappa}}{\tilde{h}}w,w\right\rangle_{\mathcal F_h^{\mathrm D}}\\
&\quad
+\langle\gamma_{\mathrm R}^{+}w,w\rangle_{\mathcal F_h^{\mathrm R}}
+\langle\gamma_{\mathrm R}^{-}w,w\rangle_{\mathcal F_h^{\mathrm R}}\\
&\le
\bigg[
(\kappa\nabla w,\nabla w)_{\Omega_h}
+\lambda^2\left\langle
\frac{\tilde{\kappa}}{\tilde{h}}\llbracket w\rrbracket,
\llbracket w\rrbracket
\right\rangle_{\mathcal F_h^{\mathcal I}}
+\lambda^2\left\langle\frac{\tilde{\kappa}}{\tilde{h}}w,w\right\rangle_{\mathcal F_h^{\mathrm D}}\\
&\qquad
+\langle\gamma_{\mathrm R}^{+}w,w\rangle_{\mathcal F_h^{\mathrm R}}
\bigg]
+\delta\bigg[
(\kappa\nabla w,\nabla w)_{\Omega_h}
+\lambda^2\left\langle
\frac{\tilde{\kappa}}{\tilde{h}}\llbracket w\rrbracket,
\llbracket w\rrbracket
\right\rangle_{\mathcal F_h^{\mathcal I}}
+\lambda^2\left\langle\frac{\tilde{\kappa}}{\tilde{h}}w,w\right\rangle_{\mathcal F_h^{\mathrm D}}
\bigg]
\\
&\qquad
+C_\delta\|w\|_{0,\Omega_h}^2 .
\end{aligned}
\]
Since the positive Robin term is non-negative, this implies
\[
\begin{aligned}
\|w\|_{h,B}^2
&\le
(1+\delta)\bigg[
(\kappa\nabla w,\nabla w)_{\Omega_h}\\
&\qquad
+\lambda^2\left\langle
\frac{\tilde{\kappa}}{\tilde{h}}\llbracket w\rrbracket,
\llbracket w\rrbracket
\right\rangle_{\mathcal F_h^{\mathcal I}}
+\lambda^2\left\langle\frac{\tilde{\kappa}}{\tilde{h}}w,w\right\rangle_{\mathcal F_h^{\mathrm D}}
+\langle\gamma_{\mathrm R}^{+}w,w\rangle_{\mathcal F_h^{\mathrm R}}
\bigg]
+C_\delta\|w\|_{0,\Omega_h}^2 ,
\end{aligned}
\]
and hence
\[
\begin{aligned}
&(\kappa\nabla w,\nabla w)_{\Omega_h}
+\lambda^2\left\langle
\frac{\tilde{\kappa}}{\tilde{h}}\llbracket w\rrbracket,
\llbracket w\rrbracket
\right\rangle_{\mathcal F_h^{\mathcal I}}
+\lambda^2\left\langle\frac{\tilde{\kappa}}{\tilde{h}}w,w\right\rangle_{\mathcal F_h^{\mathrm D}}\\
&\qquad
+\langle\gamma_{\mathrm R}^{+}w,w\rangle_{\mathcal F_h^{\mathrm R}}\\
&\qquad\ge
\frac{1}{1+\delta}\|w\|_{h,B}^2
-\frac{C_\delta}{1+\delta}\|w\|_{0,\Omega_h}^2 .
\end{aligned}
\]
Substituting this lower bound into the preceding estimate for
\(a_h^B(w,w)\) gives
\[
\begin{aligned}
a_h^B(w,w)
&\ge
\frac{c_0}{2(1+\delta)}\|w\|_{h,B}^2
-\left(
C_\delta+\frac{c_0C_\delta}{2(1+\delta)}
\right)\|w\|_{0,\Omega_h}^2 .
\end{aligned}
\]
Thus \eqref{eq:garding} holds with, for instance,
\begin{equation}
  c_{\mathrm G}:=\frac{c_0}{2(1+\delta)},\quad
C_{\mathrm G}:=
C_\delta+\frac{c_0C_\delta}{2(1+\delta)} .
\end{equation}
If
\(\gamma_{\mathrm R}^{-}=0\), the trace-defect argument is unnecessary,
and the positive lower bound before \eqref{eq:negative-robin-trace} gives
\eqref{eq:unshifted-coercivity} with \(C_{\mathrm G}=0\).
\end{proof}
\begin{lemma}[Admissible range of $\beta_0,\beta_1$]
\label{lemma:ddg-param-range}
Based on the positive energy part of the G{\aa}rding estimate for
$a_h^B(\cdot,\cdot)$, a sufficient
range for the penalty parameters is
\begin{align}
\label{eq:ddg_param_range}
\beta_0>\frac12\rho_\kappa C_{\mathrm{tr},B}^2
\left(1+2|\beta_1|C_{\mathrm{inv}}\right)^2,
\qquad \rho_\kappa:=\frac{\kappa_{\max}}{\kappa_{\min}}.
\end{align}
No sign restriction on $\beta_1$ is needed for the shifted-projection
argument.  The range \eqref{eq:ddg_param_range} is used in the analysis as a
sufficient condition for the coercivity and projection estimates that yield the
stated convergence rates; it should not be interpreted as a necessary condition
for convergence.  In particular, the constants entering
\eqref{eq:ddg_param_range} arise from trace, inverse, and Young inequalities,
so the resulting lower bound is not expected to be sharp and may be stronger
than what is needed in practice.
\end{lemma}
\begin{lemma}[Shifted coercivity and strong monotonicity]
\label{lemma:shifted-coercivity}
Let $c_{\mathrm G}$ and $C_{\mathrm G}$ be the constants in
\eqref{eq:garding}.  Choose a fixed shift satisfying
\begin{equation}
\label{eq:shift-parameter-condition}
\mu>C_{\mathrm G}+L_g,
\end{equation}
and define
\begin{equation}
\label{eq:shifted-form}
a_{h,\mu}^B(w,\psi)
:=a_h^B(w,\psi)+\mu(w,\psi)_{\Omega_h},
\qquad
\|w\|_{h,\mu,B}^2
:=\|w\|_{h,B}^2+\mu\|w\|_{0,\Omega_h}^2.
\end{equation}
Then $a_{h,\mu}^B$ is coercive on $W_h$:
\begin{equation}
\label{eq:shifted-coercivity}
a_{h,\mu}^B(w,w)
\geq C_{\mathrm{shift}}\|w\|_{h,\mu,B}^2,
\qquad
C_{\mathrm{shift}}
:=\min\left\{c_{\mathrm G},
\frac{\mu-C_{\mathrm G}}{\mu}\right\}>0.
\end{equation}
Moreover, for $p,q\in W_h$, the nonlinear shifted operator is strongly monotone:
\begin{equation}
\label{eq:shifted-monotonicity}
\begin{aligned}
&a_{h,\mu}^B(p-q,p-q)
+(g(p)-g(q),p-q)_{\Omega_h}\\
&\qquad\geq
c_{\mathrm G}\|p-q\|_{h,B}^2
+(\mu-C_{\mathrm G}-L_g)\|p-q\|_{0,\Omega_h}^2 .
\end{aligned}
\end{equation}
\end{lemma}
\begin{proof}
Adding $\mu\|w\|_{0,\Omega_h}^2$ to both sides of
\eqref{eq:garding} gives
\[
a_{h,\mu}^B(w,w)
\geq c_{\mathrm G}\|w\|_{h,B}^2
+(\mu-C_{\mathrm G})\|w\|_{0,\Omega_h}^2,
\]
and hence
\[
a_{h,\mu}^B(w,w)
\ge
\min\left\{c_{\mathrm G},\frac{\mu-C_{\mathrm G}}{\mu}\right\}
\left(\|w\|_{h,B}^2+\mu\|w\|_{0,\Omega_h}^2\right).
\]
This is \eqref{eq:shifted-coercivity}.  Let \(r:=p-q\).  The Lipschitz
condition \eqref{eq:g_lipschitz} implies pointwise that
\((g(p)-g(q))r\ge -L_g r^2\), and therefore
\[
(g(p)-g(q),p-q)_{\Omega_h}
\geq-L_g\|p-q\|_{0,\Omega_h}^2.
\]
Combining this inequality with
\[
a_{h,\mu}^B(r,r)
\ge
c_{\mathrm G}\|r\|_{h,B}^2
+(\mu-C_{\mathrm G})\|r\|_{0,\Omega_h}^2
\]
proves \eqref{eq:shifted-monotonicity}.
\end{proof}
\begin{remark}[Shifted adjoint regularity]
\label{rem:adjoint-regularity}
For the optimal $L^2$ estimate, assume elliptic regularity for the shifted
adjoint operator.  More precisely, for every
$c\in L^\infty(\Omega)$ with $|c|\leq L_g$ and every
$\varphi\in L^2(\Omega)$, the problem
\begin{subequations}
\label{eq:shifted-adjoint-regularity}
\begin{equation}
 -\nabla\cdot(\kappa\nabla z)+(\mu+c)z=\varphi
\label{eq:shifted-adjoint-problem}
\end{equation}
is equipped with the homogeneous adjoint boundary condition matching the
boundary type used in the primal problem:
\begin{align}
z&=0 &&\text{on }\Gamma_{\mathrm D},\label{eq:shifted-adjoint-dirichlet}\\
\kappa\nabla z\cdot\boldsymbol n&=0
&&\text{on }\Gamma_{\mathrm N},\label{eq:shifted-adjoint-neumann}\\
\kappa\nabla z\cdot\boldsymbol n+\gamma_{\mathrm R}z&=0
&&\text{on }\Gamma_{\mathrm R},\label{eq:shifted-adjoint-robin}
\end{align}
and with periodic identification on periodic boundary pairs.  For each fixed
choice of these boundary parts, we assume that its weak solution satisfies
\begin{equation}
 \|z\|_{2,\Omega}\leq C_S\|\varphi\|_{0,\Omega}.
\label{eq:shifted-adjoint-bound}
\end{equation}
\end{subequations}
The weak formulation implies $\llbracket\kappa\nabla z\rrbracket=0$
on $\mathcal F_h^{\mathcal I}$, as required in
\eqref{eq:adjoint-consistency-identity}.
\end{remark}
\begin{lemma}[Projection consistency and adjoint identity]
\label{lem:adjoint-consistency}
\label{lemma:C_R}
Let \(r_h^B\) be defined by \eqref{eq:residual_def}.  If \(\frac12<s\le1\) and
\(w\in H^{1+s}(\Omega)\), then
\begin{equation}
\label{eq:residual-low-regularity}
|r_h^B(w;\psi)|
\le C_{\mathrm R}h^s\|w\|_{1+s,\Omega_h}\|\psi\|_{h,B},
\qquad \forall \psi\in W(h).
\end{equation}
For the higher-order exact-solution estimate, if \(s>\frac32\),
\(w\in H^{1+s}(\Omega)\), and the Hessian of \(w\) has a single-valued
interface trace, then
\begin{subequations}
\label{eq:projection-consistency-remainders}
\begin{align}
\label{eq:residual_inequality}
|r_h^B(w;\psi)|
&\le C_{\mathrm R}h^{\min\{s,k\}}
\|w\|_{1+s,\Omega_h}\|\psi\|_{h,B},
\qquad \forall \psi\in W(h),\\
\label{eq:residual-adjoint-test}
|r_h^B(w;\Pi_h \zeta-\zeta)|
&\le C_{\mathrm R}h^{\min\{s,k\}+1}
\|w\|_{1+s,\Omega_h}\|\zeta\|_{2,\Omega_h},
\end{align}
\end{subequations}
for every \(\zeta\in H^2(\Omega)\) satisfying the homogeneous adjoint
Dirichlet condition when \(\Gamma_{\mathrm D}\neq\emptyset\). 
\end{lemma}
\begin{proof}
Let \(\eta_w:=\Pi_h w-w\).  First assume \(s>\frac12\).  By Cauchy--Schwarz,
the first term in \eqref{eq:residual_def} satisfies
\[
\begin{aligned}
\left|
\langle\{\kappa\nabla\eta_w\},\llbracket\psi\rrbracket
\rangle_{\mathcal F_h^{\mathcal I}}
\right|
&\le
C h^{1/2}
\left(\sum_{K\in\Omega_h}
\|\nabla\eta_w\|_{0,\partial K}^2\right)^{1/2}
\left(
\left\langle
\frac{\tilde{\kappa}}{\tilde{h}}\llbracket\psi\rrbracket,
\llbracket\psi\rrbracket
\right\rangle_{\mathcal F_h^{\mathcal I}}
\right)^{1/2}\\
&\le
Ch^{\min\{s,k\}}
\|w\|_{1+s,\Omega_h}\|\psi\|_{h,B}.
\end{aligned}
\]
Here we used
$|\{\kappa\nabla\eta_w\}|\le
\frac{\kappa_{\max}}2(|\nabla\eta_w^+|+|\nabla\eta_w^-|)$,
the local trace estimate \eqref{eq:local-trace-projection}, and the penalty
part of \(\|\psi\|_{h,B}\).
For the Hessian flux, Cauchy--Schwarz first gives
\[
\begin{aligned}
&|\beta_1|
\left|
\left\langle
\tilde{\kappa} \tilde{h}\llbracket\nabla^2(\Pi_h w)\rrbracket,
\llbracket\psi\rrbracket
\right\rangle_{\mathcal F_h^{\mathcal I}}
\right|\\
&\qquad\le
C
\left(
\sum_{F\in\mathcal F_h^{\mathcal I}}\int_F
\tilde{\kappa}\tilde{h}^3\left|\llbracket\nabla^2(\Pi_h w)\rrbracket\right|^2\,\mathrm ds
\right)^{1/2}
\left(
\left\langle
\frac{\tilde{\kappa}}{\tilde{h}}\llbracket\psi\rrbracket,
\llbracket\psi\rrbracket
\right\rangle_{\mathcal F_h^{\mathcal I}}
\right)^{1/2}\\
&\qquad\le
C
\left(
\sum_{F\in\mathcal F_h^{\mathcal I}}\int_F
\tilde{\kappa}\tilde{h}^3\left|\llbracket\nabla^2(\Pi_h w)\rrbracket\right|^2\,\mathrm ds
\right)^{1/2}\|\psi\|_{h,B}.
\end{aligned}
\]
If \(\frac12<s\le1\), the inverse trace inequality for the polynomial
\(\Pi_h w\), local approximation by affine polynomials, and the \(L^2\)
stability of \(\Pi_h\) give
\[
\left(
\sum_{F\in\mathcal F_h^{\mathcal I}}\int_F
\tilde{\kappa}\tilde{h}^3\left|\llbracket\nabla^2(\Pi_h w)\rrbracket\right|^2\,\mathrm ds
\right)^{1/2}
\le Ch^s\|w\|_{1+s,\Omega_h}.
\]
This is the low-regularity estimate used for the \(H^2\) adjoint case.
If \(w\) has a single-valued Hessian trace on interior faces, then
\[
\llbracket\nabla^2(\Pi_h w)\rrbracket
=\llbracket\nabla^2(\Pi_h w-w)\rrbracket
=\llbracket\nabla^2\eta_w\rrbracket ,
\]
and, when \(s>\frac32\), \eqref{eq:projection_hessian_trace} yields the
sharper bound
\[
\left(
\sum_{F\in\mathcal F_h^{\mathcal I}}\int_F
\tilde{\kappa}\tilde{h}^3\left|\llbracket\nabla^2(\Pi_h w)\rrbracket\right|^2\,\mathrm ds
\right)^{1/2}
\le Ch^{\min\{s,k\}}\|w\|_{1+s,\Omega_h}.
\]
The
Dirichlet boundary term is bounded in the same way as the first term:
\[
\begin{aligned}
\left|
\langle\kappa\nabla\eta_w,\llbracket\psi\rrbracket
\rangle_{\mathcal F_h^{\mathrm D}}
\right|
&\le
C\|\nabla\eta_w\cdot\boldsymbol n^-\|_{0,\Gamma_{\mathrm D}}
h^{1/2}
\left(
\left\langle
\frac{\tilde{\kappa}}{\tilde{h}}\llbracket\psi\rrbracket,
\llbracket\psi\rrbracket
\right\rangle_{\mathcal F_h^{\mathrm D}}
\right)^{1/2}\\
&\le
Ch^{\min\{s,k\}}
\|w\|_{1+s,\Omega_h}\|\psi\|_{h,B}.
\end{aligned}
\]
Adding the three estimates proves \eqref{eq:residual-low-regularity} in the
case \(\frac12<s\le1\), and proves \eqref{eq:residual_inequality} in the
single-valued Hessian case.

If \(\psi=\Pi_h \zeta-\zeta\), the same estimates apply with the penalty
factor replaced by the \(H^2\) projection trace bound
\[
h^{-1/2}\|\llbracket \Pi_h \zeta-\zeta\rrbracket
\|_{0,\mathcal F_h^{\mathcal I}}
+h^{-1/2}\|\Pi_h \zeta-\zeta\|_{0,\Gamma_{\mathrm D}}
\le Ch\|\zeta\|_{2,\Omega_h}.
\]
On \(\Gamma_{\mathrm D}\), the homogeneous adjoint condition gives
\(\zeta=0\), so the boundary trace of
\(\Pi_h\zeta-\zeta\) is precisely its Dirichlet jump.  The first and
Dirichlet residual terms therefore gain one power of \(h\).  The Hessian
residual term gains the same factor because the single-valued Hessian
assumption gives the bound on
\(\llbracket\nabla^2(\Pi_h w)\rrbracket\) above, while the test factor is now
bounded by the displayed \(H^2\) projection trace estimate.  This proves
\eqref{eq:residual-adjoint-test}.  If \(\Gamma_{\mathrm D}=\emptyset\), the
last term in \eqref{eq:residual_def} vanishes.
\end{proof}
\begin{theorem}[Spatial error estimates]
  Under the preceding regularity, continuity, and G{\aa}rding assumptions,
  the error between the exact solution $(u,v)$ of
  \eqref{eq:system} and the SDDG solution $(u_h,v_h)$, with any of the
  boundary conditions considered above and initial values
  $u_h(0)=\Pi_hu_0$, $v_h(0)=\Pi_hv_0$, satisfies, for sufficiently small $h$,
  \begin{align}
\label{eq:spatial-error-estimates}
      \|e_u\|_{L^\infty(J;L^2(\Omega))}
      &\leq C_Uh^{\min\{\sigma,k\}+1},&
      \|e_v\|_{L^\infty(J;L^2(\Omega))}
      &\leq C_Vh^{\min\{\sigma,k\}}.
  \end{align}
  For smooth solutions, the theorem yields
    \begin{align*}
        \|e_u\|_{L^\infty(J;L^2(\Omega))}
        &\leq C_Uh^{k+1},&
        \|e_v\|_{L^\infty(J;L^2(\Omega))}
        &\leq C_Vh^k.
  \end{align*}
\end{theorem}
\subsection{Error estimate of $u$}
\begin{lemma}
\label{lemma:bound_w_h}
  For an exact solution $u\in H^{1+\sigma}(\Omega)$, let $w_h\in W_h$ be the solution of 
  \begin{equation}
    \label{eq:eval_w_h}
    a_{h,\mu}^B(w_h,w) + (g(w_h),w)_{\Omega_h}
    = a_{h,\mu}^B(u,w) + (g(u),w)_{\Omega_h}
      - r_h^B(u;w),\forall w \in W_h,
  \end{equation}
  Then we have
\begin{equation}
\label{eq:nonlinear-projection-estimates}
  \|u-w_h\|_{h,B} \leq C_Eh^{\min\{\sigma,k\}}\|u\|_{1+\sigma,\Omega_h},
\quad
  \|u-w_h\|_{0,\Omega_h} \leq C_Lh^{\min\{\sigma,k\}+1}\|u\|_{1+\sigma,\Omega_h}.
\end{equation}
\end{lemma}
\begin{proof}
The operator on the left-hand side of \eqref{eq:eval_w_h} is strongly
monotone by \eqref{eq:shifted-monotonicity}; hence $w_h$ exists uniquely.
Set
\[
  e:=u-w_h,\qquad
  \delta_h:=\Pi_hu-w_h.
\]
By the triangle inequality and \eqref{eq:projection_energy_estimate},
\begin{equation}
\label{eq:projection-triangle-h}
\|e\|_{h,B}
\leq \|\eta_u\|_{h,B}+\|\delta_h\|_{h,B}
\leq C_Ah^{\min\{\sigma,k\}}\|u\|_{1+\sigma,\Omega_h}
+\|\delta_h\|_{h,B}.
\end{equation}
It remains first to estimate \(\delta_h\).  From
\eqref{eq:eval_w_h}, tested with \(\delta_h\), we have
\[
\begin{aligned}
a_{h,\mu}^B(\delta_h,\delta_h)
={}&a_{h,\mu}^B(\eta_u,\delta_h)
+(g(w_h)-g(u),\delta_h)_{\Omega_h}
+r_h^B(u;\delta_h).
\end{aligned}
\]
We now estimate the three terms on the right-hand side explicitly.  By
\eqref{eq:shifted-coercivity},
\[
C_{\mathrm{shift}}\|\delta_h\|_{h,\mu,B}^2
\le a_{h,\mu}^B(\delta_h,\delta_h).
\]
For the first term, the shifted form gives
\[
\begin{aligned}
|a_{h,\mu}^B(\eta_u,\delta_h)|
&\le |a_h^B(\eta_u,\delta_h)|
   +\mu\|\eta_u\|_{0,\Omega_h}\|\delta_h\|_{0,\Omega_h}\\
&\le C_{\mathrm{cont}}^{\Pi}
h^{\min\{\sigma,k\}}\|u\|_{1+\sigma,\Omega_h}
\|\delta_h\|_{h,B}
+ C h^{\min\{\sigma,k\}+1}\|u\|_{1+\sigma,\Omega_h}
\|\delta_h\|_{h,\mu,B}\\
&\le
C C_{\mathrm{cont}}^{\Pi}
h^{\min\{\sigma,k\}}\|u\|_{1+\sigma,\Omega_h}
\|\delta_h\|_{h,\mu,B}.
\end{aligned}
\]
Here the second line uses the projection-defect estimate
\eqref{eq:projection-defect-bound} for the unshifted form,
\eqref{eq:proj_error} for
\(\|\eta_u\|_{0,\Omega_h}\), and
\(\|\delta_h\|_{h,B}\le\|\delta_h\|_{h,\mu,B}\); in the last line we use
\(h\le 1\) and the fixed value of \(\mu\), enlarging
\(C_{\mathrm{cont}}^{\Pi}\), if necessary, to include this fixed mass
contribution.  The nonlinear term is bounded by the Lipschitz continuity of
\(g\):
\[
\begin{aligned}
|(g(w_h)-g(u),\delta_h)_{\Omega_h}|
&\le L_g\|w_h-u\|_{0,\Omega_h}
       \|\delta_h\|_{0,\Omega_h}\\
&=L_g\|e\|_{0,\Omega_h}\|\delta_h\|_{0,\Omega_h}
\le C L_g\|e\|_{0,\Omega_h}\|\delta_h\|_{h,\mu,B}.
\end{aligned}
\]
The residual term is controlled by Lemma~\ref{lemma:C_R}:
\[
|r_h^B(u;\delta_h)|
\le C_{\mathrm R}h^{\min\{\sigma,k\}}\|u\|_{1+\sigma,\Omega_h}
\|\delta_h\|_{h,B}
\le C_{\mathrm R}h^{\min\{\sigma,k\}}\|u\|_{1+\sigma,\Omega_h}
\|\delta_h\|_{h,\mu,B}.
\]
Combining these three bounds with the identity above gives
\[
\begin{aligned}
C_{\mathrm{shift}}\|\delta_h\|_{h,\mu,B}^2
&\le
C\left(C_{\mathrm{cont}}^{\Pi}+C_{\mathrm R}\right)
h^{\min\{\sigma,k\}}\|u\|_{1+\sigma,\Omega_h}
\|\delta_h\|_{h,\mu,B}\\
&\quad
+CL_g\|e\|_{0,\Omega_h}\|\delta_h\|_{h,\mu,B}.
\end{aligned}
\]
If \(\delta_h=0\), the desired estimate is immediate.  Otherwise, dividing
by \(\|\delta_h\|_{h,\mu,B}\), using
\(\|\delta_h\|_{h,B}\le\|\delta_h\|_{h,\mu,B}\), and absorbing
\(C_{\mathrm{shift}}^{-1}\) into the generic constant yields
\begin{equation}
\label{eq:nonlinear-projection-energy-identity}
\begin{aligned}
\|\delta_h\|_{h,B}
&\leq
C\left(C_{\mathrm{cont}}^{\Pi}+C_{\mathrm R}\right)
h^{\min\{\sigma,k\}}\|u\|_{1+\sigma,\Omega_h}
+CL_g\|e\|_{0,\Omega_h}.
\end{aligned}
\end{equation}
Substitution in \eqref{eq:projection-triangle-h} gives
\begin{equation}
\label{eq:discrete-projection-energy-bound}
\|e\|_{h,B}
\leq
C h^{\min\{\sigma,k\}}\|u\|_{1+\sigma,\Omega_h}
+CL_g\|e\|_{0,\Omega_h}.
\end{equation}

We now prove the \(L^2\) estimate by duality.  Define the averaged derivative
\[
c_e(x):=\int_0^1
g'\bigl(u(x)+\theta(w_h(x)-u(x))\bigr)\,d\theta .
\]
Then \(c_e\in L^\infty(\Omega)\), \(|c_e|\le L_g\), and the fundamental
theorem of calculus gives
\[
g(w_h)-g(u)=-c_e e .
\]
Let \(z\) solve the shifted adjoint problem
\[
 -\nabla\cdot(\kappa\nabla z)+(\mu+c_e)z=e
\]
with the homogeneous boundary condition associated with \(B\).  By
\eqref{eq:shifted-adjoint-bound},
\[
\|z\|_{2,\Omega_h}\leq C_S\|e\|_{0,\Omega_h}.
\]
Since \(z\) is single-valued and satisfies the homogeneous adjoint boundary
condition, \(r_h^B(u;z)=0\).  Moreover,
\eqref{eq:adjoint-consistency-identity} gives
\[
\|e\|_{0,\Omega_h}^2
=a_{h,\mu}^B(e,z)+(c_e e,z)_{\Omega_h}-r_h^B(z;e).
\]
Splitting \(z=(z-\Pi_h z)+\Pi_h z\) and using \eqref{eq:eval_w_h} with
\(\Pi_h z\), together with \(g(w_h)-g(u)=-c_e e\), gives the exact identity
\begin{equation}
\label{eq:nonlinear-projection-duality}
\begin{aligned}
\|e\|_{0,\Omega_h}^2
={}&a_{h,\mu}^B(e,z-\Pi_h z)
+(g(w_h)-g(u),\Pi_h z-z)_{\Omega_h}\\
{}&+r_h^B(u;\Pi_h z-z)
-r_h^B(z;e).
\end{aligned}
\end{equation}
We bound the terms on the right-hand side one by one.  For the first term,
split \(e=\delta_h-\eta_u\), where
\(\delta_h=\Pi_hu-w_h\in W_h\) and \(\eta_u=\Pi_hu-u\).  Since \(a_h^B\)
is symmetric and \(z-\Pi_hz=-(\Pi_hz-z)\),
\eqref{eq:projection-defect-bound} with \(w=z\) gives
\[
|a_h^B(\delta_h,z-\Pi_hz)|
=|a_h^B(\Pi_hz-z,\delta_h)|
\le Ch\|z\|_{2,\Omega}\|\delta_h\|_{h,B}.
\]
For the remaining unshifted part, note that
\(\Pi_h\eta_u=0\) and \(\Pi_h(z-\Pi_hz)=0\).  Hence all consistency terms in
\eqref{eq:ah_interior}--\eqref{eq:ah_dirichlet} which contain
\(\nabla(\Pi_h\cdot)\) or \(\nabla^2(\Pi_h\cdot)\) vanish, and only the
volume, penalty, and Robin terms remain.  Cauchy--Schwarz and
\eqref{eq:projection_energy_estimate} therefore give
\[
|a_h^B(\eta_u,z-\Pi_hz)|
\le C\|\eta_u\|_{h,B}\|z-\Pi_hz\|_{h,B}
\le Ch^{\min\{\sigma,k\}+1}\|u\|_{1+\sigma,\Omega_h}\|z\|_{2,\Omega}.
\]
The shifted mass contribution satisfies
\[
\mu|(e,z-\Pi_hz)_{\Omega_h}|
\le Ch^2\|e\|_{0,\Omega_h}\|z\|_{2,\Omega}.
\]
Combining these bounds with
\(\|\delta_h\|_{h,B}\le\|e\|_{h,B}+\|\eta_u\|_{h,B}\) and
\eqref{eq:projection_energy_estimate} gives
\begin{equation}
\label{eq:adjoint-approximation-bound}
\begin{aligned}
|a_{h,\mu}^B(e,z-\Pi_h z)|
&\leq
Ch\left(\|e\|_{h,B}
+h^{\min\{\sigma,k\}}\|u\|_{1+\sigma,\Omega_h}\right)\|z\|_{2,\Omega}\\
&\quad+Ch^2\|e\|_{0,\Omega_h}\|z\|_{2,\Omega}.
\end{aligned}
\end{equation}
The remaining nonlinear projection term satisfies
\[
|(g(w_h)-g(u),\Pi_h z-z)_{\Omega_h}|
\leq CL_gh^2\|e\|_{0,\Omega_h}\|z\|_{2,\Omega}.
\]
Furthermore, \eqref{eq:residual-adjoint-test} gives
\[
|r_h^B(u;\Pi_h z-z)|
\leq C_Rh^{\min\{\sigma,k\}+1}\|u\|_{1+\sigma,\Omega_h}\|z\|_{2,\Omega},
\]
and \eqref{eq:residual-low-regularity} with \(s=1\) gives
\[
|r_h^B(z;e)|\leq C_{\mathrm R}h\|z\|_{2,\Omega}\|e\|_{h,B}.
\]
Inserting these bounds into \eqref{eq:nonlinear-projection-duality} and
using \(\|z\|_{2,\Omega}\leq C_S\|e\|_{0,\Omega_h}\), we obtain
\[
\begin{aligned}
\|e\|_{0,\Omega_h}^2
&\leq
C h\left(\|e\|_{h,B}
+h^{\min\{\sigma,k\}}\|u\|_{1+\sigma,\Omega_h}\right)
\|e\|_{0,\Omega_h}\\
&\quad
+Ch^2\|e\|_{0,\Omega_h}^2 .
\end{aligned}
\]
For sufficiently small \(h\), the last term can be absorbed into the left-hand
side.  Thus
\begin{equation}
\label{eq:duality-intermediate-estimate}
\begin{aligned}
\|e\|_{0,\Omega_h}
&\leq
Ch\|e\|_{h,B}
+Ch^{\min\{\sigma,k\}+1}\|u\|_{1+\sigma,\Omega_h}.
\end{aligned}
\end{equation}
Combining \eqref{eq:duality-intermediate-estimate} with
\eqref{eq:discrete-projection-energy-bound} gives
\[
\|e\|_{h,B}
\leq
Ch^{\min\{\sigma,k\}}\|u\|_{1+\sigma,\Omega_h}
+CL_gh\|e\|_{h,B}.
\]
For sufficiently small \(h\), the last term is absorbed into the left-hand
side, and hence
\[
\|u-w_h\|_{h,B}
\leq C_Eh^{\min\{\sigma,k\}}\|u\|_{1+\sigma,\Omega_h}.
\]
Substituting this bound back into
\eqref{eq:duality-intermediate-estimate} proves
\begin{equation}
\label{eq:L2_bound_u-w_h}
\|u-w_h\|_{0,\Omega_h}
\leq C_Lh^{\min\{\sigma,k\}+1}\|u\|_{1+\sigma,\Omega_h}.
\end{equation}
\end{proof}

\begin{lemma}
\label{lemma:bound_w_h_partial_t}
For the solution $w_h$ of \eqref{eq:eval_w_h}, there exist constants
$C_E'$ and $C_L'$, independent of $h$, such that
\begin{equation}
\label{eq:time-derivative-projection-estimates}
  \|\partial_t (u-w_h)\|_{h,B}
  \leq C_E'h^{\min\{\sigma,k\}}, \quad
  \|\partial_t (u-w_h)\|_{0,\Omega_h}
  \leq C_L'h^{\min\{\sigma,k\}+1}.
\end{equation}
\end{lemma}
\begin{proof}
Differentiate \eqref{eq:eval_w_h} for a fixed test function
\(w\in W_h\).  Since the mesh, \(\Pi_h\), \(a_{h,\mu}^B\), and \(\mu\) are
time-independent, and since \(\partial_tu=v\), the residual is linear in its
first argument and satisfies
\(\partial_t r_h^B(u;w)=r_h^B(v;w)\).  Hence
\[
a_{h,\mu}^B(\partial_t w_h,w)
+(g'(w_h)\partial_t w_h,w)_{\Omega_h}
=a_{h,\mu}^B(v,w)+(g'(u)v,w)_{\Omega_h}-r_h^B(v;w).
\]
Subtracting
\(a_{h,\mu}^B(v,w)+(g'(w_h)v,w)_{\Omega_h}\) from both sides gives
\begin{equation}
\label{eq:differentiated-projection-equation}
\begin{aligned}
 a_{h,\mu}^B(\partial_tw_h-v,w)
 +(g'(w_h)(\partial_tw_h-v),w)_{\Omega_h}
 ={}&((g'(u)-g'(w_h))v,w)_{\Omega_h}\\
 &-r_h^B(v;w).
\end{aligned}
\end{equation}
Set \(\ell=\min\{\sigma,k\}\), \(c_h=g'(w_h)\), and
\(\rho=(g'(u)-g'(w_h))v\). Define
\[
y:=\partial_tw_h-v=\delta_h+\eta_v,
\qquad \delta_h:=\partial_tw_h-\Pi_hv\in W_h.
\]
Since \(|c_h|\le L_g\), the discrete G{\aa}rding inequality gives
\[
a_{h,\mu}^B(\delta_h,\delta_h)+(c_h\delta_h,\delta_h)_{\Omega_h}
\ge C_{\partial t}\|\delta_h\|_{h,\mu,B}^2,
\qquad
C_{\partial t}:=\min\left\{c_{\mathrm G},
\frac{\mu-C_{\mathrm G}-L_g}{\mu}\right\}>0.
\]
Taking \(w=\delta_h\) in \eqref{eq:differentiated-projection-equation}
therefore yields
\[
\begin{aligned}
C_{\partial t}\|\delta_h\|_{h,\mu,B}^2
\le{}& |(\rho,\delta_h)_{\Omega_h}|+|r_h^B(v;\delta_h)|\\
&+|a_{h,\mu}^B(\eta_v,\delta_h)|
  +|(c_h\eta_v,\delta_h)_{\Omega_h}|.
\end{aligned}
\]
By the Lipschitz continuity of \(g'\), the Sobolev embedding for \(v\),
and \eqref{eq:nonlinear-projection-estimates},
\[
\|\rho\|_{0,\Omega_h}
\le L_{g'}\|v\|_{L^\infty(\Omega)}\|u-w_h\|_{0,\Omega_h}
\le Ch^{\ell+1}.
\]
Using \eqref{eq:proj_error}, \eqref{eq:projection-defect-bound}, and
\eqref{eq:residual_inequality} in the preceding inequality gives
\(\|\delta_h\|_{h,\mu,B}\le Ch^\ell\). Hence
\[
\|y\|_{h,B}\le \|\delta_h\|_{h,B}+\|\eta_v\|_{h,B}
\le Ch^\ell.
\]
For the \(L^2\) estimate, let \(z\) solve
\(-\nabla\cdot(\kappa\nabla z)+(\mu+c_h)z=y\) with the homogeneous
adjoint boundary conditions. Then \(\|z\|_{2,\Omega}\le C_S\|y\|_{0,\Omega_h}\).
By \eqref{eq:adjoint-consistency-identity}, testing
\eqref{eq:differentiated-projection-equation} with \(\Pi_hz\in W_h\),
and using \(r_h^B(v;z)=0\), we obtain
\[
\begin{aligned}
\|y\|_{0,\Omega_h}^2
={}&a_{h,\mu}^B(y,z-\Pi_hz)+(c_hy,z-\Pi_hz)_{\Omega_h}\\
&+(\rho,\Pi_hz)_{\Omega_h}
-r_h^B(v;\Pi_hz-z)-r_h^B(z;y).
\end{aligned}
\]
Splitting \(y=\delta_h+\eta_v\) as in the derivation of
\eqref{eq:adjoint-approximation-bound}, and using
\eqref{eq:residual-adjoint-test} and \eqref{eq:residual-low-regularity},
gives
\[
\|y\|_{0,\Omega_h}^2
\le C\left(h\|y\|_{h,B}+h^{\ell+1}+\|\rho\|_{0,\Omega_h}\right)
\|z\|_{2,\Omega}
+Ch^2\|y\|_{0,\Omega_h}\|z\|_{2,\Omega}.
\]
The energy bound and the estimate for \(\rho\), followed by absorption
for sufficiently small \(h\), imply
\(\|y\|_{0,\Omega_h}\le Ch^{\ell+1}\). This proves
\eqref{eq:time-derivative-projection-estimates}, with constants depending
on the corresponding norms of \(u\) and \(v\).
\end{proof}

To derive the error estimate for $u$, we combine
\eqref{eq:eval_w_h}, \eqref{eq:ahu}, and \eqref{eq:ahuh}.  Expanding the
shifted form in \eqref{eq:eval_w_h} gives
\[
a_h^B(u-w_h,w)+(g(u)-g(w_h),w)_{\Omega_h}
=r_h^B(u;w)-\mu(u-w_h,w)_{\Omega_h}.
\]
This identity cancels the consistency residual in \eqref{eq:ahu}.  Applying
also the identity
$(\partial_{tt}\phi, w) = \frac{\mathrm{d}}{\mathrm{d}t}(\partial_t\phi, w) 
- (\partial_t\phi, \partial_t w)$ to obtain the error identity
\begin{align}
    \label{eq:erru}
&-\bigl(\partial_t(u_h-w_h), \partial_t w\bigr)_{\Omega_h}
+ a_h^B(u_h-w_h, w)
+ (g(u_h)-g(w_h), w)_{\Omega_h}\\
= &\frac{\mathrm{d}}{\mathrm{d}t}\bigl(\partial_t(u-u_h), w\bigr)_{\Omega_h}
- \bigl(\partial_t(u-w_h), \partial_t w\bigr)_{\Omega_h}
-\mu(u-w_h,w)_{\Omega_h}. \notag
\end{align}
For fixed $\tau \in J$, let $\hat{v}(\cdot,t) = \int_t^\tau (u_h-w_h)(\cdot,s)
\,\mathrm{d}s$, so that $\hat{v}(\cdot,\tau)=0$ and 
$\partial_t\hat{v} = -(u_h-w_h)$. Choosing $w=\hat{v}$ in
\eqref{eq:erru},
using the symmetry of $a_h^B$, and integrating over $(0,\tau)$ gives
\begin{equation}
\label{eq:displacement-integrated-identity}
\begin{aligned}
    &\|(u_h-w_h)(\cdot,\tau)\|_{0,\Omega_h}^2
    +a_h^B(\hat{v}(\cdot,0),\hat{v}(\cdot,0))
    +2((v-v_h)(\cdot,0),\hat{v}(\cdot,0))_{\Omega_h}\\
    ={}&2\int_0^\tau \left[
      (\partial_t(u-w_h),u_h-w_h)_{\Omega_h}
      +(g(w_h)-g(u_h),\hat{v})_{\Omega_h}
      -\mu(u-w_h,\hat v)_{\Omega_h}
    \right]\mathrm{d}t\\
    &+\|(u_h-w_h)(\cdot,0)\|_{0,\Omega_h}^2.
\end{aligned}
\end{equation}
By \eqref{eq:garding},
\[
a_h^B(\hat v(\cdot,0),\hat v(\cdot,0))
\geq c_{\mathrm G}\|\hat v(\cdot,0)\|_{h,B}^2
-C_{\mathrm G}\|\hat v(\cdot,0)\|_{0,\Omega_h}^2,
\]
and Cauchy--Schwarz gives
\[
\|\hat v(\cdot,0)\|_{0,\Omega_h}^2
\leq T_{\mathrm f}\int_0^\tau
\|u_h-w_h\|_{0,\Omega_h}^2\,\mathrm dt.
\]
Using these bounds and the $L^2$-orthogonality of the initial projection
$v_h(\cdot,0)=\Pi_h v_0$ in
\eqref{eq:displacement-integrated-identity} yields
\begin{equation}
\label{eq:int_erru_ineq}
\begin{aligned}
&\|(u_h-w_h)(\cdot,\tau)\|_{0,\Omega_h}^2
\leq \|(u_h-w_h)(\cdot,0)\|_{0,\Omega_h}^2\\
&\quad
+ 2\int_0^\tau \bigl(\partial_t(u-w_h), u_h-w_h\bigr)_{\Omega_h}\,\mathrm{d}t\\
&\quad
+2\int_0^\tau (g(w_h)-g(u_h), \hat{v})_{\Omega_h}\,\mathrm{d}t
-2\mu\int_0^\tau(u-w_h,\hat v)_{\Omega_h}\,\mathrm{d}t\\
&\quad
+C_{\mathrm G}T_{\mathrm f}\int_0^\tau
\|u_h-w_h\|_{0,\Omega_h}^2\,\mathrm dt.
\end{aligned}
\end{equation}
The first integral on the right-hand side of
\eqref{eq:int_erru_ineq} satisfies
\begin{equation}
\label{eq:displacement-time-projection-bound}
\begin{aligned}
&2\int_0^\tau
\bigl(\partial_t(u-w_h), u_h-w_h\bigr)_{\Omega_h}\,\mathrm{d}t \\
&\quad\leq
\int_0^\tau \|\partial_t(u-w_h)\|_{0,\Omega_h}^2\,\mathrm{d}t
+ \int_0^\tau \|u_h-w_h\|_{0,\Omega_h}^2\,\mathrm{d}t .
\end{aligned}
\end{equation}
For the nonlinear integral, define
$\hat{f}(t) = \|u_h(\cdot,t)-w_h(\cdot,t)\|_{0,\Omega_h}$.  The
Lipschitz continuity of $g$ gives
\begin{equation}
\label{eq:displacement-nonlinear-bound}
\begin{aligned}
&2\left|\int_0^\tau
(g(w_h)-g(u_h),\hat v)_{\Omega_h}\,\mathrm dt\right|\\
&\quad\leq
2L_g\int_0^\tau \hat{f}(t)
\int_t^\tau \hat{f}(s)\,\mathrm{d}s\,\mathrm{d}t \\
&\quad=
L_g\left(\int_0^\tau \hat{f}(s)\,\mathrm{d}s\right)^2
\leq L_g\tau\int_0^\tau \hat{f}(s)^2\,\mathrm{d}s .
\end{aligned}
\end{equation}
Here we used the symmetry identity
$\int_0^\tau \hat{f}(t)\int_t^\tau \hat{f}(s)\,\mathrm{d}s\,\mathrm{d}t 
= \frac{1}{2}\left(\int_0^\tau \hat{f}(s)\,\mathrm{d}s\right)^2$
and the Cauchy--Schwarz inequality.  The shift term satisfies
\begin{equation}
\label{eq:displacement-shift-bound}
2\mu\left|\int_0^\tau(u-w_h,\hat v)_{\Omega_h}\,\mathrm{d}t\right|
\leq
\mu T_{\mathrm f}\int_0^\tau\|u-w_h\|_{0,\Omega_h}^2\,\mathrm{d}t
+\mu T_{\mathrm f}\int_0^\tau
\|u_h-w_h\|_{0,\Omega_h}^2\,\mathrm{d}t .
\end{equation}
Combining \eqref{eq:displacement-time-projection-bound}--%
\eqref{eq:displacement-shift-bound} gives
\begin{equation}
\label{eq:displacement-gronwall-inequality}
\|(u_h-w_h)(\cdot,\tau)\|_{0,\Omega_h}^2 \leq C_0 
+ (1+L_g\tau+\mu T_{\mathrm f}+C_{\mathrm G}T_{\mathrm f})
\int_0^\tau \|u_h(\cdot,s)-w_h(\cdot,s)\|_{0,\Omega_h}^2\,\mathrm{d}s,
\end{equation}
where $C_0 = \|(u_h-w_h)(\cdot,0)\|_{0,\Omega_h}^2 
+ \int_0^{T_{\mathrm{f}}}\|\partial_t(u-w_h)\|_{0,\Omega_h}^2\,\mathrm{d}t
+\mu T_{\mathrm f}\int_0^{T_{\mathrm f}}
\|u-w_h\|_{0,\Omega_h}^2\,\mathrm{d}t$.
Applying Gr\"{o}nwall's inequality to
\eqref{eq:displacement-gronwall-inequality} yields
\begin{equation}
\label{eq:displacement-discrete-bound}
    \|u_h-w_h\|_{L^\infty(J;L^2(\Omega_h))}^2
\le C_0\exp\!\left((1+L_gT_{\mathrm f}
+\mu T_{\mathrm f}+C_{\mathrm G}T_{\mathrm f})
T_{\mathrm f}\right).
\end{equation}
The approximation estimates
\eqref{eq:proj_error}, \eqref{eq:nonlinear-projection-estimates}, and
\eqref{eq:time-derivative-projection-estimates} imply
$C_0\leq Ch^{2\min\{\sigma,k\}+2}$.  Therefore,
\begin{equation}
\label{eq:displacement-final-error}
\|e_u\|_{L^\infty(J;L^2(\Omega_h))}
\leq C_U h^{\min\{\sigma,k\}+1}.
\end{equation}
\subsection{Error estimate of $v$}
Substituting \eqref{eq:ahu} and \eqref{eq:ahuh} into
the velocity error equation and simplifying gives
\begin{equation}
\label{eq:velocity-error-equation}
(\partial_t\xi_v, w)_{\Omega_h} + a_h^B(\xi_u, w)
= (g(u)-g(u_h), w)_{\Omega_h} - (\partial_t\eta_v, w)_{\Omega_h} 
- a_h^B(\eta_u, w) - r_h^B(u; w).
\end{equation}
Choosing $w = \xi_v$, using the symmetry of $a_h^B$, and integrating over
$(0,\tau)$ with initial projections $\xi_u(\cdot,0) = \xi_v(\cdot,0) = 0$ yields
\begin{equation}
\label{eq:velocity-energy-identity}
\begin{aligned}
&\frac{1}{2}\|\xi_v(\cdot,\tau)\|_{0,\Omega_h}^2
+ \frac{1}{2}a_h^B(\xi_u(\cdot,\tau),\xi_u(\cdot,\tau))\\
&\quad= \int_0^\tau (g(u)-g(u_h),\xi_v)_{\Omega_h}\,\mathrm{d}s
- \int_0^\tau (\partial_t\eta_v,\xi_v)_{\Omega_h}\,\mathrm{d}s\\
&\qquad
- \int_0^\tau a_h^B(\eta_u,\xi_v)\,\mathrm{d}s
- \int_0^\tau r_h^B(u;\xi_v)\,\mathrm{d}s.
\end{aligned}
\end{equation}
Since $\partial_t\xi_u=\xi_v$, integration by parts in the last two
terms of \eqref{eq:velocity-energy-identity} gives the following
identity.  Here
$\partial_t\eta_u=\eta_v$ and
$\partial_t r_h^B(u;\cdot)=r_h^B(v;\cdot)$:
\begin{equation}
\label{eq:velocity-integration-by-parts}
\begin{aligned}
&-\int_0^\tau a_h^B(\eta_u,\xi_v)\,\mathrm ds
-\int_0^\tau r_h^B(u;\xi_v)\,\mathrm ds\\
&\quad =
-\bigl[a_h^B(\eta_u,\xi_u)+r_h^B(u;\xi_u)\bigr]_{s=0}^{s=\tau}\\
&\qquad
+\int_0^\tau a_h^B(\eta_v,\xi_u)\,\mathrm ds
+\int_0^\tau r_h^B(v;\xi_u)\,\mathrm ds .
\end{aligned}
\end{equation}
The lower-regularity projection estimate
\eqref{eq:L2-proj-lower-regularity} gives
\begin{equation}
\label{eq:time-velocity-projection-error}
\|\partial_t\eta_v\|_{L^1(J;L^2(\Omega_h))}
\leq C_{\mathrm{proj}}h^{\min\{k+1,\sigma\}}
\|\partial_tv\|_{L^1(J;H^\sigma(\Omega))}.
\end{equation}
Applying the G{\aa}rding inequality, the Lipschitz continuity of $g$,
Young's inequality, the projection-defect estimate
\eqref{eq:projection-defect-bound}, and
\eqref{eq:residual_inequality} to
\eqref{eq:velocity-energy-identity}--%
\eqref{eq:velocity-integration-by-parts}, and using
\eqref{eq:projection_energy_estimate} and
\eqref{eq:time-velocity-projection-error} gives the first inequality
below.  Moreover,
\[
\|\xi_u\|_{L^\infty(J;L^2(\Omega_h))}
\leq \|e_u\|_{L^\infty(J;L^2(\Omega_h))}
+\|\eta_u\|_{L^\infty(J;L^2(\Omega_h))}
\leq Ch^{\min\{\sigma,k\}+1}
\]
by \eqref{eq:displacement-final-error} and \eqref{eq:proj_error}; hence
the G{\aa}rding defect is of higher order.  Consequently,
\begin{equation}
\label{eq:velocity-energy-estimate}
\begin{aligned}
&\frac{1}{2}\|\xi_v(\cdot,\tau)\|_{0,\Omega_h}^2 
+ \frac{1}{2}c_{\mathrm G}\|\xi_u(\cdot,\tau)\|_{h,B}^2 \\
&\quad\leq
T_{\mathrm{f}}L_g\|e_u\|_{L^\infty(J;L^2(\Omega_h))}\|\xi_v\|_{L^\infty(J;L^2(\Omega_h))}
+ \|\partial_t\eta_v\|_{L^1(J;L^2(\Omega_h))}\|\xi_v\|_{L^\infty(J;L^2(\Omega_h))}\\
&\qquad+ (2C_{\mathrm{cont}}^{\Pi} + 2C_R)h^{\min\{\sigma,k\}}
\|u\|_{L^\infty(J;H^{1+\sigma}(\Omega_h))}
\|\xi_u\|_{L^\infty(J;W(h))}\\
&\qquad+ (C_{\mathrm{cont}}^{\Pi}+C_R)T_{\mathrm{f}}h^{\min\{\sigma,k\}}
\|v\|_{L^\infty(J;H^{1+\sigma}(\Omega_h))}
\|\xi_u\|_{L^\infty(J;W(h))}\\
&\qquad+\frac{1}{2}C_{\mathrm G}
\|\xi_u(\cdot,\tau)\|_{0,\Omega_h}^2\\
&\quad\leq \frac{1}{8}\|\xi_v\|_{L^\infty(J;L^2(\Omega_h))}^2 
+ \frac{1}{8}c_{\mathrm G}\|\xi_u\|_{L^\infty(J;W(h))}^2
+ Ch^{2\min\{\sigma,k\}}.
\end{aligned}
\end{equation}
Taking the supremum of each nonnegative term on the left separately,
adding the resulting inequalities, and absorbing gives
\[
\frac14\left(
\|\xi_v\|_{L^\infty(J;L^2(\Omega_h))}^2
+c_{\mathrm G}\|\xi_u\|_{L^\infty(J;W(h))}^2
\right)\le Ch^{2\min\{\sigma,k\}}.
\]
Hence
\begin{equation}
\label{eq:discrete-velocity-error}
\|\xi_v\|_{L^\infty(J;L^2(\Omega_h))}^2 
\leq Ch^{2\min\{\sigma,k\}},
\end{equation}
where $C$ depends on $C_U$, $T_{\mathrm{f}}$, $L_g$, and the solution
regularity. Combining \eqref{eq:discrete-velocity-error} with
\eqref{eq:velocity-error-decomposition} and \eqref{eq:proj_error} yields
\begin{equation}
\label{eq:velocity-final-error}
\|e_v\|_{L^\infty(J;L^2(\Omega_h))}
\leq C_V h^{\min\{\sigma,k\}}.
\end{equation}

This completes the proof.
\section{Numerical experiments}
\label{sec:num_experiments}
\setcounter{topnumber}{3}
\setcounter{bottomnumber}{2}
\setcounter{totalnumber}{5}
\renewcommand{\topfraction}{0.85}
\renewcommand{\bottomfraction}{0.8}
\renewcommand{\textfraction}{0.1}
\renewcommand{\floatpagefraction}{0.8}
In this section, we present numerical experiments for linear and semi-linear wave equations with two objectives: to verify the theoretical convergence rates, and to demonstrate the advantages of the symplectic Hamiltonian DDG method through systematic comparison of different spatial and temporal discretizations. We also include a two-dimensional Josephson transmission-line test to examine whether the method preserves the qualitative soliton dynamics of reflection and cloning. The time step is chosen as $\Delta t = \mathrm{CFL}\cdot h$.  The sufficient admissible range in Lemma~\ref{lemma:ddg-param-range} is not imposed as a necessary tuning rule in the computations; the practical parameter choices can be less restrictive while still displaying the predicted convergence behavior.  The flux parameters $(\beta_0, \beta_1)$ are listed in Table~\ref{tab:DDG-params}. Unless otherwise stated, the CFL numbers in Table~\ref{tab:cfl-values} are used.
\begin{table}[htbp]
\footnotesize
\centering
\caption{Parameters of (S)DDG methods.}
\label{tab:DDG-params}
\begin{tabular}{c|c|c}
\hline
$k$ &  $\beta_0$ & $\beta_1$ \\
\hline
$2$ & $4.5$ & $-\frac{1}{10}$ \\
[3pt]
$3$ & $9$ &  $-\frac{1}{20}$ \\
[3pt]
$4$ & $13$& $-\frac{1}{40}$ \\
\hline
\end{tabular}
\end{table}
\begin{table}[htbp]
\centering
\caption{CFL numbers for the linear and semi-linear propagation equations.}
\label{tab:cfl-values}
\footnotesize
\begin{tabular}{c|ccc}
\hline
$k$ & $2$ & $3$ & $4$ \\
\hline
CFL for Explicit Time Integrators &$0.1$ & $0.05$ & $0.025$ \\
\hline
CFL for Implicit Time Integrators & $0.5$ & $0.25$ & $0.15$ \\
\hline
\end{tabular}
\end{table}
We measure the errors $\mathrm{err}_{L^2}(u_h) := \|u-u_h\|_{L^\infty(\bar{J};L^2(\Omega_h))}$ 
and $\mathrm{err}_{L^\infty}(u_h) := \|u-u_h\|_{L^\infty(\bar{J};L^\infty(\Omega_h))}$, 
with analogous definitions for $v_h$.
In addition, the evolution of the relative discrete energy error is measured by $
\text{err}(\mathcal{E}_h)
:= \frac{|\mathcal{E}_h-H|}{H}
$, where $H$ denotes the continuous energy evaluated using the exact solution.
For the spatial discretization, we compare the SDDG method with the standard DDG method. For time integration, we compare four Runge-Kutta schemes: ESPRK, explicit Runge-Kutta method (ERK) \cite{butcher2016numerical}, SDIRK, and diagonally implicit Runge-Kutta method (DIRK) \cite{RogerAlexander}. The coefficients of ERK and DIRK are given in Appendix~\ref{app:comparison-coefficients}. For test cases with known analytical solutions, the initial conditions are obtained directly from the $L^2-$projection of the exact solutions. For the Josephson transmission-line test, the prescribed kink initial data in \eqref{eq:josephson-ic} are projected onto the DG space.
\subsection{Linear wave propagation}
In the subsection, we consider the linear wave equation $\partial_{tt} u - \Delta u = 0,\boldsymbol{x}\in\Omega$.
\begin{example}[1D Standing wave]\label{example:harmonic} Consider the following exact solution,
\begin{equation*}
    u(x,t) = \frac{1}{\pi}\sin\big(\pi x\big)\cos\big(\pi t\big),
     \quad \Omega=[-1,1].
\end{equation*}
\end{example}
\begin{example}[1D Travelling wave]\label{example:travelling}
    Consider the following exact solution,
    \begin{equation*}
        u(x,t)=\frac{1}{m\pi}\sin(m\pi(x-t)), \quad  m\in2\mathbb{N}^+,\quad \Omega=[0,1].
    \end{equation*}
\end{example}
\begin{example}[1D Pulse propagation]\label{example:pulse}
Consider an initial pulse of the form
\begin{equation*}
    u(x,0) = \sum_{j\in\mathbb Z}\chi(\frac{x-{x_0}-j}{l}), \quad \chi(\tilde{x}) = \begin{cases} (2\tilde{x}-1)^{10}(2\tilde{x}+1)^{10}, & |\tilde{x}| < 0.5, \\ 0, & \text{otherwise.} \end{cases}\quad \Omega=[0,1].
\end{equation*}
With periodic boundary conditions, the exact solution is
\begin{equation*}
    u(x,t) = \sum_{j\in\mathbb Z}\chi(\frac{x-{x_0}-t-j}{l}),\quad\Omega=[0,1].
\end{equation*}
\end{example}
\begin{example}[2D Travelling wave]\label{example:2dtravelling} Consider the following exact solution,
\begin{equation*}
    u(x,y,t)= \sin{\left(2m\pi x+2n\pi y-\omega t\right)},\ \ m,n\in\mathbb{N}^+,\Omega=[0,1]^2,J=(0,T_{\mathrm{f}}],
\end{equation*}
where $\omega=2\pi\sqrt{m^2+n^2}$.
\end{example}
\subsubsection{Verification of the convergence properties}
We consider the standing wave solution (Example~\ref{example:harmonic}). We report the maximum errors over one full wave period. Table~\ref{tab:linear-errors-low-freq-ESPRK} reports the $L^2$ and $L^\infty$ errors for $k=2,3,4$
under successive mesh refinements. For all polynomial degrees, $u_h$
 achieves optimal convergence rate and $v_h
 $ achieves suboptimal convergence rate, consistent with the theoretical estimates in Section~\ref{sec:error-analysis}.

\begin{table}[tbp]
\footnotesize
\setlength{\tabcolsep}{5pt} 
\centering
\caption{{Example \ref{example:harmonic}}, $T_{\mathrm{f}}=2.0$. History of convergence of the numerical approximations by the schemes SDDG($k$)-ESPRK($k+2$).}
\label{tab:linear-errors-low-freq-ESPRK}
\begin{tabular}{c|c|cc|cc|cc|cc}
\hline
$k$ & $N$
& \multicolumn{2}{c|}{$u_h$}
& \multicolumn{2}{c|}{$v_h$}
& \multicolumn{2}{c|}{$u_h$}
& \multicolumn{2}{c}{$v_h$}
\\
 & 
 & $\mathrm{err}_{L^2}$ & Order 
 & $\mathrm{err}_{L^2}$ & Order
 & $\mathrm{err}_{L^\infty}$ & Order 
 & $\mathrm{err}_{L^\infty}$ & Order
\\
\hline														
&$	64	$&$	2.29\mathrm{e}-06	$&$	-	$&$	5.80\mathrm{e}-04	$&$	-	$&$	2.98\mathrm{e}-06	$&$	-	$&$	9.08\mathrm{e}-04	$&$	-	$\\
$2$&$	128	$&$	2.86\mathrm{e}-07	$&$	3.00	$&$	1.45\mathrm{e}-04	$&$	2.00	$&$	3.74\mathrm{e}-07	$&$	3.00	$&$	2.27\mathrm{e}-04	$&$	2.00	$\\
&$	256	$&$	3.58\mathrm{e}-08	$&$	3.00	$&$	3.61\mathrm{e}-05	$&$	2.00	$&$	4.68\mathrm{e}-08	$&$	3.00	$&$	5.67\mathrm{e}-05	$&$	2.00	$\\
&$	512	$&$	4.48\mathrm{e}-09	$&$	3.00	$&$	9.03\mathrm{e}-06	$&$	2.00	$&$	5.85\mathrm{e}-09	$&$	3.00	$&$	1.42\mathrm{e}-05	$&$	2.00	$\\
\hline													
&$	32	$&$	2.68\mathrm{e}-07	$&$	-	$&$	1.68\mathrm{e}-05	$&$	-	$&$	3.84\mathrm{e}-07	$&$	-	$&$	2.81\mathrm{e}-05	$&$	-	$\\
$3$&$	64	$&$	1.67\mathrm{e}-08	$&$4.01	$&$	2.00\mathrm{e}-06	$&$	3.07	$&$	2.40\mathrm{e}-08	$&$	4.00	$&$	3.27\mathrm{e}-06	$&$	3.10	$\\
&$	128	$&$	1.05\mathrm{e}-09	$&$	3.99	$&$	2.46\mathrm{e}-07	$&$	3.02	$&$	1.51\mathrm{e}-09	$&$	3.99	$&$	4.03\mathrm{e}-07	$&$	3.02	$\\
&$	256	$&$	6.56\mathrm{e}-11	$&$	4.00	$&$	3.06\mathrm{e}-08	$&$	3.01	$&$	9.44\mathrm{e}-11	$&$	4.00	$&$	5.00\mathrm{e}-08	$&$	3.01	$\\
\hline													
&$	16	$&$	5.96\mathrm{e}-08	$&$	-	$&$	2.22\mathrm{e}-05	$&$	-	$&$	1.05\mathrm{e}-07	$&$	-	$&$	4.37\mathrm{e}-05	$&$	-	$\\
$4$&$	32	$&$	1.93\mathrm{e}-09	$&$	4.95	$&$	1.40\mathrm{e}-06	$&$	3.99	$&$	3.21\mathrm{e}-09	$&$	5.03	$&$	2.75\mathrm{e}-06	$&$	3.99	$\\
&$	64	$&$	6.04\mathrm{e}-11	$&$	4.99	$&$	8.76\mathrm{e}-08	$&$	4.00	$&$	1.05\mathrm{e}-10	$&$	4.93	$&$	1.74\mathrm{e}-07	$&$	3.98	$\\
&$	128	$&$	1.89\mathrm{e}-12	$&$	5.00	$&$	5.48\mathrm{e}-09	$&$	4.00	$&$	3.30\mathrm{e}-12	$&$	4.99	$&$	1.09\mathrm{e}-08	$&$	4.00	$\\
\hline
\end{tabular}
\end{table}
\begin{table}[tbp]
\footnotesize
\setlength{\tabcolsep}{5pt} 
\centering
\caption{{Example \ref{example:2dtravelling}}, $T_{\mathrm{f}}=1.0$. History of convergence of the numerical approximations by the schemes SDDG($k$)-ESPRK($k+2$).}
\label{tab:linear-errors-low-freq-ESPRK-2d}
\begin{tabular}{c|c|cc|cc|cc|cc}
\hline
$k$ & $N$
& \multicolumn{2}{c|}{$u_h$}
& \multicolumn{2}{c|}{$v_h$}
& \multicolumn{2}{c|}{$u_h$}
& \multicolumn{2}{c}{$v_h$}
\\
 & 
 & $\mathrm{err}_{L^2}$ & Order 
 & $\mathrm{err}_{L^2}$ & Order
 & $\mathrm{err}_{L^\infty}$ & Order 
 & $\mathrm{err}_{L^\infty}$ & Order
\\
\hline
&16	&$2.81\mathrm{e-}04$	&-	&2.00$\mathrm{e-}02$	&-	&$9.98\mathrm{e-}04$	&-	&$7.65\mathrm{e-}02$	&-\\
2&32	&$3.02\mathrm{e-}05$	&3.21	&$4.51\mathrm{e-}03$	&2.15	&$1.14\mathrm{e-}04$	&3.13	&$1.57\mathrm{e-}02$	&2.28\\
&64	&$3.40\mathrm{e-}06$	&3.15	&$1.09\mathrm{e-}03$	&2.04	&$1.24\mathrm{e-}05$	&3.21	&$3.82\mathrm{e-}03$	&2.04\\
&128	&$4.16\mathrm{e-}07$	&3.03	&$2.69\mathrm{e-}04$	&2.02	&$1.48\mathrm{e-}06$	&3.06	&$9.34\mathrm{e-}04$	&2.03\\
\hline
&16	&$1.40\mathrm{e-}05$	&-	&$9.36\mathrm{e-}04$	&-	&$6.06\mathrm{e-}05$	&-	&$3.91\mathrm{e-}03$	&-\\
3&32	&$8.03\mathrm{e-}07$	&4.12	&$1.03\mathrm{e-}04$	&3.19	&$3.54\mathrm{e-}06$	&4.09	&$4.62\mathrm{e-}04$	&3.08\\
&64	&$4.96\mathrm{e-}08$	&4.02	&$1.30\mathrm{e-}05$	&2.99	&$2.22\mathrm{e-}07$	&4.00	&$5.60\mathrm{e-}05$	&3.04\\
&128	&$3.01\mathrm{e-}09$	&4.04	&$2.00\mathrm{e-}06$	&2.70	&$1.35\mathrm{e-}08$	&4.04	&$7.00\mathrm{e-}06$	&3.00\\
\hline
&16	&$1.52\mathrm{e-}07$	&-	&$1.01\mathrm{e-}05$	&-	&4.16$\mathrm{e-}07$	&-	&$6.67\mathrm{e-}05$	&-\\
4&32	&$4.16\mathrm{e-}09$	&5.19	&$4.96\mathrm{e-}07$	&4.35	&$1.24\mathrm{e-}08$	&5.07	&$2.60\mathrm{e-}06$	&4.68\\
&64	&$1.20\mathrm{e-}10$	&5.11	&$2.97\mathrm{e-}08$	&4.06	&$3.37\mathrm{e-}10$	&5.20	&$1.05\mathrm{e-}07$	&4.63\\
&128	&$4.46\mathrm{e-}12$	&4.76	&$1.91\mathrm{e-}09$	&3.96	&$1.27\mathrm{e-}11$	&4.73	&$5.49\mathrm{e-}09$	&4.26\\
\hline
\end{tabular}
\end{table}
\subsubsection{Solution profile}
We examine the solution profiles for the travelling wave and pulse propagation (Example~\ref{example:travelling} and \ref{example:pulse}), shown in Figures~\ref{fig:travelling_implicit_compar_sol}--\ref{fig:travelling_explicit_compar_sol} and \ref{fig:pulse_implicit_compar_sol} respectively.

Under implicit integration, the symplectic Hamiltonian DDG method (SDDG-SDIRK) achieves the best overall performance: SDIRK preserves wave amplitudes over long times, while SDDG method consistently reduces dispersion errors more effectively than DDG method (Figures~\ref{fig:travelling_DDG-DIRK} versus \ref{fig:travelling_SDDG-DIRK}, \ref{fig:travelling_DDG-SDIRK} versus \ref{fig:travelling_SDDG-SDIRK}, \ref{fig:pulse_DDG-DIRK} versus \ref{fig:pulse_SDDG-DIRK} and \ref{fig:pulse_DDG-SDIRK} versus \ref{fig:pulse_SDDG-SDIRK}). In contrast, the non-symplectic DIRK method introduces numerical dissipation that causes spurious amplitude decay (Figure~\ref{fig:travelling_SDDG-DIRK}, \ref{fig:pulse_SDDG-DIRK}).

Under explicit integration, SDDG method produces significantly smaller phase errors than DDG method, as seen by comparing Figures~\ref{fig:travelling_DDG-ERK} and \ref{fig:travelling_SDDG-ERK}, \ref{fig:travelling_DDG-ESPRK} and \ref{fig:travelling_SDDG-ESPRK}.  This reflects the dispersion reduction from the discrete Hamiltonian structure. The choice between symplectic and non-symplectic explicit time integrators has no pronounced impact on wave profile accuracy at the same time step. In our tests, the symplectic integrators allow larger CFL numbers without loss of stability.
\begin{figure}[tbp]
  \centering
    \begin{subfigure}[t]{0.48\textwidth}
    \centering
    \includegraphics[width=\textwidth]{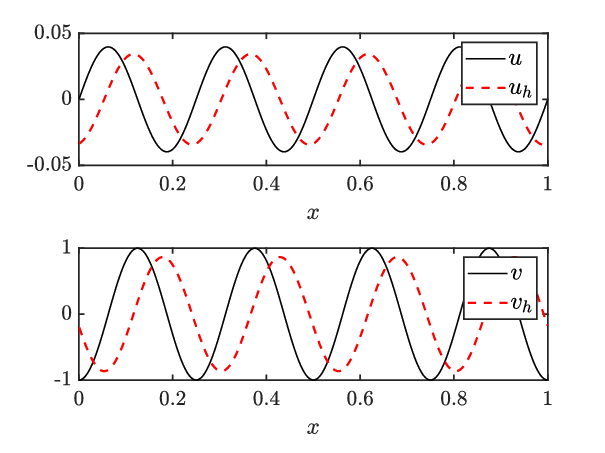}
    \caption{DDG(2)-DIRK(3,4)}
    \label{fig:travelling_DDG-DIRK}
  \end{subfigure}
  \hfill
  \begin{subfigure}[t]{0.48\textwidth}
    \centering
    \includegraphics[width=\textwidth]{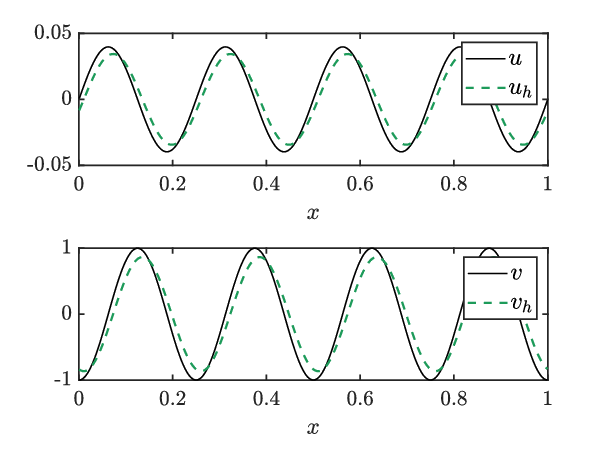}
    \caption{SDDG(2)-DIRK(3,4)}
    \label{fig:travelling_SDDG-DIRK}
  \end{subfigure}\\
  \begin{subfigure}[t]{0.48\textwidth}
    \centering
    \includegraphics[width=\textwidth]{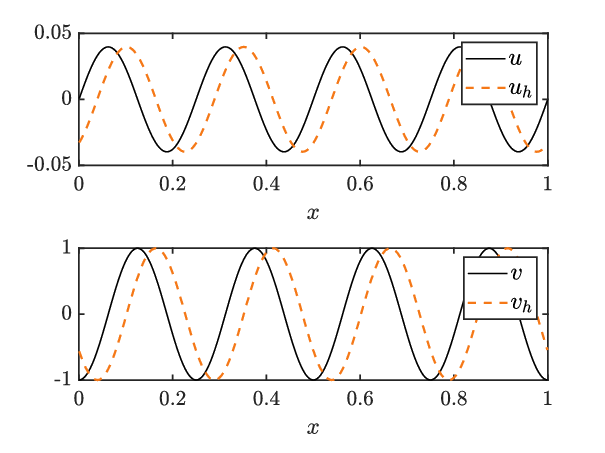}
    \caption{DDG(2)-SDIRK(3,4)}
    \label{fig:travelling_DDG-SDIRK}
  \end{subfigure}
  \hfill
  \begin{subfigure}[t]{0.48\textwidth}
    \centering
    \includegraphics[width=\textwidth]{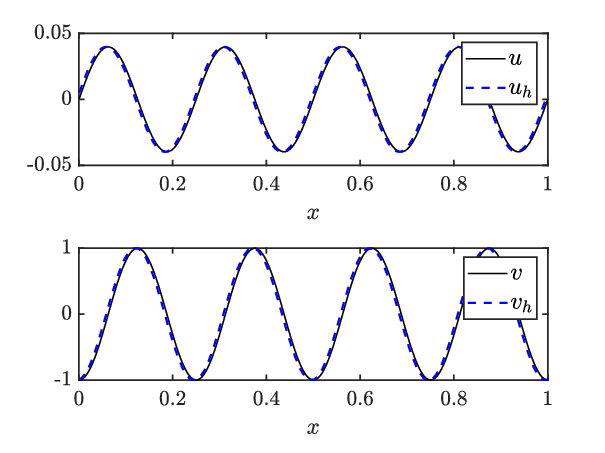}
    \caption{SDDG(2)-SDIRK(3,4)}
    \label{fig:travelling_SDDG-SDIRK}
  \end{subfigure}

\caption{Example~\ref{example:travelling}, $m=8$, $T_{\mathrm{f}}=50$, 
$N=64$, $\mathrm{CFL}=0.5$. Exact solutions and numerical approximations 
by implicit time integrators.
$\mathrm{err}_{L^\infty}(u_h)$: 0.047, 0.011, 0.038, 0.004;
$\mathrm{err}_{L^\infty}(v_h)$: 1.172, 0.275, 0.945, 0.118.}
\label{fig:travelling_implicit_compar_sol}
\end{figure}
\begin{figure}[tbp]
  \centering
  \begin{subfigure}[t]{0.48\textwidth}
    \centering
    \includegraphics[width=\textwidth]{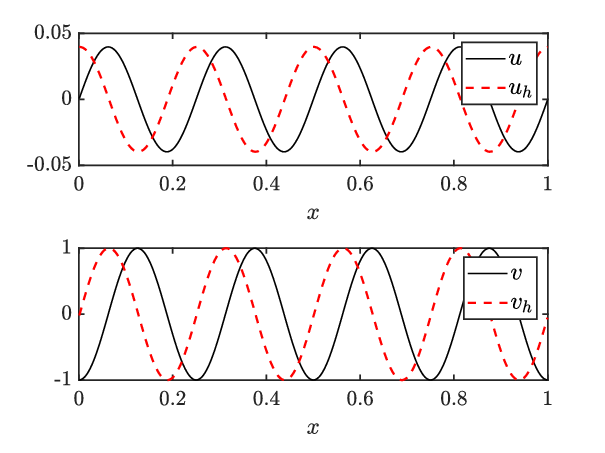}
    \caption{DDG(2)-ERK(4,4)}
    \label{fig:travelling_DDG-ERK}
  \end{subfigure}
  \hfill
  \begin{subfigure}[t]{0.48\textwidth}
    \centering
    \includegraphics[width=\textwidth]{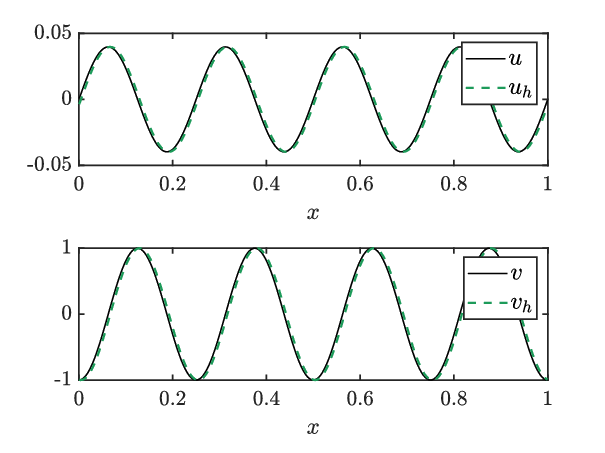}
    \caption{SDDG(2)-ERK(4,4)}
    \label{fig:travelling_SDDG-ERK}
  \end{subfigure}\\
  \begin{subfigure}[t]{0.48\textwidth}
    \centering
    \includegraphics[width=\textwidth]{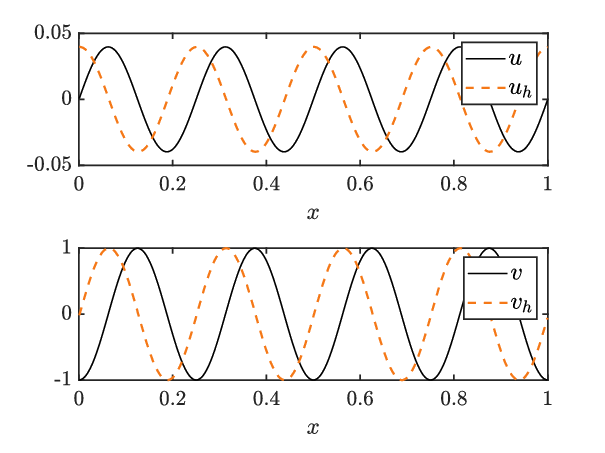}
    \caption{DDG(2)-ESPRK(6,4)}
    \label{fig:travelling_DDG-ESPRK}
  \end{subfigure}
  \hfill
  \begin{subfigure}[t]{0.48\textwidth}
    \centering
    \includegraphics[width=\textwidth]{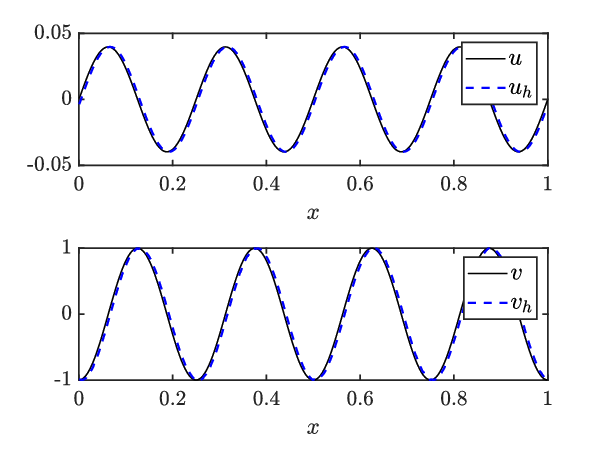}
    \caption{SDDG(2)-ESPRK(6,4)}
    \label{fig:travelling_SDDG-ESPRK}
  \end{subfigure}
  
  \caption{{Example \ref{example:travelling}}, $m=8,T_{\mathrm{f}}=500, N=64, \mathrm{CFL}=0.1$. Exact solutions and numerical approximations by explicit time integrators. $\text{err}_{L^\infty}(u_h)$: 0.080, 0.004, 0.080,0.004 and  $\text{err}_{L^\infty}(v_h)$ : 2.001, 0.109, 2.004,0.120.}
  \label{fig:travelling_explicit_compar_sol}
\end{figure}

\begin{figure}[tbp]
  \centering
 \begin{subfigure}[t]{0.48\textwidth}
    \centering
    \includegraphics[width=\textwidth]{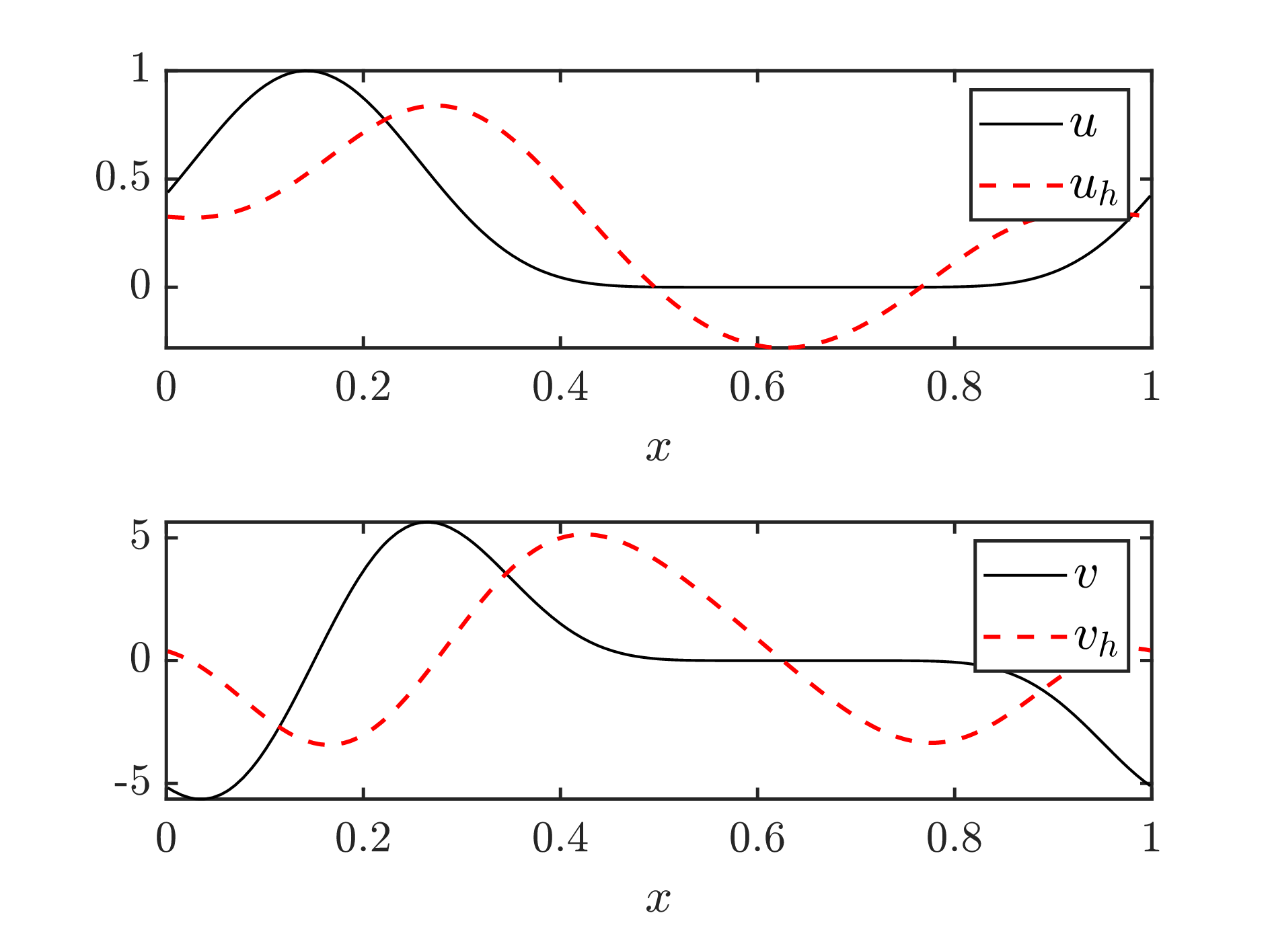}
    \caption{DDG(2)-DIRK(3,4)}
    \label{fig:pulse_DDG-DIRK}
  \end{subfigure}
  \hfill
  \begin{subfigure}[t]{0.48\textwidth}
    \centering
    \includegraphics[width=\textwidth]{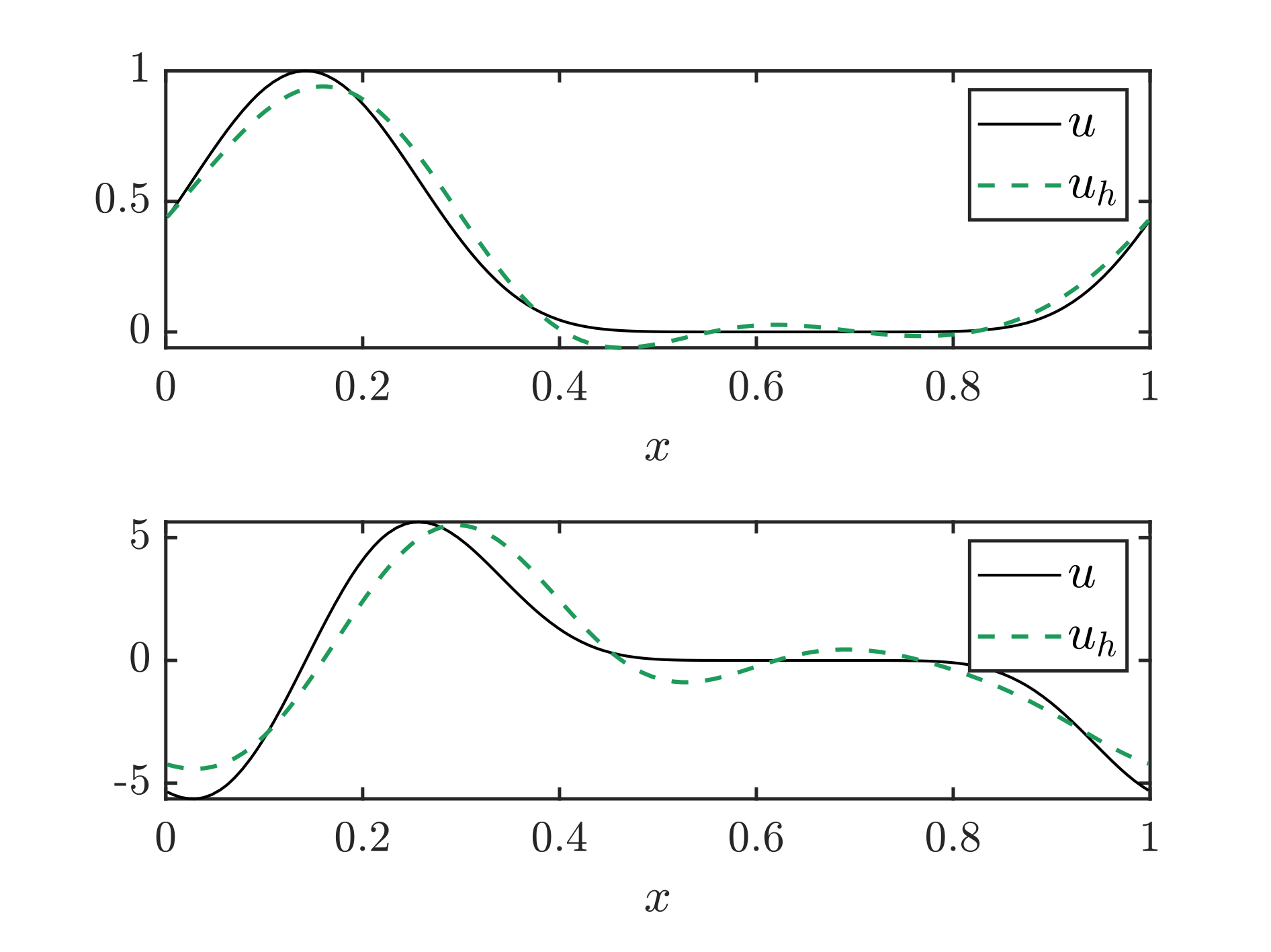}
    \caption{SDDG(2)-DIRK(3,4)}
    \label{fig:pulse_SDDG-DIRK}
  \end{subfigure}\\
  \begin{subfigure}[t]{0.48\textwidth}
    \centering
    \includegraphics[width=\textwidth]{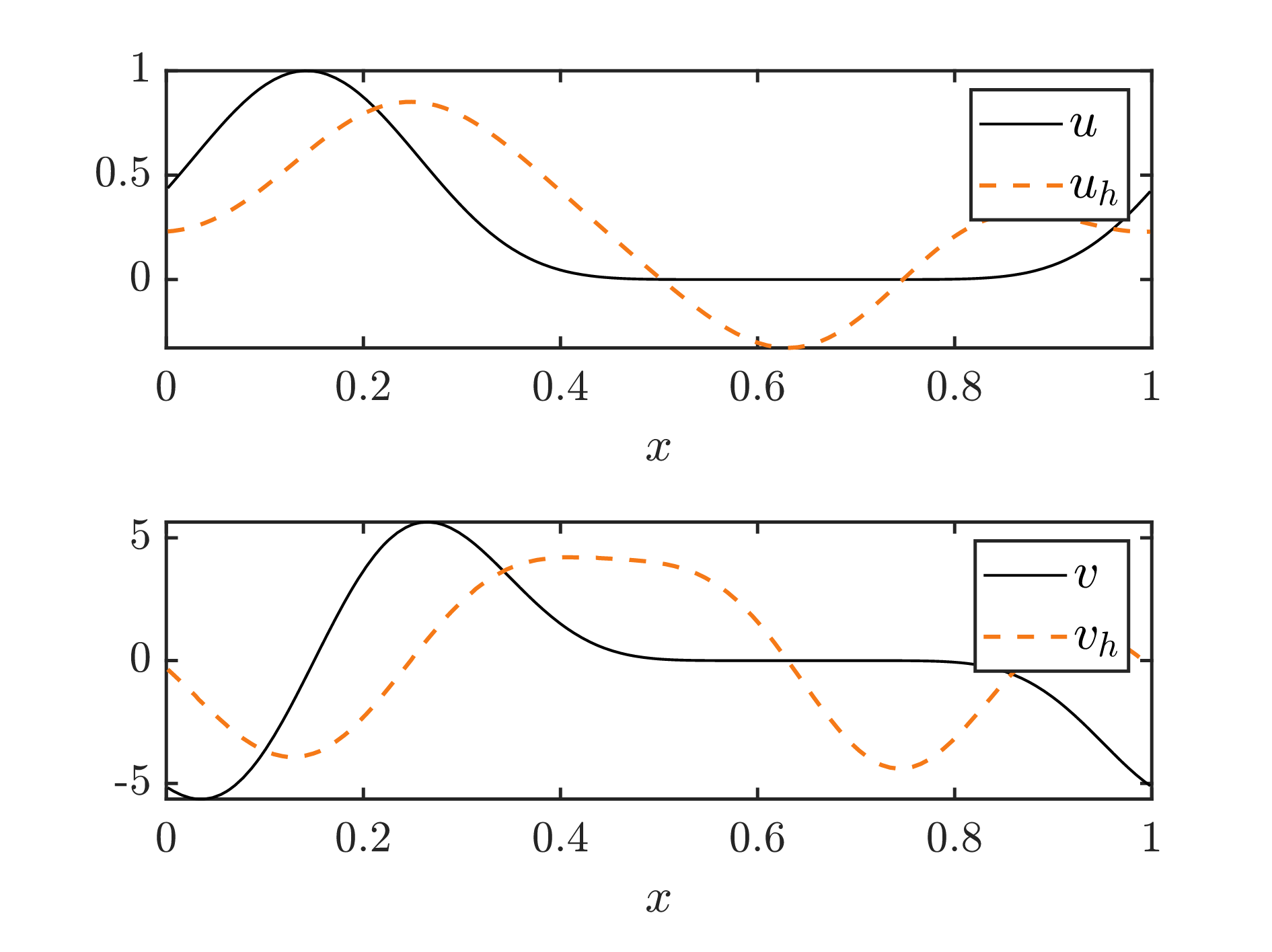}
    \caption{DDG(2)-SDIRK(3,4)}
    \label{fig:pulse_DDG-SDIRK}
  \end{subfigure}
  \hfill
  \begin{subfigure}[t]{0.48\textwidth}
    \centering
    \includegraphics[width=\textwidth]{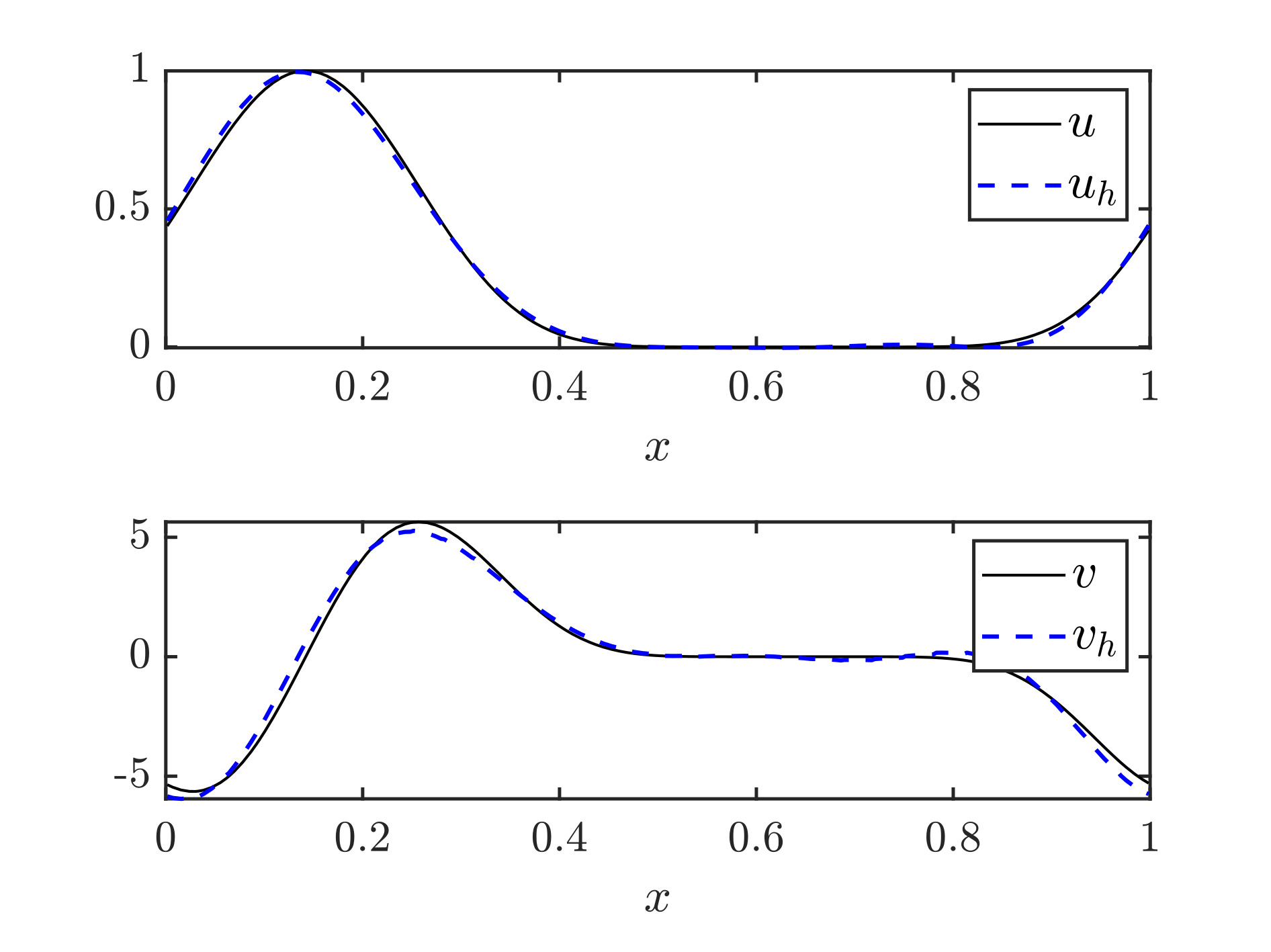}
    \caption{SDDG(2)-SDIRK(3,4)}
    \label{fig:pulse_SDDG-SDIRK}
  \end{subfigure}
\caption{Example~\ref{example:pulse}, $x_0=0.15$, $l=1$, $T_{\mathrm{f}}=200$, $N=64$, $\mathrm{CFL}=0.35$. Exact solutions and numerical approximations by implicit time integrators. $\mathrm{err}_{L^\infty}(u_h)$: 0.542, 0.100, 0.488, 0.033; $\mathrm{err}_{L^\infty}(v_h)$: 7.440, 1.703, 6.466, 0.675.}
  \label{fig:pulse_implicit_compar_sol}
\end{figure}
\subsubsection{Discrete energy evolution}
Figure~\ref{fig:travelling_energy} shows the discrete energy evolution for the travelling wave solution. The implicit symplectic Hamiltonian DDG method (SDDG-SDIRK) delivers the best energy conservation. For symplectic time integrators, the discrete energy stays bounded near the initial projection error, and SDDG method yields smaller energy errors than DDG, as it admits a discrete Hamiltonian structure. For non-symplectic implicit time integrators, artificial dissipation dominates, leading to rapid energy error growth.

For non-symplectic explicit time integrator, dissipation is mild. The SDDG energy error initially decreases before rising, suggesting that dissipation temporarily offsets the projection error before temporal errors dominate. This explains the comparable wave profile accuracy between symplectic and non-symplectic explicit schemes noted in the previous subsection.
\begin{figure}[tbp]
\centering
  \begin{subfigure}[t]{0.48\textwidth}
    \centering
    \includegraphics[width=\textwidth]{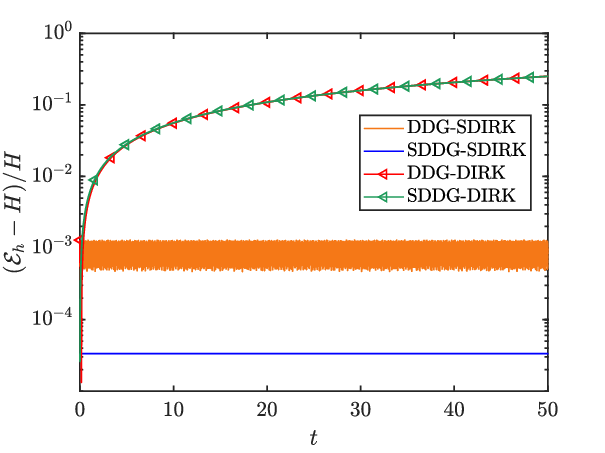}
    \caption{Implicit time integrators (CFL = 0.5)}
    \label{fig:travelling_energy_implicit}
  \end{subfigure}
  \hfill
  \begin{subfigure}[t]{0.48\textwidth}
    \centering
    \includegraphics[width=\textwidth]{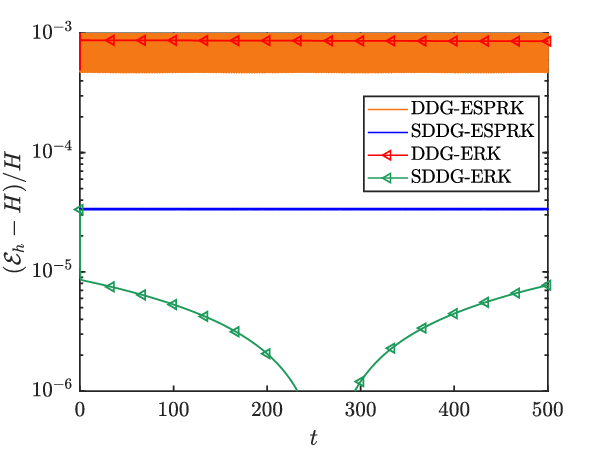}
    \caption{Explicit time integrators (CFL = 0.1)}
    \label{fig:travelling_energy_explicit}
  \end{subfigure}
\caption{Example~\ref{example:travelling}, $m=8$, $k=2$, $N=64$. 
Semi-log plot of relative discrete energy errors versus time for all combinations of fourth-order symplectic/non-symplectic 
time integrators and SDDG/DDG spatial discretizations ($k=2$).}
\label{fig:travelling_energy}
\end{figure}


\subsection{Semi-linear wave propagation}
In this subsection we consider the semi-linear SG (Sine-Gordon) equation $\partial_{tt} u - \Delta u + \sin(u) = 0,
    \quad x\in\Omega,\ t\in (0,T_{\mathrm{f}}]$. The one-dimensional kink and breather tests provide exact-solution benchmarks, while the two-dimensional Josephson transmission-line test examines qualitative soliton dynamics on a T-shaped domain.
\begin{example}[1D Periodic kink]\label{example:kink}
    We consider an infinite periodic train of kinks \cite{novkoski2023numerical}, for which an analytical solution is available and given by
\begin{equation*}
    u(x,t) = \pi + 2\mathrm{am}\!\left(
    \frac{x-ct-x_0}{\sqrt{\mu(1-c^2)}},\mu
    \right), \quad v(x,t) = -\frac{2c\mathrm{dn}}{\sqrt{\mu(1-c^2)}}, \quad \Omega=[0,L],
\end{equation*}
where $\mathrm{am}(\cdot, \mu)$ is the Jacobi amplitude function, for $\mu\in(0,1)$ and $\mathrm{dn}$ is the \emph{delta} amplitude of the elliptic function. The solution has a "periodic" boundary condition with a spatial period $L=2\sqrt{\mu(1-c^2)}\mathrm{K}(\mu)$. It is defined as the distance between two successive periodic kinks, where $\mathrm{K}$ is the complete elliptic function of the first kind. The periodicity is defined as follows:
\begin{align*}
    u(x+L,t)=u(x,t) \mod{2\pi},\quad  v(x+L,t)=v(x,t).
\end{align*}
The prescribed $2\pi$ offset is included in the jump across the paired boundary faces.
\end{example}
\begin{example}[1D Breather]\label{example:breather}
    We consider a classical breather solution \cite{REICH2000473} of the SG equation and the exact solution takes the form
\begin{align*}
    u(x,t) = 4\tan^{-1}\left(
    \frac{A\cos(\omega t)}{\cosh\big(\sqrt{1-\omega^2}x\big)}
    \right),\quad
    v(x,t)
    =  -\frac{4A\omega\,\sin(\omega t)}
        {(1 + A^2 B^2)\cosh(\sqrt{1-\omega^2}x)}.
\end{align*}
where $
    A = \frac{\sqrt{1-\omega^2}}{\omega}, B = \frac{\cos(\omega t)}{\cosh(\sqrt{1-\omega^2}x)}$.
Although the breather solution is periodic on $(-\infty,\infty)$, it is exponentially localized in space.  
Therefore, by choosing a sufficiently large interval $\Omega= \left[-\frac{L}{2},\, \frac{L}{2}\right]$ and imposing periodic boundary conditions, the truncation error introduced at the boundaries is negligible. This makes the breather solution well suited as an exact solution for the present numerical tests.
\end{example}
\begin{example}[Josephson transmission line]\label{ex:josephson}
We consider the T-shaped domain $\Omega$ illustrated in Fig.~\ref{fig:Tdomain}, 
modeling a Josephson transmission line~\cite{gulevich2006flux,sanchez2024symplectic}. 
Homogeneous Neumann boundary conditions are imposed on 
$\Gamma$. 
\begin{figure}[htbp]
\centering
\includegraphics[width=0.5\textwidth]{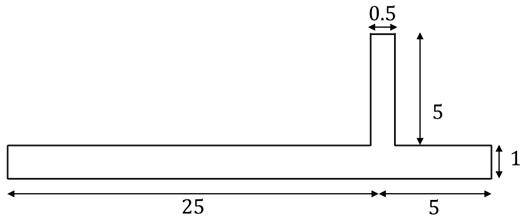}
\caption{T-shaped domain $\Omega$ for the Josephson transmission line.}
\label{fig:Tdomain}
\end{figure}
The initial conditions are
\begin{equation}\label{eq:josephson-ic}
   u_0(x,y) 
   = 4\arctan\!\left(\exp\!\left(\frac{x - x_0}{\sqrt{1 - c^2}}\right)\right),
   v_0(x,y) 
   = -\frac{2c}{\sqrt{1 - c^2}}\,
     \mathrm{sech}\!\left(\frac{x - x_0}{\sqrt{1 - c^2}}\right),
\end{equation}
which describe a one-dimensional kink-shaped soliton, initially 
localized at $x = x_0$ on the horizontal arm and propagating along 
the $x$-direction with velocity $c$. Below a critical velocity $c^* \approx 0.76$, the soliton fails to traverse the T-junction and remains confined to the horizontal branch 
(non-cloning regime); above $c^*$, it splits at the junction into two solitons propagating along the horizontal and vertical branches 
(cloning regime). 
We compute two simulations, with $c = 0.7$ (non-cloning) and $c = 0.8$ (cloning), to verify whether the numerical methods reproduce both 
regimes.
\end{example}
\subsubsection{Verification of the convergence properties}
The breather solution of the SG equation is adopted as the test case. 
We report the maximum errors over one full oscillation period. Table~\ref{tab:nonlinear-error-ESPRK} reports the $L^2$ and $L^\infty$ errors for $k=2,3,4$ under successive mesh refinements. For all polynomial degrees, $u_h$ achieves $(k+1)$-th order convergence and $v_h$ achieves $k$-th order convergence, consistent with the theoretical estimates in Section~\ref{sec:error-analysis}.
\begin{table}[tbp]
\setlength{\tabcolsep}{4pt} 
\footnotesize
\centering
\caption{Example \ref{example:breather}, $\omega=0.9, L=100, T_{\mathrm{f}}=2\pi/\omega$, History of convergence of the numerical approximations by the schemes SDDG($k$)-ESPRK($k+2$).}
\label{tab:nonlinear-error-ESPRK}

\begin{tabular}{c |c | cc | cc | cc | cc }
\hline
$k$ & $N$
& \multicolumn{2}{c|}{$u_h$}
& \multicolumn{2}{c|}{$v_h$}
& \multicolumn{2}{c|}{$u_h$}
& \multicolumn{2}{c}{$v_h$}
\\
 & 
 & $\mathrm{err}_{L^2}$ & Order 
 & $\mathrm{err}_{L^2}$ & Order
 & $\mathrm{err}_{L^\infty}$ & Order 
 & $\mathrm{err}_{L^\infty}$ & Order
\\
\hline											
&$	128	$&$	1.12\mathrm{e}-03	$&$	-	$&$	1.17\mathrm{e}-02	$&$	-	$&$	8.15\mathrm{e}-04	$&$	-	$&$	1.12\mathrm{e}-02	$&$	-	$\\
$2$&$	256	$&$	1.39\mathrm{e}-04	$&$	3.01	$&$	3.09\mathrm{e}-03	$&$	1.93	$&$	1.02\mathrm{e}-04	$&$	3.00	$&$	2.75\mathrm{e}-03	$&$	2.03	$\\
&$	512	$&$	1.78\mathrm{e}-05	$&$	2.97	$&$	7.83\mathrm{e}-04	$&$	1.98	$&$	1.30\mathrm{e}-05	$&$	2.98	$&$	6.82\mathrm{e}-04	$&$	2.01	$\\
&$	1024	$&$	2.30\mathrm{e}-06	$&$	2.95	$&$	1.97\mathrm{e}-04	$&$	1.99	$&$	1.66\mathrm{e}-06	$&$	2.96	$&$	1.71\mathrm{e}-04	$&$	1.99	$\\
\hline																		
&$	128	$&$	5.86\mathrm{e}-05	$&$	-	$&$	5.70\mathrm{e}-04	$&$	-	$&$	5.13\mathrm{e}-05	$&$ -	$&$	7.47\mathrm{e}-04	$&$	-	$\\
$3$&$	256	$&$	3.97\mathrm{e}-06	$&$	3.88	$&$	4.97\mathrm{e}-05	$&$	3.52	$&$	3.90\mathrm{e}-06	$&$	3.72	$&$	6.29\mathrm{e}-05	$&$	3.57	$\\
&$	512	$&$	2.53\mathrm{e}-07	$&$	3.98	$&$	5.15\mathrm{e}-06	$&$	3.27	$&$	2.56\mathrm{e}-07	$&$	3.93	$&$	6.27\mathrm{e}-06	$&$	3.33	$\\
&$	1024	$&$	1.58\mathrm{e}-08	$&$	3.99	$&$	6.04\mathrm{e}-07	$&$	3.09	$&$	1.62\mathrm{e}-08	$&$	3.98	$&$	7.13\mathrm{e}-07	$&$	3.14	$\\
\hline																		
&$	64	$&$	6.99\mathrm{e}-05	$&$	-	$&$	2.27\mathrm{e}-03	$&$	- $&$	6.22\mathrm{e}-05	$&$-	$&$	2.19\mathrm{e}-03	$&$	-$\\
$4$&$	128	$&$	1.90\mathrm{e}-06	$&$	5.20	$&$	1.25\mathrm{e}-04	$&$	4.18	$&$	2.17\mathrm{e}-06	$&$	4.84	$&$	1.49\mathrm{e}-04	$&$	3.88	$\\
&$	256	$&$	6.09\mathrm{e}-08	$&$	4.96	$&$	7.41\mathrm{e}-06	$&$	4.08	$&$	7.08\mathrm{e}-08	$&$	4.94	$&$	9.15\mathrm{e}-06	$&$	4.03	$\\
&$	512	$&$	2.05\mathrm{e}-09	$&$	4.89	$&$	4.76\mathrm{e}-07	$&$	3.96	$&$	2.21\mathrm{e}-09	$&$	5.00	$&$	5.90\mathrm{e}-07	$&$	3.95	$\\
\hline
\end{tabular}
\end{table}
\subsubsection{Solution profile}
For the kink solution, the numerical results are qualitatively similar to those of the linear case: SDDG-SDIRK delivers the best performance, SDIRK preserves wave amplitudes while DIRK introduces significant error, and SDDG reduces dispersion errors compared to DDG. Representative results are shown in Figure~\ref{fig:kink_SDDG_SDIRK} versus \ref{fig:kink_DDG_SDIRK} and \ref{fig:kink_DDG_DIRK}.

For the breather solution, the symplectic Hamiltonian DDG method (SDDG-SDIRK and SDDG-ESPRK) achieves the best solution profiles under both implicit and explicit integration. Under implicit time integration, non-symplectic DIRK leads to rapid profile degradation regardless of spatial scheme (Figure~\ref{fig:breather_DDG_DIRK} and~\ref{fig:breather_SDDG_DIRK}), while SDDG method consistently outperforms DDG method under SDIRK (Figure~\ref{fig:breather_DDG_SDIRK} versus~\ref{fig:breather_SDDG_SDIRK}). Under explicit integration, ESPRK method reduces $\mathrm{err}_{L^\infty}(u_h)$ and $\mathrm{err}_{L^\infty}(v_h)$ by approximately 73\% compared to ERK method (Figure~\ref{fig:breather_SDDG_ERK} versus~\ref{fig:breather_SDDG_ESPRK}), and DDG method suffers severe profile distortion with displacement errors reaching two wave amplitudes (Figure~\ref{fig:breather_DDG_ESPRK} versus~\ref{fig:breather_SDDG_ESPRK}).
\begin{figure}[tbp]
  \centering  
    \begin{subfigure}[t]{0.48\textwidth}
    \centering
    \includegraphics[width=\textwidth]{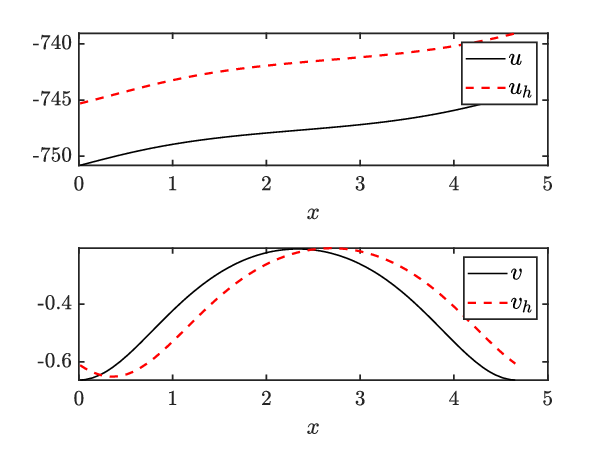}
    \caption{DDG(2)-DIRK(3,4)}
    \label{fig:kink_DDG_DIRK}
  \end{subfigure}
  \hfill
  \begin{subfigure}[t]{0.48\textwidth}
    \centering
    \includegraphics[width=\textwidth]{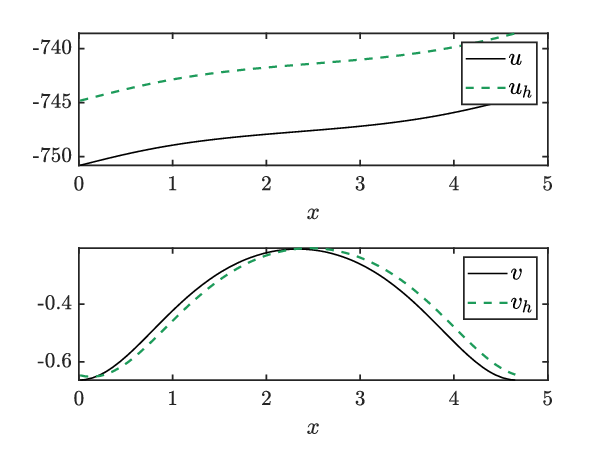}
    \caption{SDDG(2)-DIRK(3,4)}
    \label{fig:kink_SDDG_DIRK}
  \end{subfigure}
\\
  \begin{subfigure}[t]{0.48\textwidth}
    \centering
    \includegraphics[width=\textwidth]{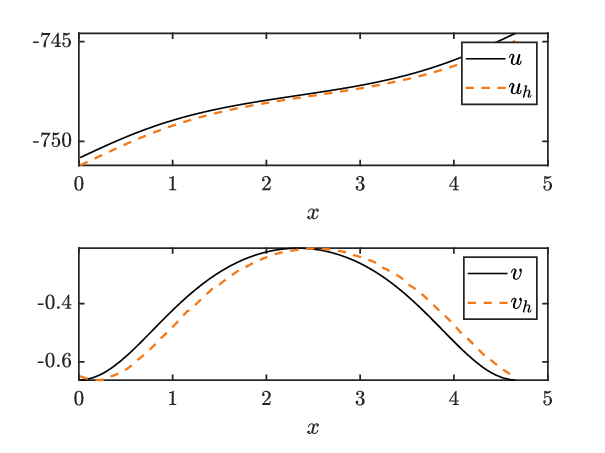}
    \caption{DDG(2)-SDIRK(3,4)}
    \label{fig:kink_DDG_SDIRK}
  \end{subfigure}
  \hfill
  \begin{subfigure}[t]{0.48\textwidth}
    \centering
    \includegraphics[width=\textwidth]{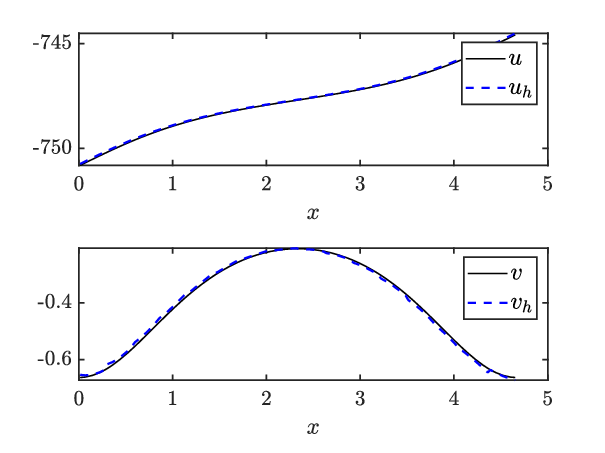}
    \caption{SDDG(2)-SDIRK(3,4)}
    \label{fig:kink_SDDG_SDIRK}
  \end{subfigure}
  \caption{Example~\ref{example:kink}, $c=0.3$, $\mu=0.9$, $T_{\mathrm{f}}=1866.5$, $N=16$, $\mathrm{CFL}=1$. Exact solutions and numerical approximations by implicit time integrators. $\mathrm{err}_{L^\infty}(u_h)$: 6.031, 6.189, 0.389, 0.069; $\mathrm{err}_{L^\infty}(v_h)$: 0.455, 0.455, 0.064, 0.027.}
\label{fig:kink_implicit_compar}
\end{figure}
\begin{figure}[tbp]
  \centering
  \begin{subfigure}[t]{0.48\textwidth}
    \centering
    \includegraphics[width=\textwidth]{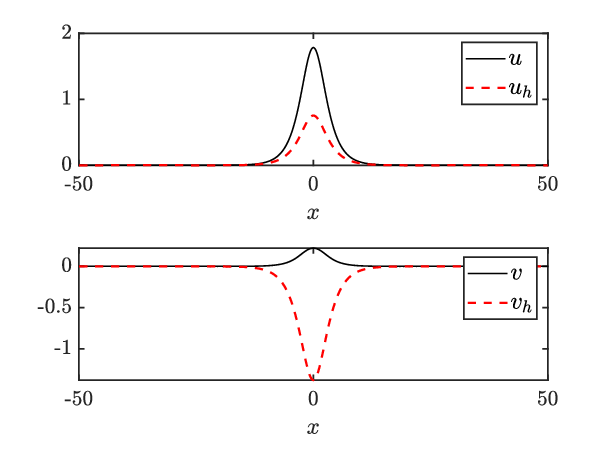}
    \caption{DDG(2)-DIRK(3,4)}
    \label{fig:breather_DDG_DIRK}
  \end{subfigure}
  \hfill
  \begin{subfigure}[t]{0.48\textwidth}
    \centering
    \includegraphics[width=\textwidth]{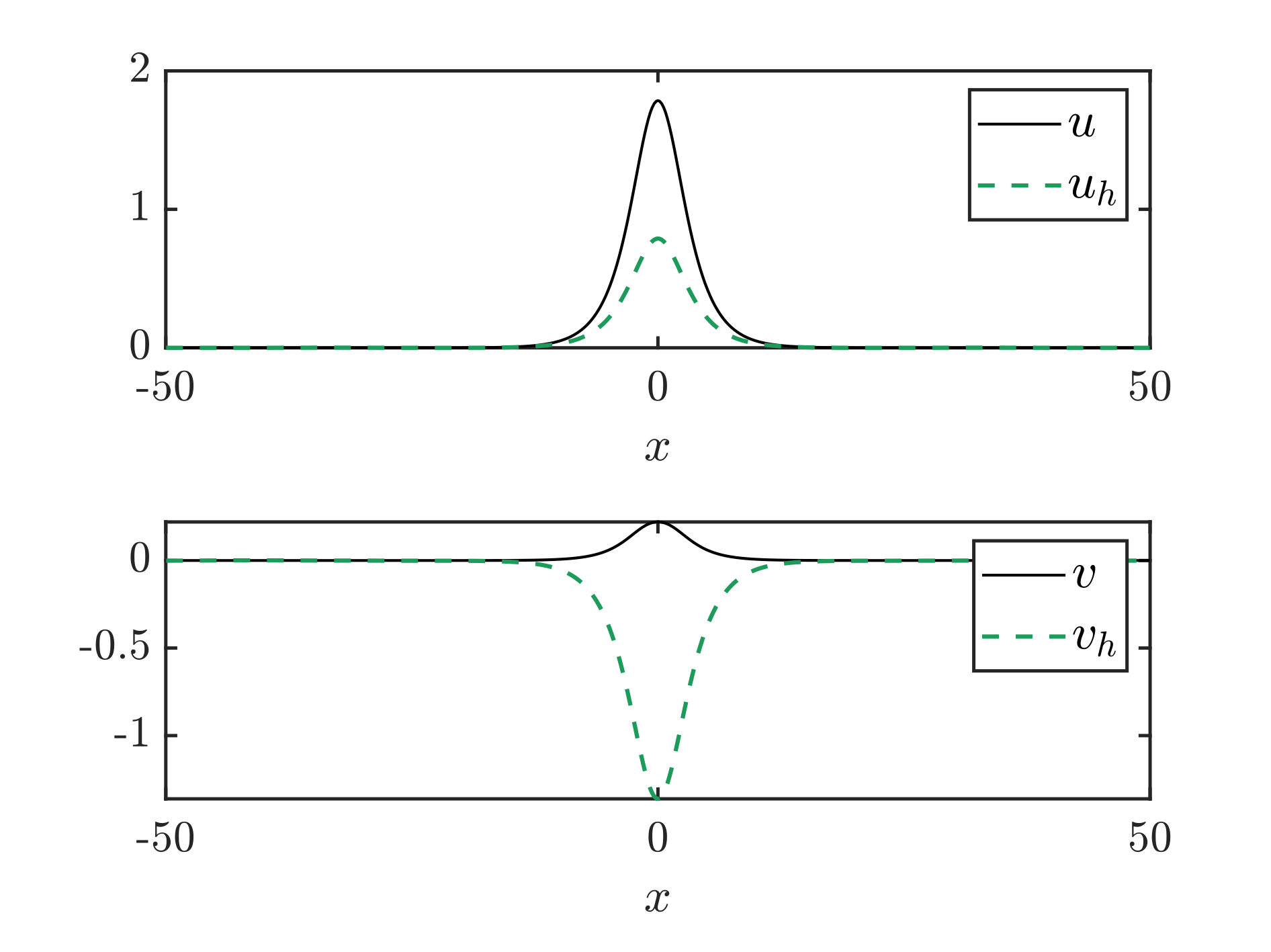}
    \caption{SDDG(2)-DIRK(3,4)}
    \label{fig:breather_SDDG_DIRK}
  \end{subfigure}
\\
  \begin{subfigure}[t]{0.48\textwidth}
    \centering
    \includegraphics[width=\textwidth]{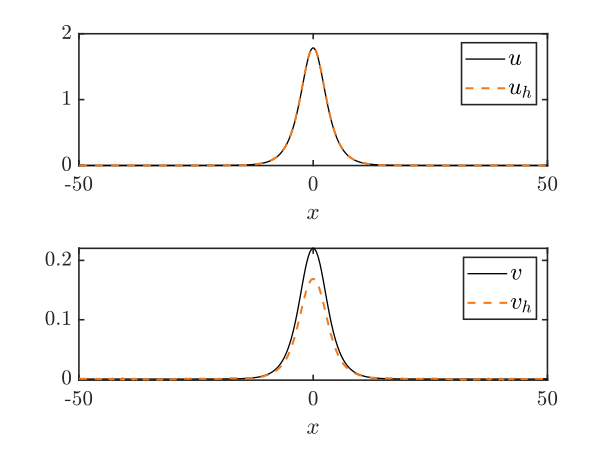}
    \caption{DDG(2)-SDIRK(3,4)}
    \label{fig:breather_DDG_SDIRK}
  \end{subfigure}
  \hfill
  \begin{subfigure}[t]{0.48\textwidth}
    \centering
    \includegraphics[width=\textwidth]{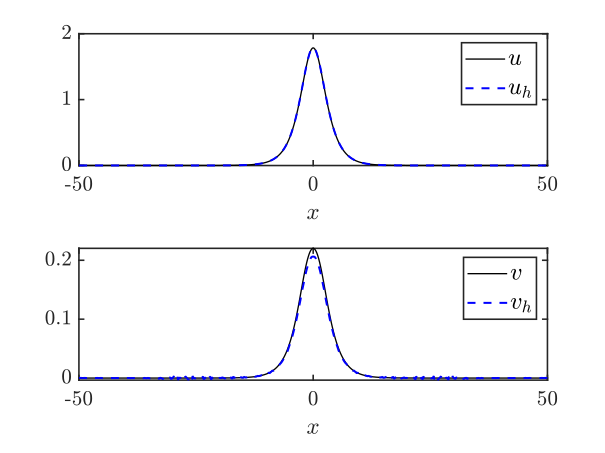}
    \caption{SDDG(2)-SDIRK(3,4)}
    \label{fig:breather_SDDG_SDIRK}
  \end{subfigure}
\caption{Example~\ref{example:breather}, $\omega=0.9$, $L=100$, $T_{\mathrm{f}}=139.6$, 
$N=128$, $\mathrm{CFL}=0.35$. Exact solutions and numerical approximations 
by implicit time integrators. 
$\mathrm{err}_{L^\infty}(u_h)$: 2.064, 2.036, 0.069, 0.018; 
$\mathrm{err}_{L^\infty}(v_h)$: 1.731, 1.703, 0.052, 0.014.}
  \label{fig:breather_implicit_compar}
  
\end{figure}
\begin{figure}[tbp]
  \centering
  \begin{subfigure}[t]{0.48\textwidth}
    \centering
    \includegraphics[width=\textwidth]{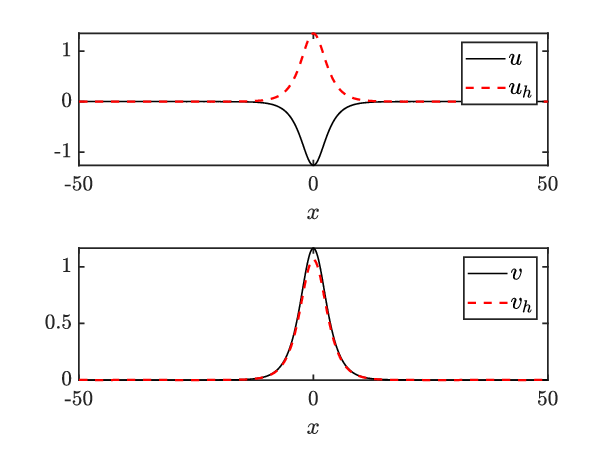}
    \caption{DDG(2)-ERK(4,4)}
    \label{fig:breather_DDG_ERK}
  \end{subfigure}
  \hfill
  \begin{subfigure}[t]{0.48\textwidth}
    \centering
    \includegraphics[width=\textwidth]{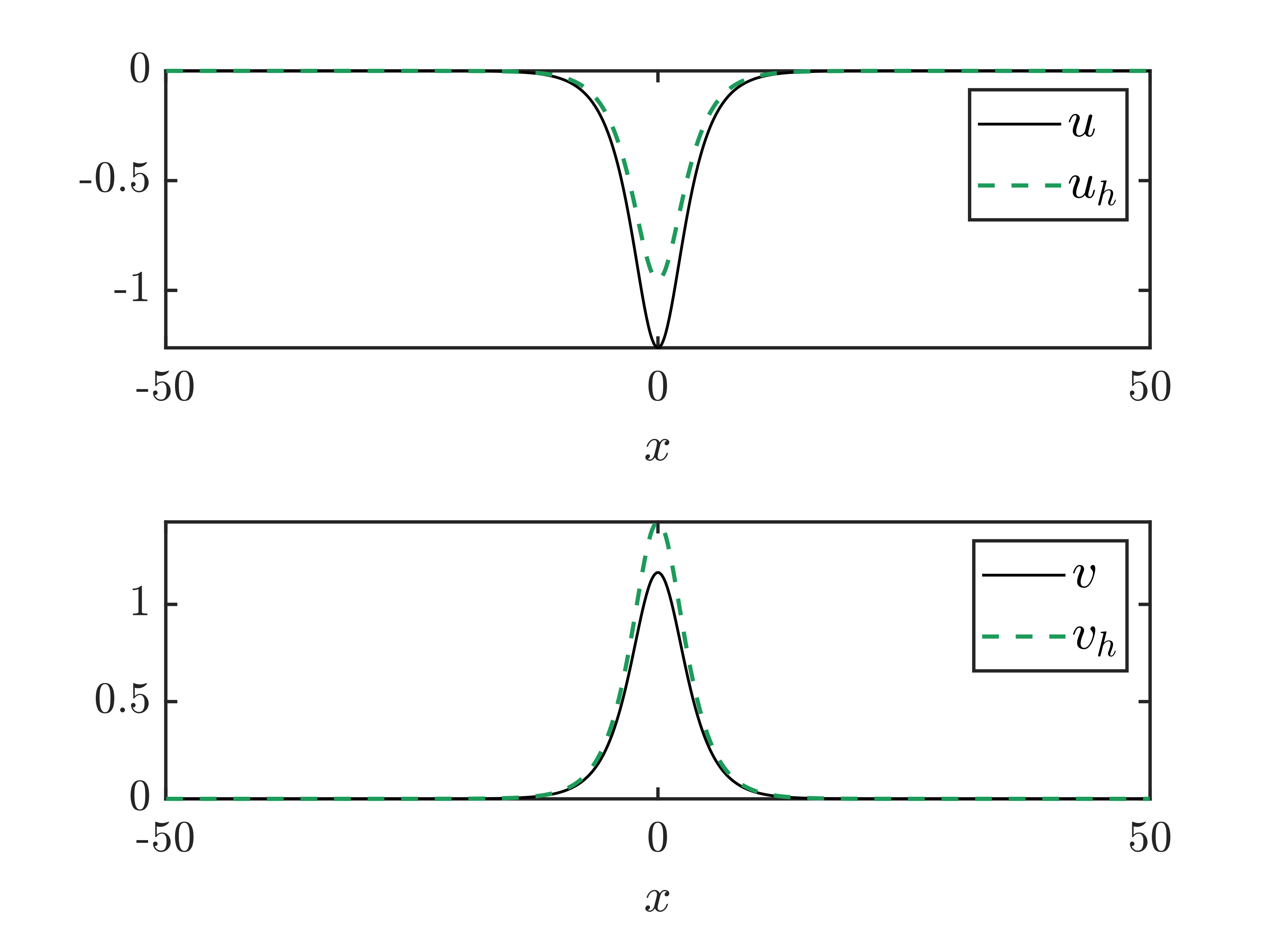}
    \caption{SDDG(2)-ERK(4,4)}
    \label{fig:breather_SDDG_ERK}
  \end{subfigure}
\\
  \begin{subfigure}[t]{0.48\textwidth}
    \centering
    \includegraphics[width=\textwidth]{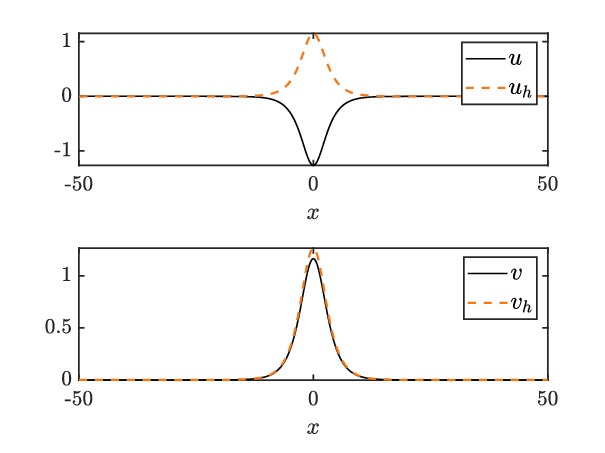}
    \caption{DDG(2)-ESPRK(6,4)}
    \label{fig:breather_DDG_ESPRK}
  \end{subfigure}
  \hfill
  \begin{subfigure}[t]{0.48\textwidth}
    \centering
    \includegraphics[width=\textwidth]{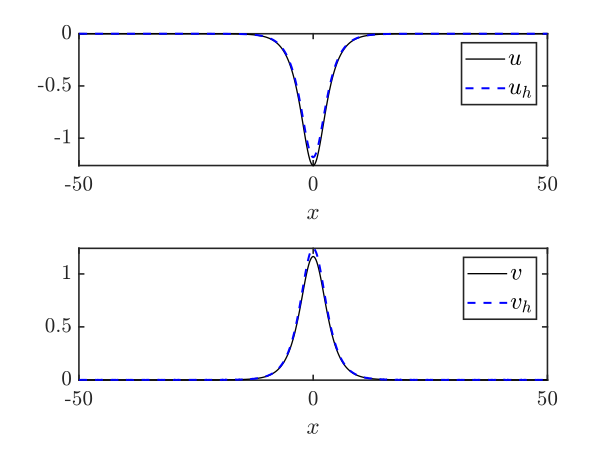}
    \caption{SDDG(2)-ESPRK(6,4)}
    \label{fig:breather_SDDG_ESPRK}
  \end{subfigure}
\caption{Example~\ref{example:breather}, $\omega=0.9$, $L=100$, $T_{\mathrm{f}}=6978.8$, $N=128$, $\mathrm{CFL}=0.1$. Exact solutions and numerical approximations by explicit time integrators. $\mathrm{err}_{L^\infty}(u_h)$: 2.61, 0.41, 2.41, 0.11; $\mathrm{err}_{L^\infty}(v_h)$: 2.18, 0.30, 1.97, 0.08.}
  \label{fig:breather_explicit_compar}
\end{figure}

For the two-dimensional Josephson transmission line, the SDDG method also
captures the expected qualitative soliton dynamics.  Figure~\ref{fig:josephson_snapshots}
shows snapshots of \(u_h\) at \(t=12,18,24,30\).  In the non-cloning case
\(c=0.7<c^*\), the incoming soliton is reflected by the T-junction and remains
confined to the horizontal branch.  In the cloning case \(c=0.8>c^*\), the
soliton passes through the junction and splits into two outgoing solitons along
the horizontal and vertical branches.  These results demonstrate that the
symplectic Hamiltonian DDG method reproduces the reflection and cloning
phenomena of the sine-Gordon soliton on the T-shaped domain.
\begin{figure}[tbp]
\centering
  \begin{subfigure}[t]{0.48\textwidth}
    \centering
    \includegraphics[width=\textwidth]{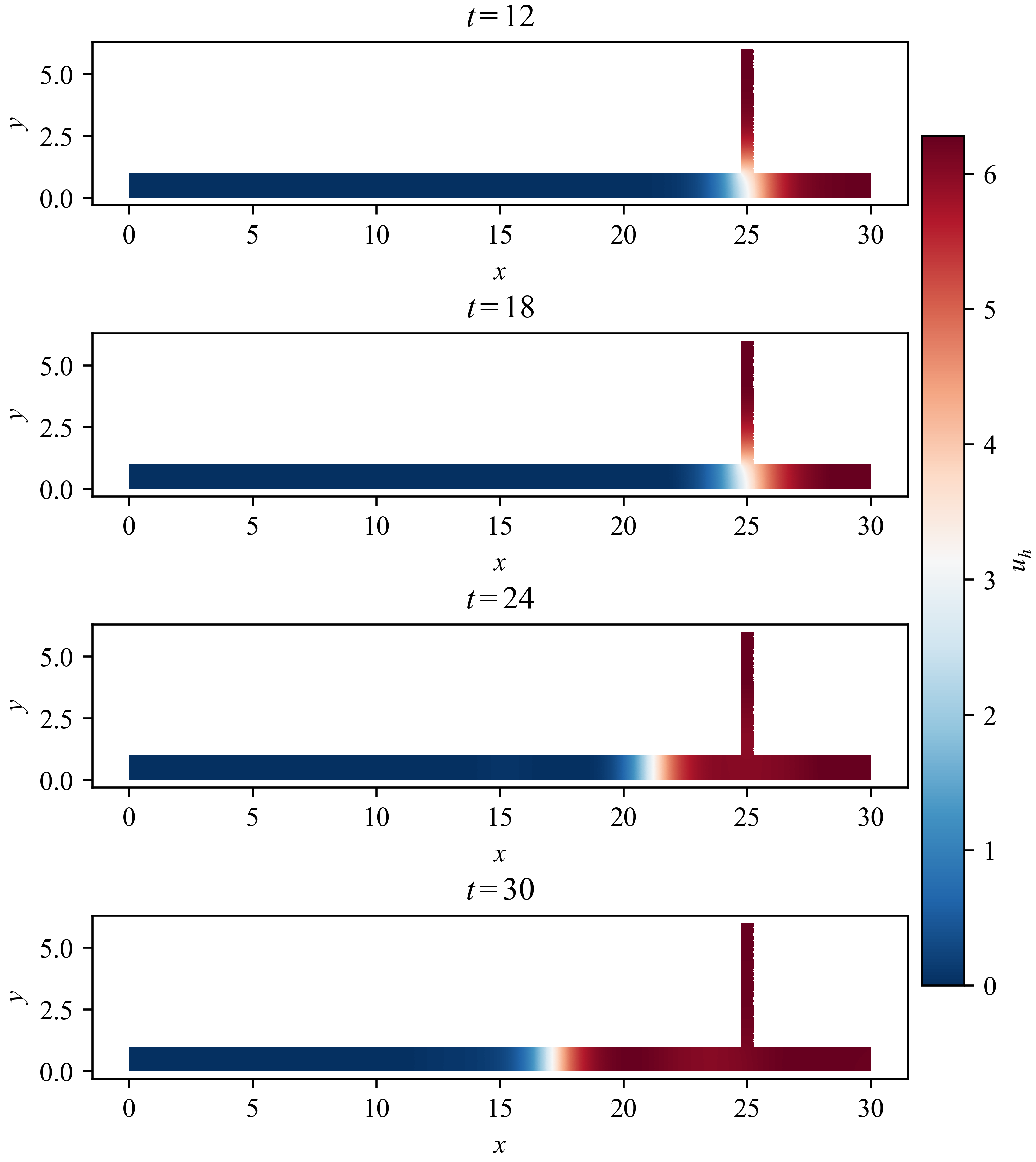}
    \caption{\(c=0.7\): reflection}
    \label{fig:josephson_c07_snapshots}
  \end{subfigure}
  \hfill
  \begin{subfigure}[t]{0.48\textwidth}
    \centering
    \includegraphics[width=\textwidth]{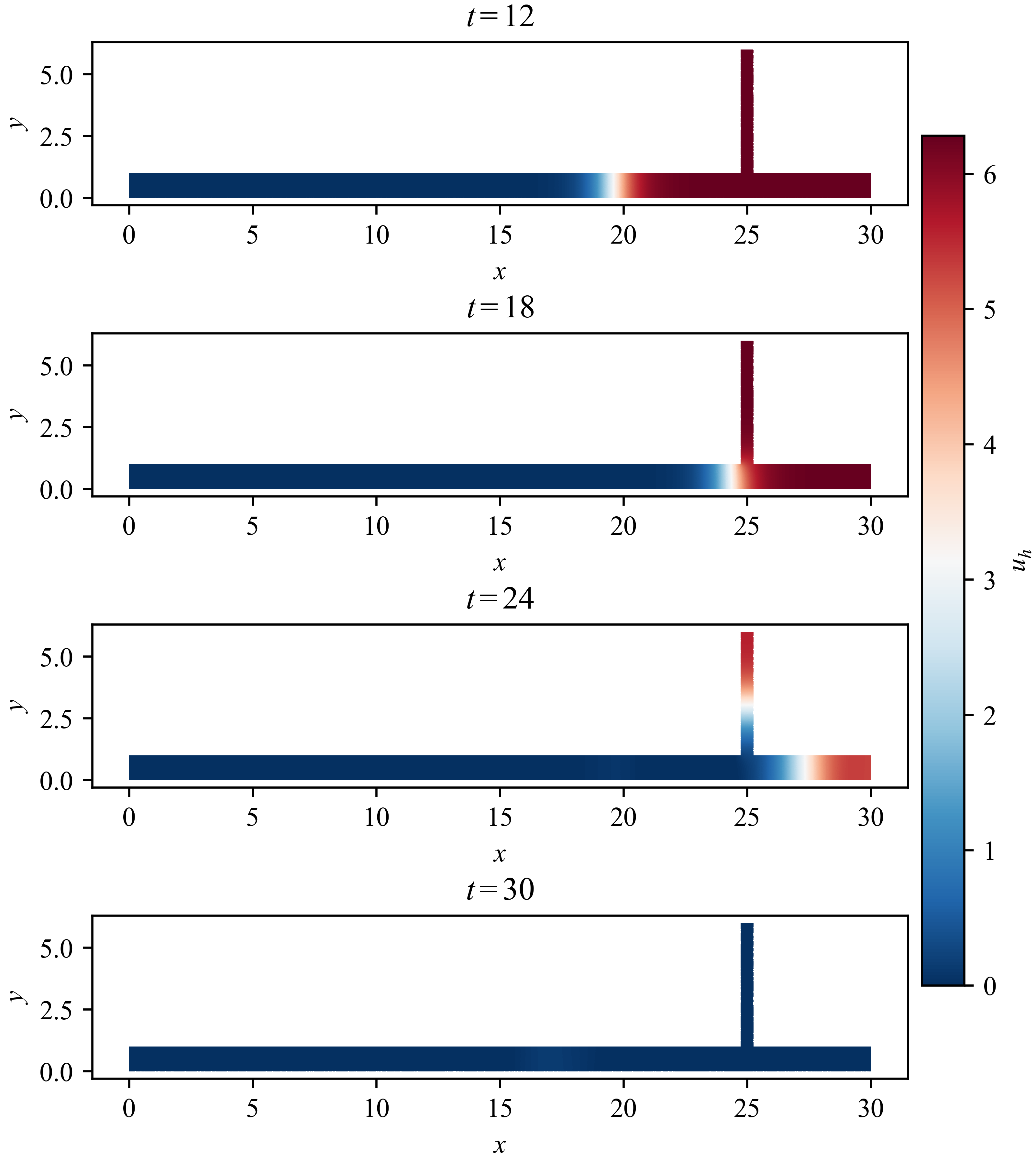}
    \caption{\(c=0.8\): cloning}
    \label{fig:josephson_c08_snapshots}
  \end{subfigure}
\caption{Example~\ref{ex:josephson}. Snapshots of the numerical displacement
\(u_h\) on the T-shaped Josephson transmission line at
\(t=12,18,24,30\).  The SDDG(2)-SDIRK(3,4) method with \(h=0.1,\mathrm{CFL}=0.35\) recovers the physically expected reflection for \(c=0.7\) and soliton cloning for \(c=0.8\).}
\label{fig:josephson_snapshots}
\end{figure}
\subsubsection{Discrete energy evolution}
We first examine the discrete energy evolution of the breather solution. The symplectic Hamiltonian DDG method (SDDG-SDIRK and SDDG-ESPRK) achieves the best energy conservation under both implicit and explicit integration. Under implicit integration, symplectic integrators maintain the energy within a small bounded oscillation, while non-symplectic DIRK introduces severe dissipation that causes the energy error to grow (Figure~\ref{fig:breather_implicit}). Under explicit integration, non-symplectic ERK exhibits cumulative energy drift that grows steadily in time, while symplectic ESPRK keeps the energy error stable (Figure~\ref{fig:breather_explicit}). The kink solution, whose energy behaviour is similar to the linear case, is not shown separately.
\begin{figure}[tbp]
\centering
  \begin{subfigure}[t]{0.48\textwidth}
    \centering
    \includegraphics[width=\textwidth]{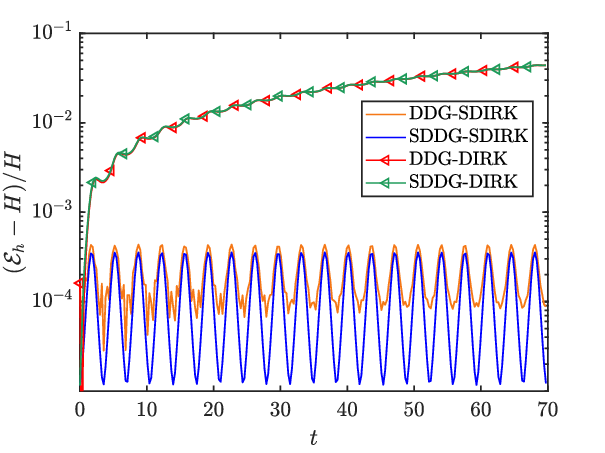}
    \caption{Implicit time integrators (CFL = 0.35)}
    \label{fig:breather_implicit}
  \end{subfigure}
  \hfill
  \begin{subfigure}[t]{0.48\textwidth}
    \centering
    \includegraphics[width=\textwidth]{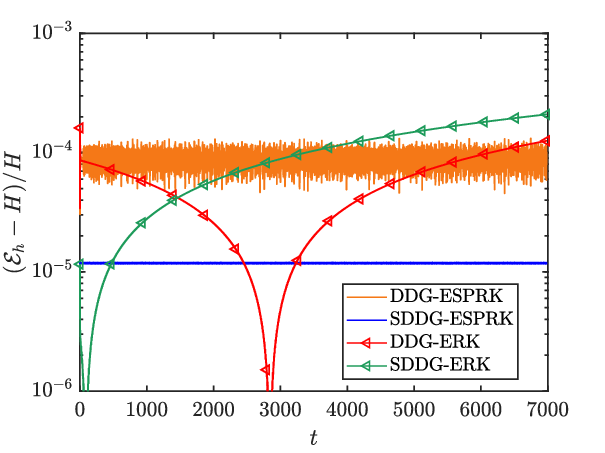}
    \caption{Explicit time integrators (CFL = 0.1) }
    \label{fig:breather_explicit}
  \end{subfigure}
\caption{Example~\ref{example:breather}, $\omega=0.9$, $L=100$, $k=2$, $N=128$. 
Semi-log plot of relative discrete energy errors versus time for all combinations of fourth-order symplectic/non-symplectic time 
integrators and SDDG/DDG spatial discretizations ($k=2$).}
\label{fig:breather_energy}
\end{figure}

The same conclusion is observed for the two-dimensional Josephson test with
\(c=0.8\), shown in Figure~\ref{fig:josephson_energy}.  Among all combinations
of spatial discretization and time integrator, the symplectic Hamiltonian DDG
schemes produce the smallest relative energy error throughout the computation.
In particular, SDDG combined with the symplectic time integrator remains below
the corresponding DDG and non-symplectic variants, confirming that the
Hamiltonian-compatible spatial discretization and symplectic time integration
are both important for robust long-time energy behaviour in the nonlinear
two-dimensional soliton dynamics.
\begin{figure}[tbp]
\centering
  \begin{subfigure}[t]{0.48\textwidth}
  \raggedright
    \includegraphics[width=0.933\textwidth]{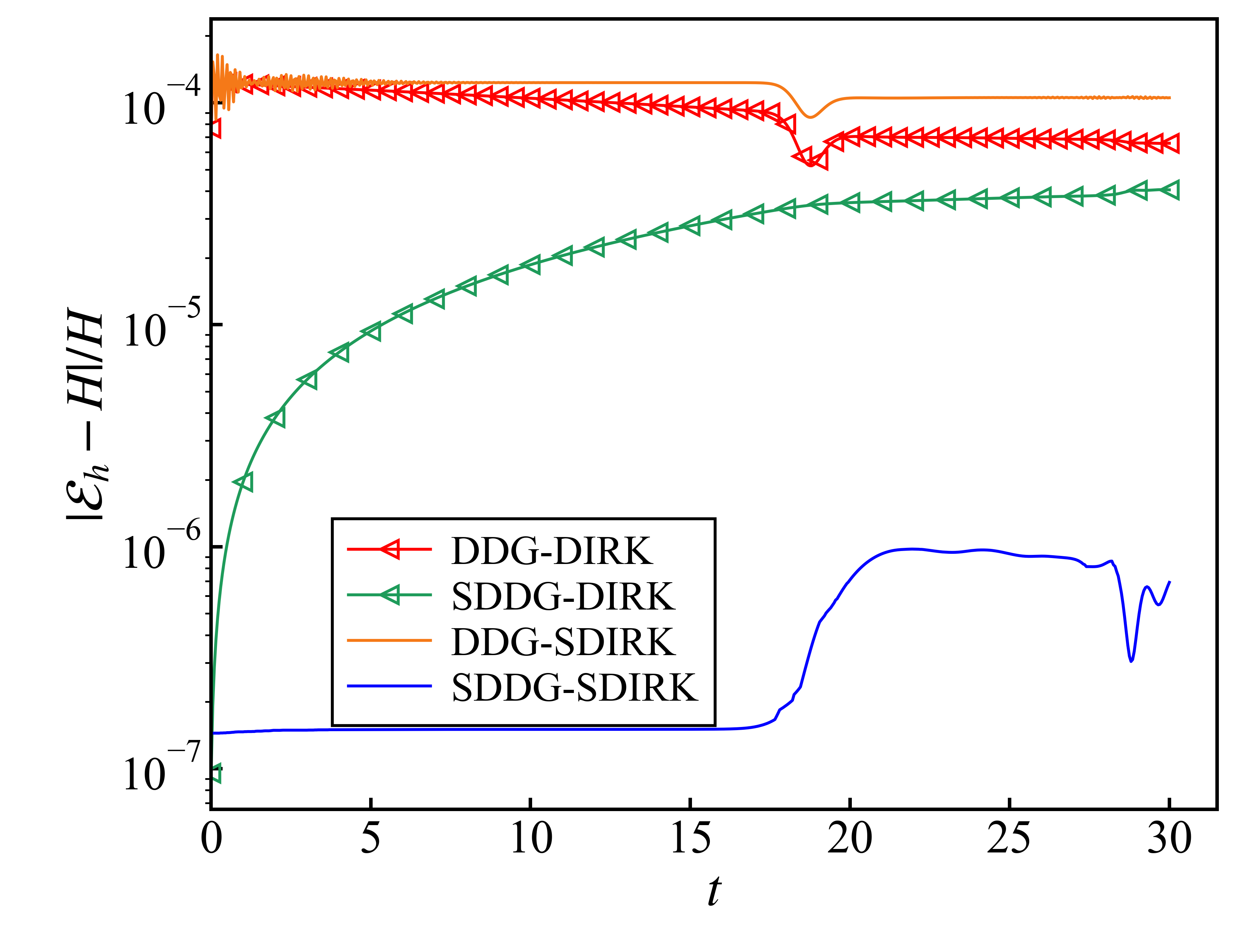}
    \caption{Implicit time integrators($\mathrm{CFL=0.35}$)}
    \label{fig:josephson_energy_implicit}
  \end{subfigure}
  \hfill
  \begin{subfigure}[t]{0.48\textwidth}
  \raggedright
    \includegraphics[width=0.933\textwidth]{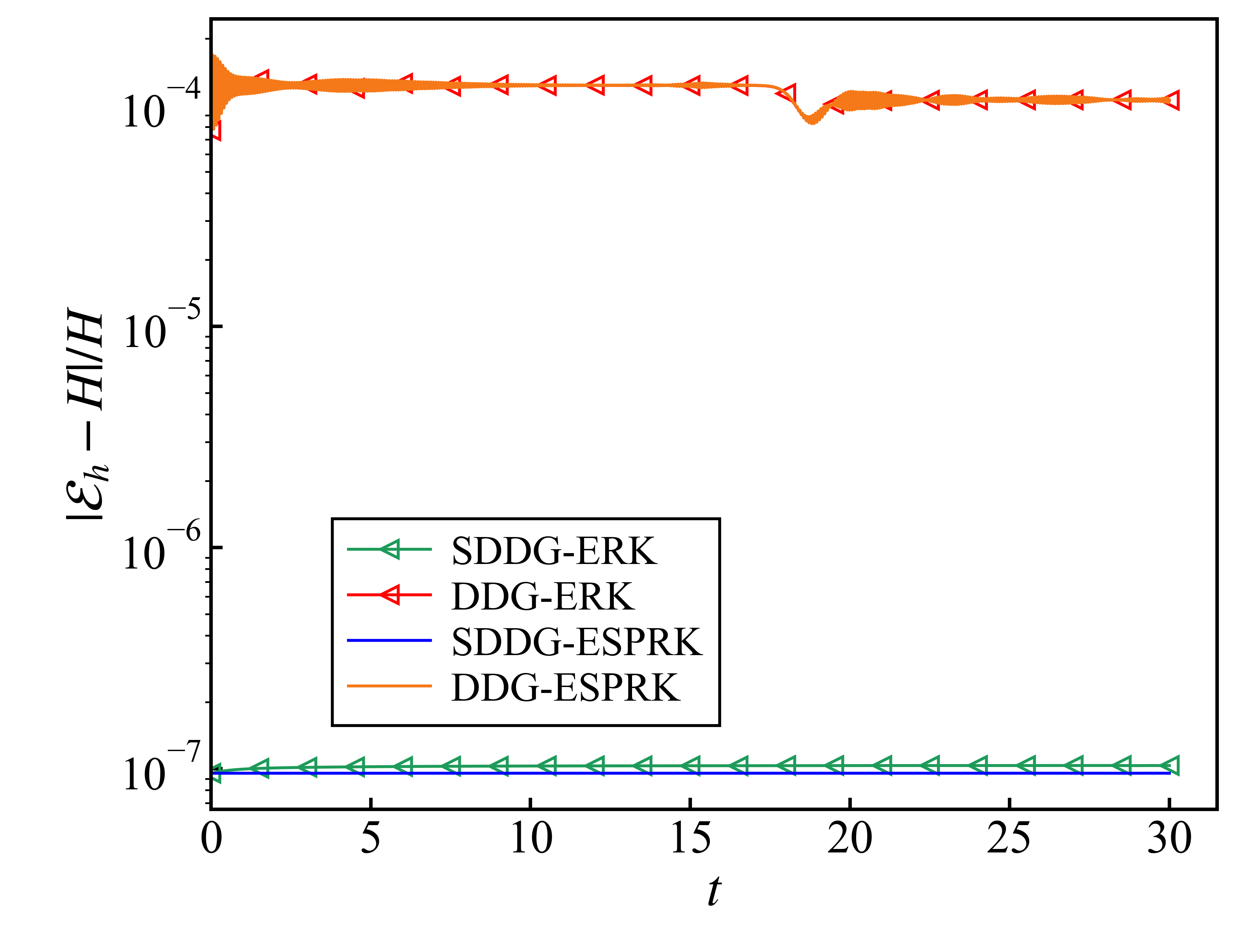}
    \caption{Explicit time integrators($\mathrm{CFL=0.1}$)}
    \label{fig:josephson_energy_explicit}
  \end{subfigure}
\caption{Example~\ref{ex:josephson}, \(c=0.8,h=0.1\). Semi-log plot of relative discrete energy errors versus time for all combinations of fourth-order symplectic/non-symplectic time integrators and SDDG/DDG spatial discretizations ($k=2$).}
\label{fig:josephson_energy}
\end{figure}
\section{Conclusion}
\label{sec:conclusion}
We investigated auxiliary-variable-free DG methods for semi-linear wave equations and established that the symmetry of flux bilinear form is equivalent to the existence of a discrete Hamiltonian. This result identifies SDDG method as a natural candidate for spatial discretization, and combining it with symplectic time integrators yields the symplectic Hamiltonian DDG method studied in this work.
Error analysis for the SDDG method yields optimal convergence rate for displacement and suboptimal convergence rate for velocity, together with verifiable parameter constraints reusable when extending the method to new problem classes.

Numerical experiments on linear and semilinear wave equations support the following conclusions. SDDG method consistently reduces dispersion errors compared to DDG method, independently of the time integrator, confirming that dispersion control should be addressed at the spatial level. In our tests, even without a discrete Hamiltonian, symplectic integrators suppress long-term energy dissipation and allow larger CFL numbers, though the energy may oscillate around a biased level. Implicit symplectic schemes show clear advantages over dissipative implicit methods. For explicit integration, the advantage of symplectic time integrator is modest for linear and weakly nonlinear problems but becomes essential for strongly nonlinear phenomena such as the breather solution and the two-dimensional Josephson soliton test. In the Josephson transmission-line example, the method reproduces the expected reflection and cloning regimes of sine-Gordon solitons on the T-shaped domain. The symplectic Hamiltonian DDG method achieves the best energy conservation across all test cases.

Several directions appear promising for future work. The symmetry-based Hamiltonian structure analysis of spatial semi-discretization developed here applies directly to other Hamiltonian PDEs wherever an auxiliary-variable-free DG discretization is used. Extension to multi-dimensional problems and unstructured meshes is a natural next step, where the flux symmetry criterion remains applicable. Finally, similar symmetry-preserving principles may be incorporated into neural-network-based PDE solvers through symmetry-preserving architectures.

\section*{Acknowledgements}
This work is supported by the Foundation of State Key Laboratory of Aerodynamics and  the General Program of the National Natural Science Foundation of China (Grant No. 12571580).


\appendix
\section{Coefficients of the symplectic time integrators}
\label{app:symplectic-coefficients}
The notation $(s,p)$ denotes the number of stages and the temporal order.
Tables~\ref{tab:esprk64}--\ref{tab:sdirk-coefficients} give the coefficients
from \cite[Appendix~A]{sanchez2017symplectic}.
\setcounter{table}{0}
\renewcommand{\thetable}{\thesection.\arabic{table}}

\begin{table}[htbp]
\centering
\caption{Coefficients of ESPRK(6,4).}
\label{tab:esprk64}
\begin{tabular}{c|rr}
\hline
$i$ & $b_i$ & $\tilde b_i$ \\
\hline
1 & $7/48$ & $1/3$ \\
2 & $3/8$ & $-1/3$ \\
3 & $-1/48$ & $1$ \\
4 & $-1/48$ & $-1/3$ \\
5 & $3/8$ & $1/3$ \\
6 & $7/48$ & $0$ \\
\hline
\end{tabular}
\end{table}

\begin{table}[htbp]
\centering
\caption{Coefficients of ESPRK(6,5).}
\label{tab:esprk65}
\small
\begin{tabular}{c|rr}
\hline
$i$ & $b_i$ & $\tilde b_i$ \\
\hline
1 & $0.1193900292875672758$ & $0.3398396258391100000$ \\
2 & $0.6989273703824752308$ & $-0.0886013369030273290$ \\
3 & $-0.1713123582716007754$ & $0.5858564768259621188$ \\
4 & $0.4012695022513534480$ & $-0.6030393565364911888$ \\
5 & $0.0107050818482359840$ & $0.3235807965546976394$ \\
6 & $-0.0589796254980311632$ & $0.4423637942197494587$ \\
\hline
\end{tabular}
\end{table}

\begin{table}[htbp]
\centering
\caption{Coefficients of ESPRK(11,6).}
\label{tab:esprk116}
\small
\begin{tabular}{c|rr}
\hline
$i$ & $b_i$ & $\tilde b_i$ \\
\hline
1 & $0.0502627644003922$ & $0.148816447901042$ \\
2 & $0.413514300428344$ & $-0.132385865767784$ \\
3 & $0.0450798897943977$ & $0.067307604692185$ \\
4 & $-0.188054853819569$ & $0.432666402578175$ \\
5 & $0.541960678450780$ & $-0.016404589403618$ \\
6 & $-0.725525558508690$ & $-0.016404589403618$ \\
7 & $0.541960678450780$ & $0.432666402578175$ \\
8 & $-0.188054853819569$ & $0.067307604692185$ \\
9 & $0.0450798897943977$ & $-0.132385865767784$ \\
10 & $0.413514300428344$ & $0.148816447901042$ \\
11 & $0.0502627644003922$ & $0$ \\
\hline
\end{tabular}
\end{table}

\begin{table}[htbp]
\centering
\caption{Coefficients $b_i$ of the SDIRK schemes, with $a=1.351207191959658$.}
\label{tab:sdirk-coefficients}
\small
\begin{tabular}{c|ccc}
\hline
$i$ & SDIRK(3,4) & SDIRK(6,5) & SDIRK(7,6) \\
\hline
1 & $a$ & $0.5080048194000274$ & $0.78451361047755652$ \\
2 & $1-2a$ & $1.3601071622948270$ & $0.23557321335935860$ \\
3 & $a$ & $2.0192933591817224$ & $-1.1776799841788705$ \\
4 & --- & $0.5685658926458251$ & $1.3151863206839107$ \\
5 & --- & $-1.4598520495864393$ & $-1.1776799841788705$ \\
6 & --- & $-1.9961191839359627$ & $0.23557321335935860$ \\
7 & --- & --- & $0.78451361047755652$ \\
\hline
\end{tabular}
\end{table}

\section{Coefficients of the comparison time integrators}
\label{app:comparison-coefficients}
\setcounter{table}{0}
Table~\ref{tab:comparison-butcher} lists the Butcher tableaux of
ERK(4,4) and DIRK(3,4) \cite{butcher2016numerical,RogerAlexander}.
\begin{table}[htbp]
\centering
\caption{Butcher tableaux of the comparison schemes.}
\label{tab:comparison-butcher}
\begin{minipage}[t]{0.42\textwidth}
\centering
ERK(4,4)
\[
\renewcommand{\arraystretch}{1.5}
\begin{array}{c|cccc}
0 & 0 & 0 & 0 & 0 \\
\frac12 & \frac12 & 0 & 0 & 0 \\
\frac12 & 0 & \frac12 & 0 & 0 \\
1 & 0 & 0 & 1 & 0 \\
\hline
& \frac16 & \frac13 & \frac13 & \frac16
\end{array}
\]
\end{minipage}\hfill
\begin{minipage}[t]{0.55\textwidth}
\centering
DIRK(3,4), $\alpha=\frac{2}{\sqrt3}\cos(\frac{\pi}{18})$
\[
\renewcommand{\arraystretch}{1.5}
\begin{array}{c|ccc}
\frac{1+\alpha}{2} & \frac{1+\alpha}{2} & 0 & 0 \\
\frac12 & -\frac\alpha2 & \frac{1+\alpha}{2} & 0 \\
\frac{1-\alpha}{2} & 1+\alpha & -(1+2\alpha) & \frac{1+\alpha}{2} \\
\hline
& \frac{1}{6\alpha^2} & 1-\frac{1}{3\alpha^2} & \frac{1}{6\alpha^2}
\end{array}
\]
\end{minipage}
\end{table}


\end{document}